\documentclass[a4paper]{amsart}
\usepackage{amsmath}
\usepackage{amsfonts}
\usepackage{amssymb}
\usepackage{amsthm}
\usepackage{amsbsy}
\usepackage{enumerate}
\usepackage[shortlabels]{enumitem}
\usepackage[hidelinks]{hyperref}
\usepackage{graphicx}
\usepackage{comment}

\newcommand{\R}{\mathbb{R}}
\newcommand{\C}{\mathbb{C}}
\newcommand{\deriv}[2][]{\frac{d #1}{d #2}}
\newcommand{\pderiv}[2][]{\frac{\partial #1}{\partial #2}}
\newcommand{\codim}{\mathrm{codim}}
\newcommand{\id}{\mathrm{id}}
\newcommand{\ii}{\sqrt{-1}}
\newcommand{\cob}[1]{C^{#1}_{\mathrm{COB}}}
\newcommand{\vmv}{\overline{C^\infty}}
\newcommand{\orientedprojective}{\mathbb{P}^+}
\DeclareMathOperator{\real}{Re}
\DeclareMathOperator{\imaginary}{Im}
\DeclareMathOperator{\vol}{vol}
\DeclareMathOperator{\crit}{Crit}

\newcommand{\diff}{\mathrm{Diff}}
\newcommand{\imm}{\mathrm{Imm}}

\newcommand{\iml}{\mathrm{LI}}

\newcommand{\ilf}{\mathrm{FLI}}
\DeclareMathOperator{\im}{Im}
\newcommand{\Cs}{C^{\infty}_{cone}}

\theoremstyle{plain}
\newtheorem{thm}{Theorem}[section]
\newtheorem{lemma}[thm]{Lemma}
\newtheorem{prop}[thm]{Proposition}
\newtheorem{cor}[thm]{Corollary}

\theoremstyle{definition}
\newtheorem{dfn}[thm]{Definition}
\newtheorem{ntn}[thm]{Notation}

\theoremstyle{remark}
\newtheorem{rem}[thm]{Remark}

\begin{document}
    \title[Geodesics from special Lagrangians with boundary]{Geodesics of positive Lagrangians from special Lagrangians with boundary}
	\author[J. Solomon]{Jake P. Solomon}
	\address{Institute of Mathematics\\ Hebrew University, Givat Ram\\Jerusalem, 91904, Israel}
	\email{jake@math.huji.ac.il}
	\author[A. Yuval]{Amitai M. Yuval}
	\address{Institute of Mathematics\\ Hebrew University, Givat Ram\\Jerusalem, 91904, Israel}
	\email{amitai.yuval@mail.huji.ac.il}
	
	\date{March 2026}

	\begin{abstract}
		Geodesics in the space of positive Lagrangian submanifolds are solutions of a fully non-linear degenerate elliptic PDE.
We show that a geodesic segment in the space of positive Lagrangians corresponds to a one parameter family of special Lagrangian cylinders, called the cylindrical transform. The boundaries of the cylinders are contained in the positive Lagrangians at the ends of the geodesic. The special Lagrangian equation with positive Lagrangian boundary conditions is elliptic and the solution space is a smooth manifold, which is one dimensional in the case of cylinders. A geodesic can be recovered from its cylindrical transform by solving the Dirichlet problem for the Laplace operator on each cylinder.

Using the cylindrical transform, we show the space of pairs of positive Lagrangian spheres connected by a geodesic is open. Thus, we obtain the first examples of strong solutions to the geodesic equation in arbitrary dimension not invariant under isometries. In fact, the solutions we obtain are smooth away from a finite set of points.
	\end{abstract}

\maketitle
		
\tableofcontents

	\section{Introduction}
\subsection{Overview}	
	Let $(X,\omega,J,\Omega)$ be a \emph{Calabi-Yau manifold}. Namely, $X$ is a K\"ahler manifold with symplectic form $\omega$ and complex structure $J$, and $\Omega$ is a non-vanishing holomorphic volume form on $X$. We denote by $g$ the K\"ahler metric and by $n$ the complex dimension.
	
	An oriented Lagrangian submanifold $\Lambda\subset X$, possibly immersed, is said to be \emph{positive} if $\real \Omega|_\Lambda$ is a positive volume form. A positive Lagrangian submanifold is \emph{special} if $\imaginary \Omega|_\Lambda = 0.$ An oriented Lagrangian submanifold is called \emph{imaginary special} if $\real \Omega|_\Lambda = 0$ and $\imaginary \Omega|_\Lambda$ is a positive volume form.
	
	Let $\mathcal{O}$ be a Hamiltonian isotopy class of closed smoothly embedded positive Lagrangians diffeomorphic to a given manifold $L.$ Then $\mathcal{O}$ is naturally a smooth Fr\'echet manifold, and for $\Lambda\in\mathcal{O}$, there is a natural isomorphism
	\begin{equation}
	\label{equation: tangent space to O}
	T_\Lambda\mathcal{O}\cong\vmv(\Lambda):=\left\{h\in C^\infty(\Lambda)\left|\int_\Lambda h\real\Omega=0\right.\right\}.
	\end{equation}
	Following~\cite{solomon}, we define a Riemannian metric on $\mathcal{O}$ by
	\begin{equation}
	\label{equation: the metric}
	(h,k):=\int_\Lambda hk\real\Omega,\qquad h,k\in\vmv(\Lambda).
	\end{equation}
	It is shown in~\cite{solomon-curv} that the metric $(\cdot,\cdot)$ has a Levi-Civita connection and the associated sectional curvature is non-positive. The Levi-Civita connection, which we describe in detail in Section~\ref{subsection: calabi-yau manifolds}, gives rise to the notion of geodesics. If $\mathcal{O}$ is geodesically connected then there can exist at most one special Lagrangian in $\mathcal{O}$~\cite{solomon}. If two Lagrangians $\Lambda_0,\Lambda_1 \in \mathcal{O}$ are connected by a geodesic, the cardinality of $\Lambda_0 \cap \Lambda_1$ is bounded below by the number of critical points of a function on $\Lambda_0$ \cite{rubinstein-solomon}.
	
	The geodesic equation is a fully non-linear partial differential equation. It is shown in~\cite{rubinstein-solomon} that the geodesic equation is degenerate elliptic and the associated boundary value problem in the Euclidean setting has unique weak solutions. In~\cite{solomon-yuval} there are examples of smooth geodesics in arbitrary dimension, which are preserved by an isometric action of the orthogonal group $O(n).$ The group action allows the geodesic equation to be reduced to the one dimensional case, where it becomes an ODE. Further results on geodesics can be found in~\cite{dellatorre,xu}.

An analog to the space $\mathcal{O}$ under mirror symmetry is the space $\mathcal{H}$ of almost calibrated $(1,1)$-forms on a K\"ahler manifold. The geodesic equation in $\mathcal{H}$ is a degenerate form of the deformed Hermitian Yang-Mills equation. The space $\mathcal{H}$ and its geodesics have been studied in~\cite{chu2020space,collins2018moment}. Another analog to the space $\mathcal{O}$ is the space $\mathcal{K}$ of K\"ahler metrics in a fixed cohomology class on a K\"ahler manifold. The geodesic equation in $\mathcal{K}$ is the homogenous complex Monge-Ampere equation. The space $\mathcal{K}$ and its geodesics have been studied, for example, in~\cite{chen-kaehlermetrics,chen-tian,Donaldson-symmetric,donaldson-holomorphicdiscs,Mabuchi-K-energy,Mabuchi,semmes}.
	
	In the present work, we establish a correspondence between geodesics of positive Lagrangians and one parameter families of imaginary special Lagrangian cylinders. See Theorems~\ref{theorem: family} and~\ref{theorem: geodesic-cylinder correspondence}. We call this correspondence the cylindrical transform.
By cylinders, we mean manifolds of the form $N \times [0,1],$ where $N$ is a manifold of dimension $n-1.$
The boundary components of the cylinders corresponding to a geodesic $(\Lambda_t)_{t \in [0,1]}$ are contained in $\Lambda_0$ and $\Lambda_1$ respectively. Positive Lagrangian submanifolds such as $\Lambda_0$ and $\Lambda_1$ are an elliptic boundary condition for the imaginary special Lagrangian equation. In Theorem~\ref{theorem: space of ISL cylinders is 1-dim} we show that the space of imaginary special Lagrangian cylinders with positive Lagrangian boundary conditions is a smooth $1$-dimensional manifold.

Using the cylindrical transform and the ellipticity of the imaginary special Lagrangian equation, we establish in Theorem~\ref{theorem: perturbation of geodesic} a perturbation result for solutions of the geodesic equation. Namely, the space of pairs of positive Lagrangian spheres intersecting transversally at two points that are connected by a geodesic is open. In particular, we show the existence of geodesics connecting positive Lagrangians of arbitrary dimension without any symmetry. The geodesics of Theorem~\ref{theorem: perturbation of geodesic} are smooth away from a finite number of points. In a sequel~\cite{spider}, we strengthen Theorem~\ref{theorem: perturbation of geodesic} so that it produces geodesics of positive Lagrangians that are $C^{1,1}$ submanifolds even at the non-smooth locus.

\subsection{Statement of results}
To set up the cylindrical transform in its natural generality, we consider geodesics of possibly immersed positive Lagrangians that are smooth away from a finite number of cone point singularities. For the rest of the paper, unless otherwise specified, the term geodesic allows such cone point singularities.
We define the \emph{critical locus} of a geodesic of positive Lagrangians $(\Lambda_t)_{t \in [0,1]}$ by
	\[
	\crit((\Lambda_t)_t):=\bigcap_{t\in[0,1]}\Lambda_t.
	\]
The non-smooth cone points of the geodesics in this paper are contained in their critical loci. The possibly limited regularity at the critical locus is consistent with the result of~\cite{rubinstein-solomon} that the symbol of the linearized geodesic equation has a 1-dimensional kernel except at the critical locus, where the kernel is $(n-1)$-dimensional.
A full account of our notion of geodesics is given in Definition~\ref{definition: geodesic}.	

	Define a positive function $\rho : X \to \R$ by
	\begin{equation}
	\label{equation: rho}
	\rho^2\omega^n/n! = (-1)^{\frac{n(n-1)}2}\left(\frac{\ii}{2}\right)^n \Omega \wedge \overline \Omega.
	\end{equation}
	Let $f: L \to X$ be an immersion. Define
\[
\Delta_\rho: C^\infty(L)\to C^\infty(L)
\]
by $u\mapsto*d((\rho \circ f)*du)$, where $*$ is the Hodge star operator associated to the Riemannian metric $f^*g.$ Then $\Delta_\rho$ is elliptic (see Lemma~\ref{lemma: Delta rho is a cool operator}). A geodesic $(\Lambda_t)_{t\in[0,1]}$ has an associated \emph{Hamiltonian}, which is the family of functions $h_t \in \vmv(\Lambda_t)$ such that $h_t =  \deriv{t} \Lambda_t.$
The functions $h_t, t \in [0,1],$ are related by parallel transport of the Levi-Cevita connection on $\mathcal{O}$ and, in particular, have the same image and diffeomorphic level sets.

	\begin{thm}\label{theorem: family}
		Let $(\Lambda_t)_{t\in [0,1]}$ be a geodesic of positive Lagrangians and let $(h_t)_{t\in[0,1]}$ denote the associated Hamiltonian. For $c \in \R,$ let
		\[
		L_c := \{(p,t)| t \in [0,1],\; p \in h_t^{-1}(c)\setminus \crit((\Lambda_t)_t) \subset \Lambda_t\}.
		\]
		Then $L_c$ is a smooth immersed submanifold of $X \times [0,1]$ diffeomorphic to the cylindrical manifold $\left(h_0^{-1}(c)\setminus \crit((\Lambda_t)_t)\right)\times [0,1]$, and the map
		\[
		\Phi_c : L_c \to X
		\]
		given by $\Phi_c(p,t) = p$ is an imaginary special Lagrangian immersion mapping the boundary components of $L_c$ to $\Lambda_0$ and $\Lambda_1$. See Figure~\ref{fig:theorem family}. Moreover, the map
		\[
		\sigma_c : L_c \to [0,1]
		\]
		given by $\sigma_c(p,t) = t$ satisfies $\Delta_\rho \sigma_c =  0.$
	\end{thm}
\begin{figure}[ht]
\centering
\includegraphics[width=12cm]{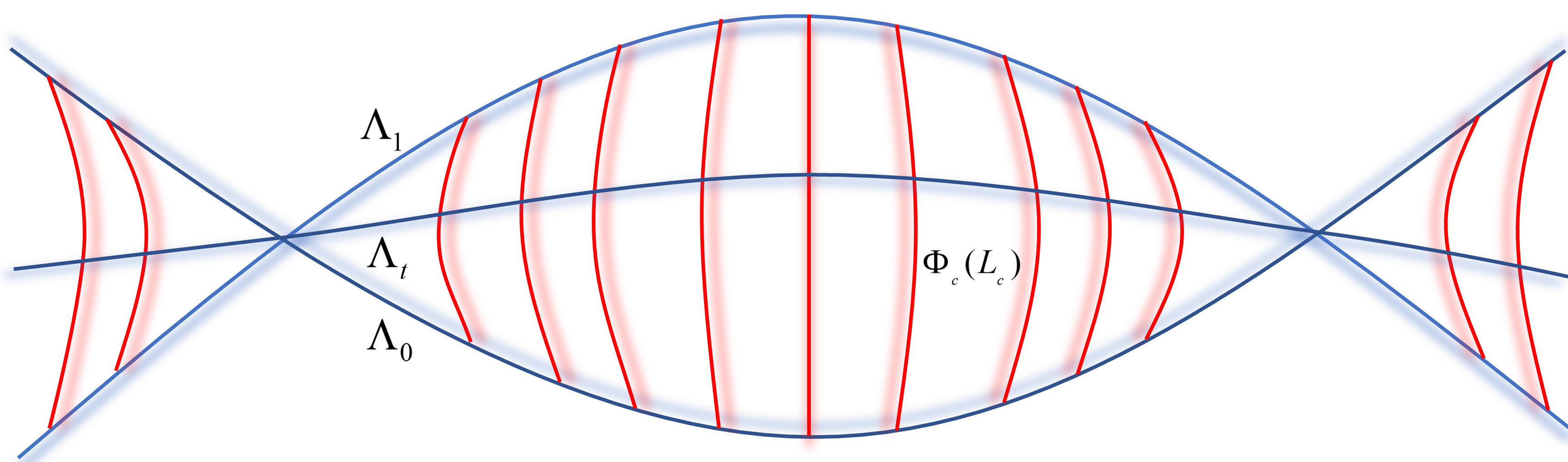}
\caption{The imaginary special Lagrangian cylinders corresponding to a geodesic $(\Lambda_t)_t$ according to Theorem~\ref{theorem: family}.}
\label{fig:theorem family}
\end{figure}
	
Let $\Lambda_0,\Lambda_1\subset X$ be smoothly embedded positive Lagrangians and let $N$ be a manifold of dimension $n-1$. We denote by $\mathcal{SLC}(N;\Lambda_0,\Lambda_1)$ the space of imaginary special Lagrangian submanifolds of $X,$ perhaps immersed, diffeomorphic to $N\times[0,1],$ such that the boundary corresponding to $N \times \{i\}$ is embedded in $\Lambda_i$ for $i = 0,1.$ We denote by $\mathcal{SLC}(\Lambda_0,\Lambda_1)$ the union of the spaces $\mathcal{SLC}(N;\Lambda_0,\Lambda_1)$ as $N$ varies.
	
	\begin{thm}
		\label{theorem: space of ISL cylinders is 1-dim}
		Given two smoothly embedded positive Lagrangians $\Lambda_0,\Lambda_1,$ and a connected closed $(n-1)$-manifold $N,$ the space of imaginary special Lagrangian cylinders $\mathcal{SLC}(N;\Lambda_0,\Lambda_1)$ is a smooth $1$-dimensional manifold.
	\end{thm}
	\begin{rem}
		More generally, positive Lagrangians are natural elliptic boundary conditions for imaginary special Lagrangians with boundary of arbitrary topology, and the associated deformation theory is unobstructed. See Remark~\ref{remark: not only cylinders}. The deformation theory of closed special Lagrangians was shown to be unobstructed in~\cite{mclean,salur}.
	\end{rem}

	The next result is a refinement and partial converse to Theorem~\ref{theorem: family} in the case where $\Lambda_i,\;i=0,1,$ are smoothly embedded spheres intersecting transversally at exactly two points.
To formulate the result we need the following definition.	
	\begin{dfn}
		\label{definition: cylindrical transform}
		The \emph{cylindrical transform} of a geodesic of positive Lagrangians $(\Lambda_t)_{t \in [0,1]}$ is the subset of $\mathcal{SLC}(\Lambda_0,\Lambda_1)$ parameterized by the family of imaginary special Lagrangian immersions $\Phi_c : L_c \to X$ from Theorem~\ref{theorem: family} for $c \in \R$ such that $L_c \neq \emptyset$.
	\end{dfn}
We refer the reader to Definition~\ref{definition: regularity} for the notion of regularity of a connected component in $\mathcal{SLC}\left(S^{n-1};\Lambda_0,\Lambda_1\right).$
	\begin{thm}
		\label{theorem: geodesic-cylinder correspondence}
		Let $\Lambda_0,\Lambda_1\subset X$ be smoothly embedded positive Lagrangian spheres intersecting transversally at exactly two points. The cylindrical transform of a geodesic between $\Lambda_0$ and $\Lambda_1$ is a regular connected component in $\mathcal{SLC}\left(S^{n-1};\Lambda_0,\Lambda_1\right).$ Conversely, given a regular connected component $\mathcal{Z}\subset\mathcal{SLC}\left(S^{n-1};\Lambda_0,\Lambda_1\right),$ there exists a unique up to reparameterization geodesic between $\Lambda_0$ and $\Lambda_1$ with cylindrical transform $\mathcal{Z}.$
	\end{thm}

	Finally, we apply Theorem~\ref{theorem: geodesic-cylinder correspondence} to prove that geodesics of positive Lagrangian spheres with endpoints intersecting transversally at two points are stable under $C^{2,\alpha}$-small Hamiltonian perturbations. 		Let $\mathcal{O}$ be a Hamiltonian isotopy class of smoothly embedded positive Lagrangian spheres, and let $\mathfrak{G}_\mathcal{O}$ denote the space of geodesics $(\Lambda_t)_{t \in [0,1]}$ with $\Lambda_0,\Lambda_1 \in \mathcal{O}$ intersecting transversally at two points. We refer the reader to Definition~\ref{dfn:topology on geodesics} for the strong and weak $C^{1,\alpha}$ topologies on $\mathfrak{G}_\mathcal{O}.$ Roughly speaking, the strong topology controls closeness of all cylinders in the cylindrical transform of a geodesic while the weak topology controls closeness of a single cylinder.

	\begin{thm}
		\label{theorem: perturbation of geodesic}
Let $\Lambda_0,\Lambda_1\in\mathcal{O}$ intersect transversally at exactly two points. Suppose there exists a geodesic $(\Lambda_t)_{t\in[0,1]}$ between $\Lambda_0$ and $\Lambda_1.$ Let $\alpha\in(0,1).$ Then, there exists a $C^{2,\alpha}$-open neighborhood $\mathcal{Y}$ of $\Lambda_1$ in $\mathcal{O}$ and a weak $C^{1,\alpha}$-open neighborhood $\mathcal X$ of $(\Lambda_t)_{t \in [0,1]}$ in $\mathfrak{G}_\mathcal{O}$ such that for every $\Lambda\in\mathcal{Y}$ there exists a unique geodesic between $\Lambda_0$ and $\Lambda$ in $\mathcal{X}.$ This geodesic depends continuously on $\Lambda$ with respect to the $C^{2,\alpha}$ topology on $\mathcal{Y}$ and the strong $C^{1,\alpha}$ topology on $\mathcal{X}.$
	\end{thm}

	In the sequel~\cite{spider}, we strengthen Theorem~\ref{theorem: perturbation of geodesic} to show that if the geodesic $(\Lambda_t)_t$ is of regularity $C^{1,1}$ at the cone points, so are the geodesics connecting $\Lambda_0$ to any $\Lambda~\in~\mathcal{Y}.$
In~\cite{solomon-yuval}, there are examples of geodesics of positive Lagrangians of arbitrary dimension, many of which satisfy the conditions of Theorem~\ref{theorem: perturbation of geodesic}. However, they are all preserved by an isometric action of $O(n)$ on the ambient manifold $X.$ From Theorem~\ref{theorem: perturbation of geodesic} we obtain the following.
\begin{cor}
There exist geodesics of positive Lagrangians in arbitrary dimension that are not invariant under any isometries of the ambient manifold.
\end{cor}

	It should be possible to extend Theorems~\ref{theorem: geodesic-cylinder correspondence} and~\ref{theorem: perturbation of geodesic} to Lagrangians with more complicated topology and more critical points. Furthermore, it should be possible to extend the techniques of this paper to prove the existence of geodesics a priori. In fact, to show the existence of a geodesic and hence an isotopy between two positive Lagrangians, it may not be necessary to assume the existence of a Hamiltonian isotopy between them, but only an intersection point of Maslov index zero. In the sequel~\cite{spider}, we show that there exists a one parameter family of imaginary special Lagrangian cylinders near any such intersection point. It remains to identify situations in which this family of cylinders can be extended until it terminates at an intersection point of Maslov index $n.$ When $n = 2$ and $X$ is hyperk\"ahler, the relation between special Lagrangian submanifolds and holomorphic curves should simplify the analysis. It would also be interesting to study the analogy between the results of this paper and the work relating geodesics in the space of K\"ahler metrics with families of holomorphic disks~\cite{chen-tian,donaldson-holomorphicdiscs,ross-witt-nystrom,semmes}. We plan to address these points in future work.

\subsection{Outline}
In Section~\ref{section:background}, we collect relevant background material. First, we set up a framework for working with immersed submanifolds with boundary. An immersed submanifold is defined to be an orbit of the action of the diffeomorphism group of a manifold acting on immersions of that manifold. When a point of the orbit has trivial stabilizer,
the immersed submanifold is called free. We show that the space of free immersed Lagrangian submanifolds satisfying appropriate Lagrangian boundary conditions is a Frechet manifold. This result provides the functional analytic framework to study the families of imaginary special Lagrangian cylinders that arise from the cylindrical transform. An important step in constructing the Frechet manifold structure is a refinement of the Weinstein neighborhood theorem for Lagrangians with boundary in a collection of Lagrangians submanifolds. We conclude by recalling results from~\cite{solomon,solomon-curv} on the space of positive Lagrangian submanifolds in a Calabi-Yau manifold.

Section~\ref{subsection: oriented blowups and cone-immersed submanifolds} develops a formalism for differential analysis on immersed submanifolds with cone-points. The formalism is based on the notion of cone-smoothness. A cone-smooth map is smooth away from a finite collection of points, where the derivative does not exist, but still a ``cone-derivative'' can be defined, which is not linear, but merely $1$-homogeneous. The precise definition of cone-smoothness is formulated in terms of the oriented blowup at the finite collection of points. Cone-immersions and cone-diffeomorphisms are defined and their basic properties are established. Cone-smooth vector fields, Riemannian metrics and differential forms are defined. Pull-backs and integrals of differential forms are discussed in the cone-smooth context, as well as contractions of differential forms by vector fields. The Hessian of a cone-smooth function is defined and its properties at extrema are studied.

Section~\ref{subsection: cone-immersed Lagrangians} uses the cone-smooth analysis developed in Section~\ref{subsection: oriented blowups and cone-immersed submanifolds} to establish a theory of geodesics of cone-immersed positive Lagrangians. Integration of cone-smooth differential forms is used to extend the metric~\eqref{equation: the metric} to cone-immersed positive Lagrangians. The flow of a cone-smooth vector field is used to construct a canonical parameterization of a family of cone-smooth positive Lagrangians given a parameterization of one member of the family. Following~\cite{solomon} in the smooth case, the canonical parameterizations are called horizontal liftings and the family is a geodesic if its Hamiltonian is time independent in a horizontal lifting.

In Section~\ref{section: lagrangian and special lagrangian cylinders}, we study spaces of Lagrangian and imaginary special Lagrangian cylinders. It is shown in Section~\ref{subsection: the space of lagrangian cylinders} that the linearization of positive Lagrangian boundary conditions for the imaginary special Lagrangian equation gives rise to the boundary conditions for the Laplace-Beltrami operator on $1$-forms that arise in the Hodge theorem for manifolds with boundary. A cohomological argument then shows the space of imaginary special Lagrangian cylinders with boundary in positive Lagrangians is $1$-dimensional proving Theorem~\ref{theorem: space of ISL cylinders is 1-dim}. In the rest of Section~\ref{section: lagrangian and special lagrangian cylinders}, additional tools for working with imaginary special Lagrangian cylinders are introduced for future use in Section~\ref{section: the relation between cylinders and geodesics} in the construction of the inverse of the cylindrical transform. Section~\ref{subsection: relative Lagrangian Flux} gives the definition of relative Lagrangian flux between two Lagrangian cylinders with boundary in Lagrangian submanifolds.
Section~\ref{subsection: regular families of special Lagrangian cylinders} defines two types of ``regular'' parameterizations of open intervals in the space of imaginary special Lagrangian cylinders. Interior regular parameterizations pertain to any interval. If at one end of the interval the cylinders shrink to an intersection point of the positive Lagrangian boundary conditions, an interior regular parameterization satisfying a uniform regularity condition under rescaling is called a regular parameterization about the intersection point.
It is shown that both types of regular parameterization can be modified so that they satisfy a condition called ``compatibility with harmonics.''

Section~\ref{section: the relation between cylinders and geodesics} begins with the proof of Theorem~\ref{theorem: family}, which shows how imaginary special Lagrangian cylinders arise from a geodesic of positive Lagrangians. The proof consists of a direct calculation combined with the fact that the linearization of the imaginary special Lagrangian equation is the Laplace equation. The remainder of Section~\ref{section: the relation between cylinders and geodesics} is devoted to Theorem~\ref{theorem: geodesic-cylinder correspondence}, which asserts that the cylindrical transform is one-to-one from geodesics of positive Lagrangians onto regular components of the space of imaginary special Lagrangian cylinders. A regular component is an interval that admits an interior regular parameterization and a regular parameterization about an intersection point at each end of the interval.
Given a regular component of the space of imaginary special Lagrangian cylinders, the regular parameterizations compatible with harmonics constructed in Section~\ref{section: lagrangian and special lagrangian cylinders} are used to construct a family of cone-smooth positive Lagrangians. The cone-smooth immersions giving these positive Lagrangians are shown to be a horizontal lift. Finally, the relative Lagrangian flux of the family of cylinders, as defined in Section~\ref{section: lagrangian and special lagrangian cylinders}, is shown to be a time independent Hamiltonian function of this horizontal lift. Thus, the family of positive Lagrangians is a geodesic. Conversely, using properties of the Hessian of a cone-smooth function proved in Section~\ref{subsection: oriented blowups and cone-immersed submanifolds}, it is shown that the cylindrical transform of a cone-smooth geodesic of positive Lagrangians is a regular component of the space of imaginary special Lagrangian cylinders.

Section~\ref{section: proof of perturbation theorem} gives the proof of Theorem~\ref{theorem: perturbation of geodesic} showing that geodesics persist under perturbations of the endpoints. By Theorem~\ref{theorem: geodesic-cylinder correspondence}, it suffices to show that regular components of the space of imaginary special Lagrangian cylinders are stable under perturbation of the positive Lagrangian boundary conditions. Given any cylinder in a regular component, an implicit function theorem argument constructs a family of small intervals of imaginary special Lagrangian cylinders depending continuously on the positive Lagrangian boundary conditions. The construction gives also interior regular parameterizations of these intervals. Next we consider an intersection point of the positive Lagrangian boundary conditions to which a regular family converges. An argument combining rescaling and the implicit function theorem constructs a family of small intervals of imaginary special Lagrangian cylinders depending continuously on the positive Lagrangian boundary conditions and shrinking to intersection points thereof. The construction gives regular parameterizations of these intervals about the intersection points. Finally, the compactness of the closed interval is used to obtain a finite collection of families of small intervals that can be patched together to obtain a family of regular components of the space of imaginary special Lagrangian cylinders depending continuously on the positive boundary conditions.

\subsection{Acknowledgments}
The authors would like to thank T. Collins, I. Haran, Y. Rubinstein, P. Seidel, and G. Tian, for helpful conversations.
The authors were partially supported by ERC starting grant 337560 and ISF grant 569/18. J.~S. was partially supported by ISF grant 1127/22 and the Miriam and Julius Vinik
Chair in Mathematics. A.~Y. was partially supported by the Adams Fellowship Program of the Israel Academy of Sciences and Humanities. The authors would like to thank the Institute for Advanced Study for its hospitality in the initial stages of the research that led to this paper, funded by the Erik Ellentuck Fellowship and the IAS Fund for Math.
	
	\section{Background}
\label{section:background}
The present section recalls and establishes results on Lagrangian submanifolds that form the basis for the main results of the paper.
Section~\ref{sssec:immsubm} defines an immersed submanifold to be an orbit of the diffeomorphism group of a manifold on the space of immersions of that manifold. Points, tangent spaces, differential forms and maps of immersed submanifolds are discussed. Immersed submanifolds represented by an immersion with trivial stabilizer group are called free. Criteria for being free are established, and being free is shown to be an open condition. Along the way, the action of the diffeomorphism group on the space of immersions is shown to be proper.

Section~\ref{sssec:spimmLag} establishes notation for the space of free immersed Lagrangians with boundary in a given collection of Lagrangian submanifolds. Furthermore, a relative Weinstein neighborhood theorem for an immersed Lagrangian submanifold satisfying Lagrangian boundary conditions is proved. Namely, a neighborhood of the submanifold is the image of a local diffeomorphism from a neighborhood of the zero section of its cotangent bundle. Moreover, the local diffeomorphism carries the conormal bundle of the boundary to the boundary conditions.

Section~\ref{sssec:Frechet} proves that the space of free immersed Lagrangians satisfying Lagrangian boundary conditions is a Frechet manifold. The proof uses the relative Weinstein theorem and the properness of the diffeomorphism group action on the space of immersions. The tangent space of this Frechet manifold is described and the notion of exact paths of immersed Lagrangians is defined. Section~\ref{subsection: calabi-yau manifolds} recalls definitions and results from~\cite{solomon} concerning the space of positive Lagrangian submanifolds in a Calabi-Yau manifold and geodesics therein.
	
	\subsection{Weinstein neighborhoods and immersed Lagrangian submanifolds with boundary}
	\label{subsection: Weinstein neighborhoods}

\subsubsection{Immersed submanifolds}\label{sssec:immsubm}
Let $N,M$ be smooth manifolds, $M$ perhaps with boundary. Denote by $\diff(M)$ the diffeomorphisms of
$M$	preserving each boundary component. That is, if $\varphi \in \diff(M)$ and $B \subset \partial M$ is a component, then $\varphi(B) = B.$
\begin{dfn}
		\label{definition: immersed submanifold}
		An \emph{immersed (resp. embedded) submanifold of $N$ of type $M$} is an equivalence class of immersions (resp. embeddings),
		\[
		K=[f:M\to N],
		\]
		where the equivalence is with respect to the $\diff(M)$-action: The immersions $f$ and $f'$ are equivalent if there exists $\varphi\in\diff(M)$ such that
		\[
		f'=f\circ\varphi.
		\]
We say that an immersion $f$ or an immersed submanifold $K=[f]$ is \emph{free} if $f$ has trivial isotropy subgroup. We say that $K$ has boundary if $M$ does. In this case, to each boundary component of $M$ we associate a boundary component of $K,$ which is itself an immersed submanifold.
	\end{dfn}

\begin{dfn}\label{definition:points}
	Let $K$ be an immersed submanifold of $N$ of type $M$. A \emph{point} $p\in K$ is an equivalence class of pairs, $p=[(f,q)]$, where $f:M\to N$ represents $K$ and $q\in M$. The pairs $(f,q)$ and $(f',q')$ are equivalent if there exists $\varphi\in\diff(M)$ such that
	\[
	f'=f\circ\varphi,\qquad\varphi(q')=q.
	\]
	For a point $p=[f,q]$ of  $K=[f]$, we let $p_0$ denote the image of $p$ in $N$,
	\[
	p_0:=f(q).
	\]
We say that $p$ is \emph{embedded} if $f^{-1}(p_0) = \{q\}.$
By abuse of notation, we abbreviate $p = f(q)$ when this does not lead to confusion. That is, we may consider $f(q)$ as a point of $K$ instead of as a point of $N.$

We define the \emph{tangent space} at $p$ by
	\[
	T_pK:=df_q(T_qM)\subset T_{p_0}N.
	\]
	The tangent bundle $TK$ is naturally an immersed submanifold of $TM$. A \emph{differential form} on $K$ is an equivalence class of pairs $\eta = [(f,\tau)]$ where $f$ is a representative of $K$ and $\tau \in \Omega^*(M).$ The pairs $(f,\tau)$ and $(f',\tau')$ are equivalent if there exists $\varphi \in \diff(M)$ such that $f' = f \circ \varphi$ and $\tau' = f^* \tau.$
An \emph{open subset} of $K$ is an immersed submanifold of $N$ of the form $[f|_U]$ where $f : M \to N$ represents $K$ and $U \subset M$ is open. For a differential form $\eta = [(f,\tau)]$ on $K,$ we may write $\tau = f^*\eta.$ When $\eta$ is a zero form, that is, a function, we also write $\tau = \eta\circ f.$

Let $K_i$ be immersed submanifolds of $N_i$ of type $M_i$ for $i = 0,1.$ A \emph{smooth map} $g : K \to K'$ is an equivalence class of triples $(f_0,f_1,h)$ where $h : M_0 \to M_1$ is smooth and $f_i : M_i \to N_i$ represents $K_i.$ Two such triples $(f_0,f_1,h)$ and $(f_0',f_1',h')$ are equivalent if there exist $\varphi_i \in \diff(M_i)$ such that $f_i' = f_i \circ \varphi_i$ and $h' = \varphi_1^{-1} \circ h \circ \varphi_0.$ We say $g$ is a diffeomorphism, embedding, and so on, if it is represented by a triple $(f_0,f_1,h)$ where $h$ is a diffeomorphism, embedding, and so on.
\end{dfn}

\begin{rem}\label{remark:tomakejakehappybutdeepdownalsoamitai}
Every point of an immersed submanifold has a neighborhood that is embedded. An embedded submanifold $K = [f : M \to N]$ is canonically diffeomorphic to the submanifold of type $f(M)$ represented by the inclusion $f(M) \hookrightarrow N.$ So, the usual notions of points, maps, and so on apply, and we do not need the added complexity of Definition~\ref{definition:points}.
\end{rem}

Lemma~\ref{lemma: michor's lemma} below is a rephrasing of~\cite[Lemma~1.4]{cervera-mascaro-michor}, which provides a sufficient condition for an immersed submanifold to be free. Lemma~\ref{lemma: embedded boundary point implies freeness} is a variant. Figure~\ref{fig:ebp implies freeness} shows an immersed submanifold satisfying the hypothesis of Lemma~\ref{lemma: embedded boundary point implies freeness}.

\begin{lemma}
	\label{lemma: michor's lemma}
	Let $K$ be a connected immersed submanifold of $N$ of type $M,$ and suppose $K$ has an embedded point. Then $K$ is free.
\end{lemma}

\begin{figure}[ht]
\centering
\includegraphics[width=12cm]{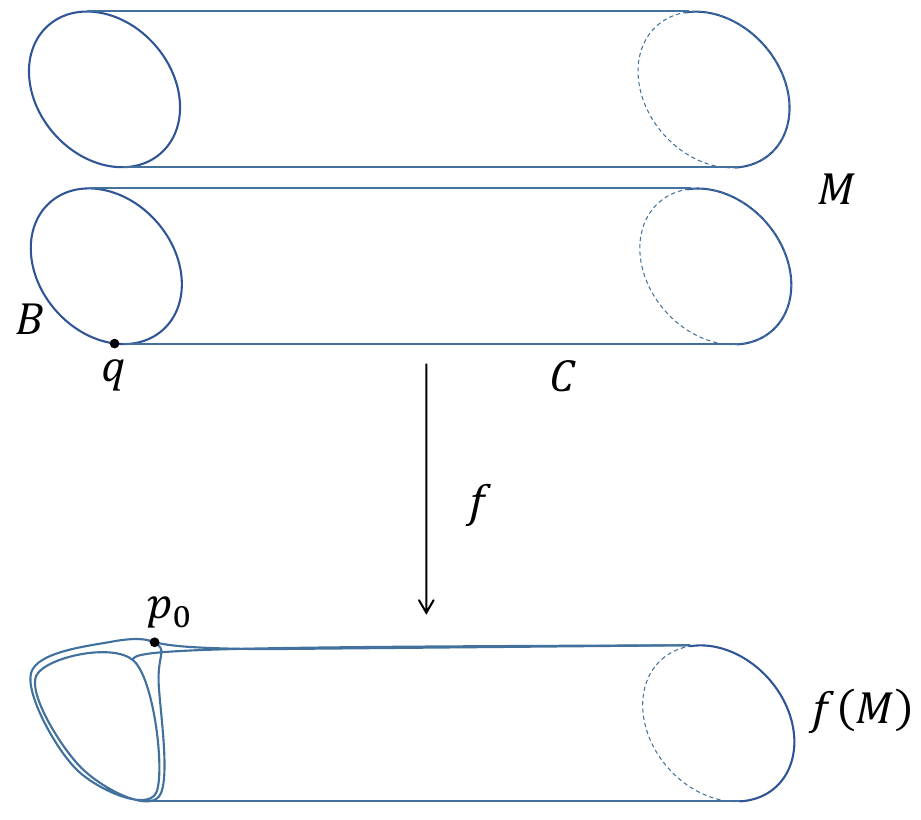}
\caption{The immersed manifold $K = [f: M \to N]$ satisfies the hypothesis of Lemma~\ref{lemma: embedded boundary point implies freeness} and therefore is free. The restriction of $f$ to the component $C \subset M$ is injective near the point $q$ of the boundary component $B \subset \partial C,$ but away from $B$ it is two-to-one. The restriction of $f$ to $M\setminus C$ is given by $f|_C$ composed with a diffeomorphism $M \setminus C \to C.$ So $K$ has no embedded points. However, the point $p = [f|_B,q]$ is an embedded point of the boundary component $[f|_B: B \to N]$.}
\label{fig:ebp implies freeness}
\end{figure}

	 \begin{lemma}
	 	\label{lemma: embedded boundary point implies freeness}
	 	Let $K$ be an immersed submanifold of $N$ modeled on $M.$ Suppose that for each connected component $C \subset M$ there is a boundary component $B \subset \partial C$ such that the corresponding boundary component of $K$ has an embedded point. Then $K$ is free.
	 \end{lemma}

 	\begin{proof}
 	Let $f: M \to N$ be a representative of $K$ and let $\varphi \in \diff(M)$ be such that $f \circ \varphi = f.$ Let $C \subset M,$ and $B \subset \partial C$ be components. Suppose $q \in B$ is such that $p = [(f|_B,q)]$ is an embedded point of the immersed submanifold $[f|_B].$  As $\varphi$ preserves each boundary component of $M$, we have $\varphi(q)\in B$. Since $p$ is an embedded point of $[f|_B],$ the equality
 		\[
 		f(\varphi(q))=f(q)
 		\]
 		implies
 		\[
 		\varphi(q)=q.
 		\]
 		Now, by local injectivity of $f$, continuity of $\varphi$ and the equality $f\circ\varphi=f,$ the subset
 		\[
 		U:=\{r\in M\;|\varphi(r)=r\}\subset M
 		\]
 		is open. As $U$ is also closed and $q\in U,$ we deduce $\varphi|_C=\id_{C}$ and complete the proof.
 	\end{proof}

Let $\imm(N,M)$ denote the Frechet manifold of smooth immersions $M \to N.$
\begin{lemma}\label{lm:proper}
Suppose $M$ is compact. Then the action of $\diff(M)$ on $\imm(N,M)$ by composition is proper.
\end{lemma}
\begin{proof}
In the following, we work with the topology of $C^\infty$ convergence on maps between manifolds. We fix Riemannian metrics on $M$ and $N$, which give rise to norms and covariant derivatives of tensors.  We need to show that if $f_i : M \to N$ is a sequence of immersions converging to an immersion $f$ and $h_i : M \to M$ is a sequence of diffeomorphisms such that $f_i \circ h_i$ converges to $g$, then after passing to a subsequence $h_i$ converges to a diffeomorphism $h.$ Indeed, consider the sequence of differentials $df_i : TM \to f_i^* TN.$ Since $f_i \to f$ and
$f$ is an immersion, there is a sequence of maps $L_i : f_i^*TN \to TM$ such that $L_i \circ df_i = \id_{TM}$ and there exists a uniform bound $|L_i|_{C^0} \leq C.$ On the other hand, since $f_i \circ h_i \to g,$ we have a uniform bound on $|d(f_i \circ h_i)|_{C^0} = |df_i \circ dh_i|_{C^0}$. Composing with $L_i,$ we obtain a uniform bound on $|dh_i|_{C^0}.$ Similarly, we have uniform bounds $|df_i|_{C^k} \leq C_k$ and $|df_i \circ dh_i|_{C^k} \leq D_k$ for all $k.$ Taking covariant derivatives, we obtain
\begin{align*}
|df_i \circ \nabla^k dh_i|_{C^0} &\leq |\nabla^k(df_i \circ dh_i)|_{C^0} + \star \\
 &\leq D_k + \text{a polynomial in $C_k$ and $|dh_i|_{C^{k-1}}$},
\end{align*}
where $\star$ denotes terms involving compositions of $\nabla^l df$ for $l \leq k$ and $\nabla^m dh$ for $m \leq k-1$.
Composing with $L_i$, by induction on $k$ we obtain uniform bounds on $|dh_i|_{C^k}$ for all $k.$ Thus, by the Arzela-Ascoli theorem and a diagonal argument, after passing to a subsequence, $h_i \to h$ in the $C^\infty$ topology.

It remains to show that $h$ is a diffeomorphism. Write $g_i = f_i \circ h_i,$ so $f_i = g_i \circ h_i^{-1}.$ Repeating the above argument with $g_i$ in place of $f_i$ and $h_i^{-1}$ in place of $h_i$ shows that after passing to a subsequence we have also $h_i^{-1} \to \hat h.$ By the continuity of composition, $\hat h \circ h = \id_M$ and $h \circ \hat h = \id_M,$ so $h$ is a diffeomorphism as required.
\end{proof}

The following is part of~\cite[Theorem 1.5]{cervera-mascaro-michor}, but we include it here for completeness.
\begin{lemma}\label{lm:freeopen}
Suppose $M$ is compact. Then free immersions form an open subset of $\imm(N,M).$
\end{lemma}
\begin{proof}
We show the complement of free immersions is closed. Let $f_i$ be a sequence of non-free immersions in $\imm(N,M)$ that converges to $f.$ Let $h_i \in \diff(M)$ be such that $h_i \neq \id_M$ and $f_i \circ h_i = f_i.$ Lemma~\ref{lm:proper} implies that after passing to a subsequence, $h_i$ converges to $h$. Moreover, by continuity $f \circ h = f.$ It remains to show that $h \neq \id_M.$ Indeed, let $\{U_j\}_{j = 1}^N$ be an open cover of $M$ such that $U_j$ is connected and $f|_{\overline{U_j}}$ is an embedding. Since $f_i \to f,$ after dropping a finite number of $i,$ we may assume $f_i|_{U_j}$ is an embedding.

We claim that if $h_i(U_j) \cap U_j \neq \emptyset$ then $h_i(x) = x$ for all $x \in U_j.$ The proof follows the idea of~\cite[Section 3.1]{cervera-mascaro-michor}: Suppose $x \in U_j$ and $h_i(x) \in U_j.$ Since $f_i(x) = f_i(h_i(x))$ and $f|_{U_j}$ is injective, it follows that $x = h_i(x).$ Thus
\begin{equation}
U_j \cap h_i(U_j) = \{x \in U_j|h_i(x) = x\}. \label{eq:ujhiuj}
\end{equation}
In particular, $U_j \cap h_i(U_j)$ is closed and open in the connected set $U_j$. Therefore, $U_j = U_j \cap h_i(U_j)$ and the claim follows from equation~\eqref{eq:ujhiuj}.

By the preceding claim, since $h_i \neq \id_M,$ after passing to a subsequence, there exists $j_0$ such that $h_i(U_{j_0}) \cap U_{j_0} = \emptyset$ for all $i.$ If $x \in U_{j_0},$ then $h(x) = \lim_i h_i(x) \notin U_{j_0}.$ It follows that $h \neq \id_M$ as desired.
\end{proof}

\subsubsection{The space of immersed Lagrangian submanifolds} \label{sssec:spimmLag}
	Let $(X,\omega)$ be a symplectic manifold of dimension $2n,$ and let $L$ be a smooth manifold of dimension $n,$ perhaps with boundary.

	\begin{dfn}
		\label{definition: Lagrangian submanifold}
		An immersion $f:L\to X$ is said to be \emph{Lagrangian} if it satisfies $f^*\omega=0$. Let $f,f':L\to X$ be immersions such that, for some $\varphi\in\diff(L),$ we have $f'=f\circ\varphi.$ Then $f$ is Lagrangian if and only if $f'$ is. We say the associated immersed submanifold $[f]$ is \emph{Lagrangian} if $f$ is Lagrangian.
	\end{dfn}

\begin{ntn}	\label{notation:space of Lagrangians}
	If $L$ has no boundary, we let $\mathcal{L}(X,L)$ denote the space of free immersed Lagrangian submanifolds of $X$ of type $L$. Suppose now that $L$ has $k$ boundary components $B_1,\ldots,B_k,$ for some $k\in\mathbb{N}$, and let $\Lambda_1,\ldots,\Lambda_k\subset X$ be fixed Lagrangians. Let $\mathcal{L}(X,L;\Lambda_1,\ldots,\Lambda_k)$ denote the space of free immersed Lagrangian submanifolds $Z\subset X$ of type $L$ with boundary components $C_1,\ldots,C_k,$ corresponding to the boundary components $B_1,\ldots,B_k,$ of $L$ such that the following hold.
	\begin{enumerate}[label=(\alph*)]
		\item For $i=1,\ldots,k,$ the boundary component $C_i$ is an immersed submanifold of $\Lambda_i$.
		\item \label{not tangent} For $i=1,\ldots,k,$ and $p\in C_i,$ we have $T_pZ\ne T_{p_0}\Lambda_i$.
	\end{enumerate}
\end{ntn}
	It is proved in Theorem~\ref{thm: frechet} below that the spaces $\mathcal{L}(X,L;\Lambda_1,\ldots,\Lambda_k)$ and, in particular, $\mathcal{L}(X,L)$ are Fr\'echet manifolds locally parameterized by spaces of closed $1$-forms. It was observed previously in~\cite[Section~2]{akveld-salamon} that the open subspace of $\mathcal{L}(L,X)$ consisting of embedded Lagrangian submanifolds is a Frechet manifold. The proof of Theorem~\ref{thm: frechet} uses the following version of the Weinstein neighborhood theorem for immersed Lagrangian submanifolds with boundary that satisfy Lagrangian boundary conditions. For a smooth manifold $M$ and a submanifold $Q\subset M$ we let $\nu_Q$ denote the conormal bundle of $Q$ in $T^*M$.
	
	\begin{lemma}
		\label{lemma: Weinstein with boundary}
		Suppose $L$ is compact with $k$ boundary components, $B_1,\ldots,B_k.$ Let $\Lambda_1,\ldots,\Lambda_k\subset X$ be fixed embedded Lagrangian submanifolds and let
\[
Z=[f:L\to X]\in\mathcal{L}(X,L;\Lambda_1,\ldots,\Lambda_k).
\]
Then there exist an open neighborhood of the zero section, $V\subset T^*L$, and a local symplectomorphism $\varphi:V\to X$, such that the following hold.
		\begin{enumerate}[label=(\alph*)]
			\item Identifying $L$ with the zero section in $T^*L,$ we have
			\[
			\varphi|_L=f.
			\]
			\item\label{it:bdrycond} For $i=1,\ldots,k$, a point $p\in B_i$ and a covector $\xi\in T^*_pL,$ we have
			\[
			\varphi(p,\xi)\in\Lambda_i\Leftrightarrow(p,\xi)\in \nu_{B_i}.
			\]
		\end{enumerate}
	\end{lemma}

	\begin{proof}
		We follow the lines of the argument by Moser presented in~\cite[Section~3.2]{babymcduff-salamon}, making the necessary adaptations. The usual argument treats embedded Lagrangians without boundary. The modification for immersed Lagrangians is relatively straightforward. More work is required to show the boundary condition~\ref{it:bdrycond} can be satisfied. The usual argument has three steps and each one has to be adapted to satisfy the boundary condition. In the first step, one constructs a Lagrangian subbundle of $f^*TX$ complementary to $TL.$ This sub-bundle must be adapted to the tangent bundles of $\Lambda_i$ along the boundary components. In the second step, one uses the exponential map to construct a local diffeomorphism $\tilde \varphi$ satisfying the requirements for $\varphi$ except that it is not a local symplectomorphism. The connection used to define the exponential map must be chosen carefully for the boundary condition~\ref{it:bdrycond} to hold. The third step uses Moser's argument to modify $\tilde \varphi$ to a local symplectomorphism. It is necessary to check that this modification preserves the boundary condition~\ref{it:bdrycond}.

First we construct a smooth subbundle $E\subset f^*TX\to L$ satisfying the following.
		\begin{enumerate}
			\item\label{fiber is Lag} For $p\in L,$ the fiber $E_p\subset T_{f(p)}X$ is Lagrangian.
			\item\label{direct sum} We have $f^*TX=df(TL)\oplus E,$ where $df : TL \to f^*TX$ is the differential of the immersion $f.$
			\item\label{compatible with boundary} For $p\in B_i,\;i=1,\ldots,k,$ we have
			\[
			E_p\cap T_{f(p)}\Lambda_i\ne\{0\}.
			\]
			It then follows that $\dim\left(E_p\cap T_{f(p)}\Lambda_i\right)=1.$
		\end{enumerate}
		To do so, let $J$ be an $\omega$-compatible almost complex structure. Note that, if the Lagrangian $Z$ were closed, the bundle $JTZ\subset f^*TX$ would be a good choice for~$E.$ In our case, however, $JTZ$ does not necessarily satisfy condition~\eqref{compatible with boundary}. We thus perform the following perturbation. For $p\in L$ and any number $a\in\R,$ the linear space
		\[
		E_{p,a}:=(J+a)df(T_{p}L)\subset T_{f(p)}X
		\]
		satisfies conditions~\eqref{fiber is Lag} and~\eqref{direct sum}. Condition~\eqref{compatible with boundary} determines $a$ uniquely for $p \in \partial L$. Here, we use condition~\ref{not tangent} in the definition of $\mathcal{L}(X,L;\Lambda_1,\ldots,\Lambda_k).$ Hence, constructing $E$ amounts to extending the smooth assignment $p\mapsto a(p)$ from $\partial L$ to all~$L$.
		
		Contraction with $\omega$ along the immersion $f$ yields an isomorphism of vector bundles $E\cong T^*L.$ We thus think of $T^*L$ as a subbundle of $f^*TX$. Condition~\eqref{compatible with boundary} in the construction of $E$ implies that for $p\in B_i,i=1,\ldots,k,$ the annihilator $(T_pB_i)^0\subset T_p^*L$ is identified with a line in $T_{f(p)}\Lambda_i.$
		
		For $i=1,\ldots,k,$ let $\nabla^i$ be a connection on $TX,$ with respect to which the Lagrangian $\Lambda_i$ is totally geodesic. Let $\left(\alpha_i:L\to\mathbb{R}\right)_{i=1}^k$ be a smooth partition of unity with
		\[
		\alpha_i|_{B_i}\equiv1,\quad i=1,\ldots,k.
		\]
		For a connection $\nabla$ on $TX,$ let $\exp^\nabla$ denote the exponential map of $\nabla.$ To each $p\in L$ we assign a connection $\nabla^p$ on $TX$ given by
		\[
		\nabla^p:=\sum_i\alpha_i(p)\nabla^i.
		\]
		Set
		\[
		\widetilde{V}:=\left\{(p,\xi)\in T^*L\,\left|\,\exp^{\nabla^p}_{f(p)}(\xi)\;\mathrm{exists}\right.\right\}
		\]
		and define
		\[
		\widetilde{\varphi}:\widetilde{V}\to X,\qquad(p,\xi)\mapsto\exp^{\nabla^p}_{f(p)}(\xi).
		\]
		Shrinking $\widetilde{V}$ if necessary, $\widetilde{\varphi}$ is an immersion and one-to-one on each fiber. By construction, we have
\begin{equation}
\label{equation: conormal mapped into Lambda}
\widetilde{\varphi}\left(\nu_{B_i}\cap\widetilde{V}\right)\subset\Lambda_i,\quad i=1,\ldots,k.
\end{equation}

Let $\omega_0$ denote the canonical symplectic form on $\widetilde{V}\subset T^*L$. Write
		\[
		\omega_1:=\widetilde{\varphi}^*\omega.
		\]
		Then $\omega_0$ and $\omega_1$ coincide on $T\widetilde{V}|_L$. Also, by~\eqref{equation: conormal mapped into Lambda}, the conormal bundle $\nu_{B_i}$ is Lagrangian with respect to both forms for $i=1,\ldots,k.$ We now construct the desired $V\subset T^*L$ and an open embedding $\chi:V\hookrightarrow\widetilde{V},$ satisfying the following:
		\begin{enumerate}
			\item $\chi^*\omega_1=\omega_0$.
			\item For $q\in V$ and $i=1,\ldots,k,$ we have
			\[
			\chi(q)\in\nu_{B_i}\Leftrightarrow q\in\nu_{B_i}.
			\]
			\item For $q\in L\subset V$ we have $\chi(q)=q$.
		\end{enumerate}
		This will complete the proof since we can take $\varphi:=\widetilde{\varphi}\circ\chi$. The embedding $\chi$ is obtained as the flow of a time-dependent vector field as follows. Let
		\[
		H:\widetilde{V}\times[0,1]\to\widetilde{V},\qquad((p,\xi),t)\mapsto(p,t\xi)
		\]
		and
		\[
		\pi:\widetilde{V}\times[0,1]\to\widetilde{V},\qquad((p,\xi),t)\mapsto(p,\xi).
		\]
		Let $\eta$ be the 1-form on $\widetilde{V}$ given by
		\[
		\eta:=\pi_*H^*(\omega_1-\omega_0),
		\]
		where $\pi_*$ denotes pushing forward by integration along the fibers of $\pi$. For a discussion of the properties of integration along the fiber, see~\cite[Section~3.1]{kupferman-solomon}. Then $\eta$ satisfies
		\begin{equation}
		\label{equation: properties of eta}
		d\eta=\omega_1-\omega_0,\quad\eta|_{T_L\widetilde{V}}=0,\quad\eta|_{\nu_{B_i}}=0,\;i=1,\ldots,k.
		\end{equation}
		Shrinking $\widetilde{V}$ again, we may assume that, for $t\in[0,1],$ the closed 2-form $\omega_t:=t\omega_1+(1-t)\omega_0$ is non-degenerate. For $t\in[0,1],$ define a vector field $u_t$ on $\widetilde{V}$ by
		\[
		i_{u_t}\omega_t=-\eta.
		\]
		Then $u_t$ vanishes on $L$ and is tangent to $\nu_{B_i}$ for $i=1,\ldots,k.$ Let $\chi_t,\;t\in[0,1],$ denote the flow of $u_t$. That is,
		\[
		\chi_0=\id,\qquad\deriv{t}\chi_t=u_t\circ\chi_t.
		\]
		Let $V$ be the domain of $\chi:=\chi_1$. By the Cartan formula we have
		\begin{align*}
		\deriv{t}\chi_t^*\omega_t
		&=\chi_t^*\left(\deriv{t}\omega_t+i_{u_t}d\omega_t+di_{u_t}\omega_t\right)\\
		&=\chi_t^*((\omega_1-\omega_0)-d\eta)\\
		&=0.
		\end{align*}
		Hence we have $\chi^*\omega_1=\omega_0,$ as desired.
	\end{proof}
	
	\begin{dfn}
		We call the pair $(V,\varphi)$ of Lemma~\ref{lemma: Weinstein with boundary} an \emph{immersed Weinstein neighborhood} of $Z$ compatible with $\Lambda_1,\ldots,\Lambda_k$. Alternatively, we may call $(V,\varphi)$ an immersed Weinstein neighborhood of $f$ compatible with $\Lambda_1,\ldots,\Lambda_k$.
	\end{dfn}

\subsubsection{Frechet manifold}\label{sssec:Frechet}
	We now show the space $\mathcal{L}(X,L;\Lambda_1,\ldots,\Lambda_k)$ is a Fr\'echet manifold. For a manifold $M$ with boundary, we let $\Omega^1(M)$ denote the space of smooth 1-forms on $M$ and $\Omega_B^1(M)\subset\Omega^1(M)$ the Fr\'echet subspace consisting of closed forms annihilating the boundary.
	
	\begin{thm}
		\label{thm: frechet}
		Suppose $L$ is compact with $k$ boundary components and let $\Lambda_i\subset X,\; i=1,\ldots,k,$ be fixed embedded Lagrangians. Then $\mathcal{L}(X,L;\Lambda_1,\ldots,\Lambda_k)$ is a Fr\'echet manifold. In fact, for $Z \in \mathcal{L}(X,L;\Lambda_1,\ldots,\Lambda_k)$ an immersed Weinstein neighborhood of $Z$ compatible with $\Lambda_1,\dots,\Lambda_k,$ gives rise to a local parameterization
		\[ \mathbf{X}:\mathcal{U}\subset\Omega_B^1(L)\to\widetilde{\mathcal{U}}\subset\mathcal{L}(X,L;\Lambda_1,\ldots,\Lambda_k).
		\]
	\end{thm}

	\begin{proof}
Let $B_1,\ldots,B_k,$ denote the components of $\partial L$. Let $\iml(X,L;\Lambda_1,\ldots,\Lambda_k)$ denote the Frechet manifold of Lagrangian immersions $g: L \to X$ such that
\[
g(B_i)\subset \Lambda_i, \qquad i = 1,\ldots,k.
\]
The diffeomorphism group $\diff(L)$ acts on $\iml(X,L;\Lambda_1,\ldots,\Lambda_k)$ by composition, and this action is proper by Lemma~\ref{lm:proper}. The subset
\[
\ilf(X,L;\Lambda_1,\ldots,\Lambda_k) \subset \iml(X,L;\Lambda_1,\ldots,\Lambda_k)
\]
consisting of free immersions is open by Lemma~\ref{lm:freeopen} and, therefore, also a Frechet manifold. It is preserved by the action of $\diff(L).$ By definition,
\[
\mathcal{L}(X,L;\Lambda_1,\ldots,\Lambda_k) = \ilf(X,L;\Lambda_1,\ldots,\Lambda_k)/\diff(L).
\]
The proof of the theorem broadly follows the proof~\cite[Theorem~9.16]{Lee} that the quotient of a finite dimensional manifold by the free proper action of a Lie group is a manifold. The main step is to construct coordinate charts on $\ilf(X,L;\Lambda_1,\ldots,\Lambda_k)$ in which the orbits of the $\diff(L)$ action are the fibers of a linear projection. These charts are called adapted. Unlike in the finite dimensional setting, where the inverse function theorem is used, adapted coordinate charts on $\ilf(X,L;\Lambda_1,\ldots,\Lambda_k)$ are constructed explicitly using immersed Weinstein neighborhoods. The construction is based on ideas from~\cite[Theorem 1.5]{cervera-mascaro-michor}.

We start with the proof that $\mathcal{L}(X,L;\Lambda_1,\ldots,\Lambda_k)$ is Hausdorff. Indeed, let
\[
a: \ilf(X,L;\Lambda_1,\ldots,\Lambda_k) \times \diff(L) \to \ilf(X,L;\Lambda_1,\ldots,\Lambda_k) \times \ilf(X,L;\Lambda_1,\ldots,\Lambda_k)
\]
be given by $a(f,h) = (f,f\circ h).$ So, $\im a$ is the set of pairs identified under the quotient map $\ilf(X,L;\Lambda_1,\ldots,\Lambda_k) \to \mathcal{L}(X,L;\Lambda_1,\ldots,\Lambda_k).$ Lemma~\ref{lm:proper} asserts that $a$ is proper and hence $\im a$ is closed. It follows that $\mathcal{L}(X,L;\Lambda_1,\ldots,\Lambda_k)$ is Hausdorff.

We turn to the construction of an adapted coordinate chart on $\ilf(X,L;\Lambda_1,\ldots,\Lambda_k)$ centered at $f \in \ilf(X,L;\Lambda_1,\ldots,\Lambda_k)$. Let $(V,\varphi)$ be an immersed Weinstein neighborhood of $f$ compatible with $\Lambda_1,\ldots,\Lambda_k,$ which is guaranteed to exist by Lemma~\ref{lemma: Weinstein with boundary}. Let $\pi : T^*L \to L$ denote the projection. After possibly shrinking $V,$ we may assume that each point of $L$ has a neighborhood $U$ such that $\varphi|_{V \cap \pi^{-1}(U)}$ is an embedding. Set
\[
		\mathcal{U}^1:=\left\{\left.\alpha\in\Omega_B^1(L)\;\right|\;\mathrm{Graph}(\alpha)\subset V\right\}.
\]
Then $\mathcal{U}^1$ is open in $\Omega_B^1(L).$ Thinking of $\alpha\in\mathcal{U}^1$ as a map $\alpha : L \to V,$ we can associate to $\alpha$ the immersion
\[
		\varphi \circ \alpha : L \to X.
\]
It follows from the compatibility of $(V,\varphi)$ with the Lagrangians $\Lambda_1,\ldots,\Lambda_k,$ that $\varphi \circ \alpha(B_i) \subset \Lambda_i.$ So,
\[
\varphi \circ \alpha \in \iml(X,L;\Lambda_1,\ldots,\Lambda_k).
\]

First, we show that there exists a neighborhood $\mathcal{V}$ of the identity in $\diff(L)$ such that the map
\[
\Phi: \mathcal{U}^1 \times \mathcal{V} \longrightarrow \iml(X,L;\Lambda_1,\ldots,\Lambda_k)
\]
given by $(\alpha,h) \mapsto \varphi\circ\alpha\circ h$ is a diffeomorphism onto its image. Since each point of $L$ has a neighborhood $U$ such that $\varphi|_{V \cap \pi^{-1}(U)}$ is an embedding, and $L$ is compact, we can choose $\{U_j\}_{j = 1}^N$ an open cover of $L$ such that $\varphi|_{V \cap \pi^{-1}(U_j)}$ is an embedding. Abbreviate $\widetilde U_j = V \cap \pi^{-1}(U_j).$ Let $\{T_j\}_j$ be an open cover of $L$ such that $\overline T_j \subset U_j.$ Let $\mathcal{V} \subset \diff(L)$ be a neighborhood of the identity such that for $h \in \mathcal{V}$ we have $h(\overline T_j) \subset U_j$ for $j = 1,\ldots,N.$ Let $\mathcal{W}^1 \subset \iml(X,L;\Lambda_1,\ldots,\Lambda_k)$ be the open set of Lagrangian immersions $g : L \to X$ satisfying the given boundary conditions and such that $g(\overline T_j) \subset \varphi(\widetilde U_j).$ Let $\iml(L,V;\nu_{B_1},\ldots,\nu_{B_k})$ denote the Frechet manifold of Lagrangian immersions $e: L \to V$ such that $e(B_i) \subset \nu_{B_i}$ for $i = 1,\ldots k.$ Define
\[
\Psi : \mathcal{W}^1 \longrightarrow \iml(L,V;\nu_{B_1},\ldots,\nu_{B_k})
\]
by
\[
\Psi(g)|_{T_j} = (\varphi|_{\widetilde U_j})^{-1}\circ g|_{T_j}.
\]
Let
\[
\mathcal{W} = \{g \in \mathcal{W}^1 \,|\, \pi \circ \Psi(g) : L \to L \text{ is a diffeomorphism}\}.
\]
Let
\[
\Theta : \mathcal{W} \to \mathcal{U}^1 \times \mathcal{V}
\]
be given by
\[
\Theta(g) = (\Psi(g) \circ (\pi \circ \Psi(g))^{-1},\pi \circ \Psi(g)).
\]
Since $\Theta$ and $\Phi$ are inverse to each other, we conclude that $\Phi$ is a diffeomorphism onto its image $\mathcal{W}.$

Next, we prove that there exists $\mathcal{U} \subset \mathcal{U}^1$ such that for every orbit $\mathcal{O}$ of $\diff(L),$ either
$(\Phi|_{\mathcal{U}\times \mathcal{V}})^{-1}(\mathcal{O}) = \emptyset$ or there exists $\alpha \in \mathcal{U}$ such that $(\Phi|_{\mathcal{U}\times \mathcal{V}})^{-1}(\mathcal{O}) = \{\alpha\} \times \mathcal{V}.$ Suppose not. Then there exists a neighborhood basis $\mathcal{U}^i$ of $\Omega^1_B(L)$ at $0$ with
\[
\mathcal{U}^{i+1} \subset \mathcal{U}^i,
\]
distinct $\alpha_i,\alpha_i' \in \mathcal{U}^i$ and $h_i \in \diff(L)$ such that $\Phi(\alpha_i,\id_L) \circ h_i = \Phi(\alpha_i',\id_L).$ In particular, $\alpha_i,\alpha_i' \to 0 \in \Omega^1_B(L).$ It follows that $\varphi\circ \alpha_i \to f$ and $\varphi \circ \alpha_i' \to f$ in $\imm(L,X).$ Moreover, $\varphi\circ\alpha_i \circ h_i = \varphi\circ \alpha_i'.$ By Lemma~\ref{lm:proper} the action of $\diff(L)$ on $\imm(L,X)$ is proper, so after passing to a subsequence, we can assume $h_i \to h \in \diff(L).$ By the continuity of composition, $f \circ h = f.$ Since $f$ is a free immersion, it follows that $h = \id_L.$ So, for $i$ sufficiently large, $h_{i} \in \mathcal{V}$ and $\Phi(\alpha_i,h_i) = \Phi(\alpha_i',\id_L).$ But this contradicts the injectivity of $\Phi$ since $\alpha_i,\alpha_i'$ are distinct.

By Lemma~\ref{lm:freeopen}, perhaps after shrinking $\mathcal{U}$, we may assume
\[
\Phi(\mathcal U \times \mathcal V) \subset \ilf(X,L;\Lambda_1,\ldots,\Lambda_k).
\]
Replacing $\Phi$ with its restriction to $\mathcal U \times \mathcal V,$ we obtain
\[
\Phi : \mathcal U \times \mathcal V \to \ilf(X,L;\Lambda_1,\ldots,\Lambda_k), \qquad \Phi(0,e) = f,
\]
that is a diffeomorphism onto its image and such that for every $\diff(L)$ orbit
\[
\mathcal{O} \subset \ilf(X,L;\Lambda_1,\ldots,\Lambda_k)
\]
of $\diff(L),$
\begin{equation}\label{eq:adapted}
\Phi^{-1}(\mathcal O) = \emptyset \qquad \text{or} \qquad \exists \alpha \in \mathcal U, \qquad \Phi^{-1}(\mathcal O) = \{\alpha\} \times \mathcal{V}.
\end{equation}
We call such a map $\Phi$ an \emph{adapted chart} on $\ilf(X,L;\Lambda_1,\ldots,\Lambda_k)$ centered at $f.$

To such an adapted chart $\Phi,$ associate a map
\[
\hat \Phi : \mathcal{U} \to \ilf(X,L;\Lambda_1,\ldots,\Lambda_k)
\]
given by $\hat \Phi(\alpha) = \Phi(\alpha,e)$ and
a map
\[
\mathbf{X}:\mathcal{U}\to \mathcal{L}(X,L;\Lambda_1,\ldots,\Lambda_k)
\]
given by $\mathbf{X}(\alpha) = [\hat \Phi(\alpha)].$
It follows from~\eqref{eq:adapted} that $\mathbf{X}$ is injective. We claim that $\mathbf{X}$ is a homeomorphism onto its image. Indeed, it suffices to show that $\mathbf{X}$ is open. Let
\[
p : \ilf(X,L;\Lambda_1,\ldots,\Lambda_k) \longrightarrow \mathcal{L}(X,L;\Lambda_1,\ldots,\Lambda_k)
\]
denote the quotient map. Observe that $p$ is open~\cite[Prop.~9.15]{Lee}. So, if $U \subset \mathcal{U}$ is open, then
\[
\mathbf{X}(U) = p(\hat\Phi(U)) \overset{\eqref{eq:adapted}}{=} p(\Phi(U \times \mathcal{V}))
\]
is open as desired.

We claim that the collection of all maps $\mathbf{X}$ associated to adapted charts $\Phi$ is a smooth structure on $\mathcal{L}(X,L;\Lambda_1,\ldots,\Lambda_k).$ Indeed, let
\[
\Phi_i : \mathcal U_i \times \mathcal V_i \longrightarrow \ilf(X,L;\Lambda_1,\ldots,\Lambda_k), \qquad i = 1,2,
\]
be adapted charts, and let
\[
\hat \Phi_i : \mathcal{U}_i \to \ilf(X,L;\Lambda_1,\ldots,\Lambda_k), \qquad \mathbf{X}_i:\mathcal{U}_i \longrightarrow \mathcal{L}(X,L;\Lambda_1,\ldots,\Lambda_k)
\]
be the associated maps. It suffices to show that
\begin{equation}\label{eq:cc}
\mathbf X_1^{-1} \circ \mathbf X_2 : \mathbf X_2^{-1}(\mathbf X_1(\mathcal U_1)) \longrightarrow \mathcal U_1
\end{equation}
is smooth. Let $\alpha \in \mathbf X_2^{-1}(\mathbf X_1(\mathcal U_1)).$ Let $h \in \diff(L)$ be such that
\[
\hat \Phi_2(\alpha) \circ h \in \hat \Phi_1(\mathcal U_1).
\]
After composing $\Phi_2$ with the diffeomorphism of $\ilf(X,L;\Lambda_1,\ldots,\Lambda_k)$ given by the action of $h,$ which does not change $\mathbf{X}_2,$ we may assume without loss of generality that $\hat\Phi_2(\alpha) \in \hat\Phi_1(\mathcal U_1).$ Let
\[
W = \hat\Phi_2^{-1}(\Phi_1(\mathcal{U}_1\times \mathcal{V}_1)).
\]
Thus, $W$ is an open set containing $\alpha.$ Let $\varpi : \mathcal{U}_1 \times \mathcal{V}_1 \to \mathcal{U}_1$ denote the projection. It follows from~\eqref{eq:adapted} that
\[
\mathbf{X_1}^{-1} \circ \mathbf{X}_2|_W = \varpi \circ \Phi_1^{-1} \circ \hat\Phi_2|_W,
\]
which is clearly smooth. Since $\alpha$ was arbitrary, we conclude that the composition~\eqref{eq:cc} is smooth everywhere.
\end{proof}

	Lemma~\ref{lemma: tangent space} and Remark~\ref{remark: tangent space hands on} below describe tangent vectors in spaces of Lagrangian submanifolds explicitly. Lemma~\ref{lemma: tangent space} \ref{tangent space of closed lag} is stated and proved in~\cite[Lemma~2.1]{akveld-salamon}, while part \ref{tangent space of lag with boundary} is the analogue for Lagrangians with boundary. Remark~\ref{remark: tangent space hands on} provides a hands-on approach.

	\begin{lemma}
		\label{lemma: tangent space}$\;$
		\begin{enumerate}[label=(\alph*)]
			\item\label{tangent space of closed lag} Suppose $L$ is closed. For $Z\in\mathcal{L}(X,L)$, the tangent space $T_Z\mathcal{L}(X,L)$ is canonically isomorphic to the space of closed $1$-forms on $Z$.
			\item\label{tangent space of lag with boundary} Suppose $L$ has $k$ boundary components. For $Z\in\mathcal{L}(X,L;\Lambda_1,\ldots,\Lambda_k)$, the tangent space $T_Z\mathcal{L}(X,L;\Lambda_1,\ldots,\Lambda_k)$ is canonically isomorphic to the space $\Omega_B^1(Z).$
		\end{enumerate}
	\end{lemma}

	\begin{rem}
		\label{remark: tangent space hands on}
		The canonical isomorphism of Lemma~\ref{lemma: tangent space} can be understood as follows. Let $(Z_t)_{t\in(-\epsilon,\epsilon)}$ be a smooth path of Lagrangians diffeomorphic to~$L,$ which may or may not have boundary. Let $\Psi_t:L\to X$ for $t\in(-\epsilon,\epsilon)$ be a smooth family of immersions such that $\Psi_t$ represents $Z_t.$ Then, recalling Definition~\ref{definition:points}, we have
		\[
		\deriv{t}Z_t=\left[\left(\Psi_t,i_{\deriv{t}\Psi_t}\omega\right)\right] \in \Omega_B^1(Z_t).
		\]
	\end{rem}

	\begin{dfn}$\;$
		\begin{enumerate}[label=(\alph*)]
			\item Suppose $L$ is closed. Let $\Lambda=(\Lambda_t)_{t\in[0,1]}\subset\mathcal{L}(X,L)$ be a smooth path. We say $\Lambda$ is \emph{exact} if $\deriv{t}\Lambda_t$ is exact for $t\in[0,1]$.
			\item Suppose $L$ has $k$ boundary components. Let
\[
Z=(Z_t)_{t\in[0,1]}\subset\mathcal{L}(X,L;\Lambda_1,\ldots,\Lambda_k)
\]
be a smooth path. We say $Z$ is \emph{exact relative to the boundary} if $\deriv{t}Z_t$ is exact relative to the boundary for $t\in[0,1]$, that is, if there exists a function $h_t\in C^\infty(Z_t)$ vanishing on $\partial Z_t$ such that $\deriv{t}Z_t=dh_t.$
		\end{enumerate}
	\end{dfn}

	\begin{rem}
		\label{remark: exact is Hamiltonian}
		Recall that a smooth path of embedded closed Lagrangians is exact if and only if it is induced by a Hamiltonian flow on $X$~\cite[Lemma~2.3]{akveld-salamon}.
	\end{rem}

	\subsection{Calabi-Yau manifolds, special Lagrangians and the space of positive Lagrangians}
	\label{subsection: calabi-yau manifolds}
	
	\begin{dfn}
		A \emph{Calabi-Yau manifold} is a quadruple $(X,\omega,J,\Omega)$, where $(X,\omega)$ is a symplectic manifold, $J$ is an $\omega$-compatible integrable complex structure, and $\Omega$ is a nowhere-vanishing holomorphic $(n,0)$ form (with respect to $J$). In particular, $X$ is a K\"ahler manifold with the metric
		\[
		g=g_J=\omega(\cdot,J\cdot).
		\]
	\end{dfn}

	\begin{rem}
		In the literature, various definitions of Calabi-Yau manifolds can be found. In previous work by the authors, the notion used here would be called an \emph{almost Calabi-Yau manifold}. A more restrictive definition adds the requirement $\rho\equiv1,$ where $\rho$ is the function defined in~\eqref{equation: rho}.
	\end{rem}

	In what follows, we fix a Calabi-Yau manifold $(X,\omega,J,\Omega)$. Recall the following observation of~\cite[Theorem~III.1.7]{harvey-lawson}.
	
	\begin{lemma}
		\label{lemma: Omega does not vanish on Lags}
		Let $p\in X$ and let $S\subset T_pX$ be an oriented Lagrangian subspace. Then $\Omega$ does not vanish on $S$. In fact, for an oriented basis $v_1,\ldots,v_n\in S,$ we have
		\[
		|\Omega(v_1,\ldots,v_n)|=\rho\vol(v_1,\ldots,v_n),
		\]
		where $\vol$ denotes the Riemannian volume form.
	\end{lemma}

	In particular, it follows from Lemma~\ref{lemma: Omega does not vanish on Lags} that if $\imaginary \Omega|_\Lambda = 0,$ then $\real\Omega|_\Lambda$ is non-vanishing. If $\real\Omega|_\Lambda = 0,$ then $\imaginary\Omega|_\Lambda$ is non-vanishing.

\begin{dfn}
Let $p \in X$ and let $S \subset T_pX$ be an oriented Lagrangian subspace. The \emph{phase} $\theta_S \in S^1$ of $S$ is the argument of $\Omega(v_1,\ldots,v_n)$ for $v_1,\ldots,v_n \in S$ an oriented basis. We say that $S$ is positive if $\theta_S \in \left(-\frac{\pi}{2},\frac{\pi}{2}\right).$ An oriented Lagrangian submanifold $\Lambda\subset X$ has a \emph{phase function} $\theta_\Lambda : \Lambda \to S^1$ given by $\theta_\Lambda(p) = \theta_{T_p\Lambda}.$ The Lagrangian $\Lambda$ is positive if all its tangent spaces are positive.
\end{dfn}
	
	Let $\mathcal{O}$ be a Hamiltonian isotopy class of closed embedded positive Lagrangians in $X$ of type $L$, and let $\Lambda\in\mathcal{O}$. By Remark~\ref{remark: exact is Hamiltonian}, the tangent space $T_\Lambda\mathcal{O}$ consists of exact 1-forms on $\Lambda$. Positivity of $\Lambda$ yields the isomorphism~\eqref{equation: tangent space to O}, which in turn gives rise to the Riemannian metric $(\cdot,\cdot)$ defined in~\eqref{equation: the metric}. As shown in~\cite[Theorem~4.1]{solomon-curv}, the metric $(\cdot,\cdot)$ has a Levi-Civita connection.
	
	\begin{dfn}
		Let $\Lambda=(\Lambda_t)_{t\in[0,1]}$ be a smooth path in $\mathcal{O}$. A \emph{lifting} of $\Lambda$ is a smooth path of embeddings $\Psi_t:L\to X,\;t\in[0,1],$ such that $\Psi_t$ represents $\Lambda_t$. A lifting $(\Psi_t)$ is \emph{horizontal} if it satisfies
		\[
		i_{\deriv{t}\Psi_t}\real\Omega=0,\quad t\in[0,1].
		\]
	\end{dfn}

	It is shown in~\cite[Section~5.3]{solomon} that, for a path $\Lambda$ as above, every embedding representing $\Lambda_0$ extends uniquely to a horizontal lifting. This allows us to describe the Levi-Civita connection of $(\cdot,\cdot)$ as follows. Let $u_t,\;t\in[0,1],$ be a vector field along $\Lambda$. That is,
	\[
	u_t\in\vmv(\Lambda_t),\qquad t\in[0,1].
	\]
	Let $(\Psi_t)$ be a horizontal lifting of $\Lambda$. Then the covariant derivative of $u_t$  is given by
	\[
	\left(\frac{D}{dt}u_t\right)\circ\Psi_t=\deriv{t}(u_t\circ\Psi_t).
	\]
	Thus, the path $\Lambda$ is a \emph{geodesic} if and only if it admits a lifting $(\Psi_t)_t$ and a family of functions $h_t:\Lambda_t\to\R$ satisfying the equations
	\begin{equation}
	\label{equation: geodesic}
	i_{\deriv{t}\Psi_t}\omega = d(h_t\circ\Psi_t), \qquad i_{\deriv{t}\Psi_t} \real \Omega = 0, \qquad \deriv{t}(h_t \circ \Psi_t) = 0.
	\end{equation}
	This equivalent definition allows us to extend the notion of geodesics to a larger class of Lagrangian paths. Indeed, it makes sense for paths of immersed Lagrangians which are not necessarily embedded or closed. In the present article we also consider geodesics of non-smooth Lagrangians with singular locus consisting of finitely many cone points. We provide an explicit definition of geodesics in a more general form after setting up the necessary theory of submanifolds with cone points.
	
	\section{Lagrangians with cone points}
	\label{section: Lagrangians with Cone Points}
	
	For the purposes of this article it is natural to consider Lagrangians with \emph{cone points} and paths thereof. In Section~\ref{subsection: oriented blowups and cone-immersed submanifolds} below we discuss differentiability on general submanifolds with cone points, where no additional geometric structure is present. In Section~\ref{subsection: cone-immersed Lagrangians} we discuss spaces of positive Lagrangians with cone points in a Calabi-Yau manifold.

Section~\ref{subsection: oriented blowups and cone-immersed submanifolds} begins with a discussion of the oriented blowup of a manifold at a finite collection of points. A map is defined to be cone-smooth if its composition with the blowup projection is smooth. Such a map may not be differentiable at the blowup locus, but we define the cone-derivative as a substitute for the usual derivative, which is a $1$-homogeneous map instead of a linear map. The notion of cone-immersion is defined using the cone-derivative. Lemma~\ref{lemma: cone-immersions} shows that cone immersions lift locally to a map from the oriented blowup of the source to the oriented blowup of the target. This result plays an important role in the special case of cone diffeomorphisms, which are used in the definition of cone-immersed submanifolds.

In order to define cone-smooth differential forms, vector fields and Riemannian metrics, we make extensive use of the pull-back of the tangent bundle to the oriented blowup, which is called the blowup tangent bundle. Lemma~\ref{lemma: induced tangent frame} gives a convenient class of local frames for the blowup tangent bundle. These local frames are used in Lemma~\ref{lemma:blowupdifferential} to show that the differential of a cone-smooth map induces a map of the blowup tangent bundle to the pull-back of the tangent bundle of the target. The induced map is called the blowup differential, and plays a crucial role in the pull-back of differential forms and Riemannian metrics by cone-smooth maps and in the definition of cone-Hessian. Integration of cone-smooth differential forms is discussed. The local frames of Lemma~\ref{lemma: induced tangent frame} are used also to prove important properties of cone-smooth vector fields in Lemma~\ref{lemma:cone smooth vanishing tangent} and the cone-Hessian in Lemma~\ref{lemma: vanishing second derivative at extremum point}. A cone-smooth version of non-degenerate critical point is defined using the cone-Hessian and it is shown in Lemma~\ref{lemma: nice level sets} that level sets near a non-degenerate local extremum behave similarly to the smooth case.

Section~\ref{subsection: cone-immersed Lagrangians} begins by defining cone-immersed Lagrangian submanifolds and the positivity thereof. The Riemannian metric~\eqref{equation: the metric} is extended to positive cone-immersed Lagrangians. Lemma~\ref{lemma: path of cone-Lags admits horizontal liftings} gives the existence of a horizontal lifting of a path of cone-immersed Lagrangian submanifolds. The horizontal lifting is a canonical path of Lagrangian cone-immersions representing the cone-immersed Lagrangian submanifolds determined uniquely by a representative of any one of them. The path is a geodesic if the composition of the Hamiltonian with the cone-immersions of a horizontal lifting is time independent. Section~\ref{subsection: cone-immersed Lagrangians} makes extensive use of the tools developed in Section~\ref{subsection: oriented blowups and cone-immersed submanifolds}.
	
	\subsection{Oriented blowups and cone-smooth differential topology}
	\label{subsection: oriented blowups and cone-immersed submanifolds}
	
	\begin{dfn}
		\label{definition: blowup}
		$\;$
		\begin{enumerate}
			\item Let $n\in\mathbb{N}$ and let $V$ be a real vector space of dimension $n.$ For $0\ne v\in V,$ the \emph{ray} spanned by $v$ is the subset $\{\lambda v\;|\;\lambda\ge0\}\subset V.$ The \emph{oriented projective space} $\orientedprojective(V)$ is defined to be the set of rays in $V.$ As a set, $\orientedprojective(V)$ is naturally identified with the quotient $(V\setminus\{0\})/\mathbb{R}^+.$ We equip $\orientedprojective(V)$ with the smooth structure making this identification a diffeomorphism. In particular, $\orientedprojective(V)$ is diffeomorphic to the sphere $S^{n-1}.$ The \emph{oriented blowup} of $V$ is defined by
			\[
			\widetilde{V}:=\left\{\left.(q,r)\in V\times\orientedprojective(V)\;\right|\;q\in r\right\}.
			\]
			The \emph{blowup projection} $\pi:\widetilde{V}\to V$ is given by $(q,r)\mapsto q.$ We call $E=\pi^{-1}(0)$ the \emph{exceptional sphere}. Then $\widetilde{V}$ is a smooth manifold with boundary $E.$ The restricted blowup projection $\pi|_{\widetilde{V}\setminus E}:\widetilde{V}\setminus E\to V\setminus\{0\}$ is a diffeomorphism.
			\item \label{blowup at a single point} Let $M$ be an $n$-dimensional smooth manifold and let $p\in M.$ Let $p\in U\subset M$ be open with a diffeomorphism $\varphi:U\to\mathbb{R}^n$ carrying $p$ to $0.$ Let $\widetilde{U}$ denote the oriented blowup of $U$ with respect to the vector space structure induced by $\varphi$, let $\pi_U$ denote the blowup projection and let $E$ denote the exceptional sphere. We define the \emph{oriented blowup of $M$ at $p$} by gluing $\widetilde{U}$ and $M\setminus\{p\}$ along $\pi_U:$
			\[
			\widetilde{M}_p:=\left(\widetilde{U}\cup (M\setminus\{p\})\right)/(q,r)\sim q,\;(q,r)\in\widetilde{U}\setminus E.
			\]
			The projection $\pi_U$ then extends to the global blowup projection $\pi:\widetilde{M}_p\to M.$ We call $E_p=\pi^{-1}(p)$ the \emph{exceptional sphere over $p$.} One verifies that everything defined here is independent of $U$ and $\varphi,$ and the exceptional sphere $E_p$ is naturally identified with $\orientedprojective(T_pM).$ Once again, the blowup projection $\pi$ restricts to a diffeomorphism
			\[
			\pi|_{\widetilde{M}_p\setminus E}:\widetilde{M}_p\setminus E\to M\setminus\{p\}.
			\]
			\item\label{blowup at discrete set} Let $M$ be as in~\eqref{blowup at a single point}, and let $S\subset M$ be finite. For $p\in S,$ let $\widetilde{M}_p$ denote the oriented blowup of $M$ at $p.$ The \emph{oriented blowup of $M$ at $S$,} denoted by $\widetilde{M}_S,$ is obtained by gluing all the oriented blowups $\widetilde{M}_p\setminus\{S\setminus\{p\}\},\;p\in S,$ in the obvious manner. For $p\in S,$ the exceptional sphere over $p$ is naturally identified with $\orientedprojective(T_pM).$ The oriented blowup $\widetilde{M}_S$ is a smooth manifold with boundary equal to the disjoint union of the exceptional spheres. It comes with a blowup projection $\pi : \widetilde M_S \to M$ as before. We write \[
\widetilde{M}_S^\circ = \widetilde{M}_S\setminus\partial\widetilde{M}_S
\]
for the \emph{interior} of $\widetilde{M}_S.$
\item\label{cone coordinates}
Let $M$ and $S$ be as in \eqref{blowup at discrete set}. Let $p \in S$ and let $\widetilde{p} \in E_p.$ \emph{Cone coordinates} at $\widetilde{p}$ are a triple $(U,\mathbf{X},\alpha)$ where $\widetilde{p} \in U \subset E_p$ is open,
\[
\mathbf{X} = (x^1,\ldots,x^m): U \to \R^m
\]
are local coordinates, and $\alpha : U \times [0,\epsilon) \to \widetilde{M}_S$ is a smooth open embedding such that $\alpha(\widetilde q,0) = \widetilde q$ for $\widetilde q \in U.$ We generally denote the coordinate on $[0,\epsilon)$ by $s.$ Given a map $f : M \to N,$ we abbreviate
\[
\qquad\qquad f\circ \pi \circ \alpha(\widetilde q,s) = f(\widetilde q,s) = f(x_1,\ldots,x_m,s), \qquad  (x_1,\ldots,x_m) = \mathbf{X}(\widetilde q).
\]
		\end{enumerate}
	\end{dfn}

\begin{rem}\label{remark:polar coordinates}
Let $V$ be an $n$-dimensional real vector space. Given a smooth section~$\sigma$ of the quotient map $V\setminus\{0\} \to \orientedprojective(V),$ we identify
\[
\orientedprojective(V) \times [0,\infty) \overset{\sim}{\longrightarrow} \widetilde{V}
\]
by $(r,s) \mapsto (s \sigma(r),r)$. Since $\orientedprojective(V) \cong S^{n-1},$ we identify $S^{n-1} \times [0,\infty) \cong \widetilde{V}.$
\end{rem}

\begin{dfn}\label{definition:polar coordinates}
Let $M$ be a smooth manifold, let $p \in M.$ Let $\sigma$ be a smooth section of the quotient map $T_pM\setminus\{0\} \to \orientedprojective(T_pM) \cong S^{n-1}.$ A \emph{polar coordinate map} centered at $p$ associated with $\sigma$ is a smooth map
\[
\kappa : S^{n-1} \times [0,\epsilon) \to M
\]
satisfying the following.
\begin{enumerate}
\item
$\kappa|_{S^{n-1}\times\{0\}}$ is the constant map to $p.$
\item
$\kappa|_{S^{n-1}\times (0,\epsilon)}$ is an open embedding.
\item\label{item:sigma}
For $r \in S^{n-1},$
\[
\pderiv[\kappa]{s}(r,0) = \sigma(r).
\]
\end{enumerate}
An elementary argument shows that for every $\sigma$ there exist many polar coordinate maps. We may sometimes speak of a polar-coordinate map $\kappa$ without mentioning $\sigma$ explicitly. In this case, we refer to the section $\sigma$ determined by condition~(\ref{item:sigma}) as the section associated with $\kappa.$
\end{dfn}

	The oriented blowup gives rise to the following notions associated with differentiability, which are weaker than the usual ones.
	
	\begin{dfn}
		\label{definition: cone differentiability}
		Let $M$ and $N$ be smooth manifolds, let $p\in M,$ and let $\Psi:M\to N$ be continuous. Let $\widetilde{M}_p,E_p$ and $\pi:\widetilde{M}_p\to M$ be as in Definition~\ref{definition: blowup}.
		\begin{enumerate}
			\item The map $\Psi$ is said to be \emph{cone-smooth} at $p$ if there exists an open $E_p\subset\widetilde{U}\subset\widetilde{M}_p$ such that the composition $\Psi\circ\pi|_{\widetilde{U}}:\widetilde{U}\to N$ is smooth.
			\item Suppose $\Psi$ is cone-smooth at $p.$ The \emph{cone-derivative} of $\Psi$ at $p$ is the unique map
			\[
			d\Psi_p:T_pM\to T_{\Psi(p)}N
			\]
			satisfying the equality
			\[
			d(\Psi\circ\pi)_{\widetilde{p}}=d\Psi_p\circ d\pi_{\widetilde{p}}
			\]
			for $\widetilde{p}\in E_p.$ One verifies that the cone-derivative is well-defined and homogeneous of degree 1. Also, the restricted map $d\Psi_p|_{T_pM\setminus\{0\}}$ is smooth. Nevertheless, the cone-derivative is not linear in general.
			\item Suppose $\Psi$ is cone-smooth at $p.$ We say $\Psi$ is \emph{cone-immersive} at $p$ if the restricted map
			\[
			d\Psi_p|_{T_pM\setminus\{0\}}:T_pM\setminus\{0\}\to T_{\Psi(p)}N
			\]
			is a smooth immersion. In particular, in this case we have $d\Psi_p(v)\ne0$ for $0\ne v\in T_pM.$
		\end{enumerate}
	\end{dfn}

\begin{rem}
Let $\Psi : M \to N$ be cone-smooth at $p \in M,$ and let $\kappa : S^{n-1}\times [0,\epsilon) \to M$ be a polar coordinate map centered at $p.$ It follows from Remark~\ref{remark:polar coordinates} and the definition of the oriented blowup that, possibly after shrinking $\epsilon,$ the composition $\Psi \circ \kappa$ is smooth.
\end{rem}
	
	Recall that the \emph{Euler vector field} on a real vector space is the radial vector field that integrates to rescaling by $e^t.$ We let $\varepsilon$ denote the Euler vector field.
\begin{lemma}\label{lemma:immersion criterion}
Let $\Psi : M \to N$ be cone-smooth at $p$ and let $\kappa: S^{n-1}\times [0,\epsilon) \to M$ be a polar-coordinate map at $p.$ Then $\Psi$ is cone-immersive at $p$ if and only if
\[
\left.\pderiv{s}(\Psi \circ \kappa)\right|_{s = 0} : S^{n-1} \to T_{\Psi(p)}N
\]
is an immersion nowhere tangent to the Euler vector field.
\end{lemma}
\begin{proof}
Let $\sigma$ be the section associated with $\kappa.$ Let $r \in S^{n-1}.$ Observe that
\[
\pderiv{s}(\Psi \circ \kappa)(r,0) = d\Psi_p(\sigma(r)).
\]
So, by the chain rule,
\[
d(d\Psi_p)_{\sigma(r)}\circ d\sigma_r = d\left(\left.\pderiv{s}(\Psi \circ \kappa)\right|_{s=0}\right)_r.
\]
Also,
\[
d(d\Psi_p)_{\sigma(r)}(\varepsilon(\sigma(r))) = \varepsilon(d\Psi_p(\sigma(r))).
\]
Finally, observe that
\[
T_{\sigma(r)}T_pM  = \R\langle\varepsilon(\sigma(r))\rangle \oplus d\sigma_r\left(T_rS^{n-1}\right).
\]
The lemma follows.
\end{proof}

	\begin{lemma}
		\label{lemma: cone-immersions}
		In the setting of Definition~\ref{definition: cone differentiability}, suppose $\Psi$ is cone-immersive at $p.$
		\begin{enumerate}[label=(\alph*)]
			\item \label{sort of locally one-to-one} There exists an open neighborhood $p\in V\subset M$ such that, for $p\ne q\in V,$ we have $\Psi(q)\ne\Psi(p).$
			\item \label{existence of lifting} Let $p\in V\subset M$ as in \ref{sort of locally one-to-one} and let $\widetilde{V}:=\pi^{-1}(V)\subset\widetilde{M}_p.$ Then there exists a unique continuous $\widetilde{\Psi}:\widetilde{V}\to\widetilde{N}_{\Psi(p)}$ satisfying
			\begin{equation}
			\label{equality: strict transform}
			\pi_N\circ\widetilde{\Psi}=\Psi\circ\pi|_{\widetilde{V}}\;,
			\end{equation}
			where $\widetilde{N}_{\Psi(p)}$ denotes the oriented blowup of $N$ at $\Psi(p)$ and $\pi_N$ is the associated blowup projection. Moreover, for some open subset $E_p\subset \widetilde{V}'\subset\widetilde{V},$ the restricted map $\widetilde{\Psi}|_{\widetilde{V}'}$ is a smooth immersion carrying $E_p$ into $E_{\Psi(p)}$ and satisfying, for $\widetilde{p}\in E_p$ and $v\in T_{\widetilde{p}}\widetilde{M}_p,$
			\begin{equation}\label{implication:notangent}
			d\widetilde{\Psi}_{\widetilde{p}}(v)\in T_{\widetilde{\Psi}\left(\widetilde{p}\right)}E_{\Psi(p)}\Leftrightarrow v\in T_{\widetilde{p}}E_p.
			\end{equation}
		\end{enumerate}
	\end{lemma}

	The proof of Lemma~\ref{lemma: cone-immersions} relies on the following lemma (compare with~\cite[Lemma~2.1]{milnor}), which is also used extensively below.
	
	\begin{lemma}
		\label{lemma: milnor lemma}
		Let $f:\mathbb{R}^k\times\mathbb{R}\to\mathbb{R}$ be smooth with $f(x,0)=0,\;x\in\mathbb{R}^k.$ Then we have $f(x,s)=s\cdot g(x,s)$ for $(x,s)\in\mathbb{R}^k\times\mathbb{R},$ where $g:\mathbb{R}^k\times\mathbb{R}\to\mathbb{R}$ is smooth and satisfies
		\[
		g(x,0)=\pderiv[f]{s}(x,0),\;\pderiv[g]{s}(x,0)=\frac{1}{2}\frac{\partial^2f}{\partial s^2}(x,0),\quad x\in\mathbb{R}^k.
		\]
	\end{lemma}

	\begin{proof}[Proof of Lemma~\ref{lemma: cone-immersions}]
		As the lemma is local, we may assume
\[
M=\mathbb{R}^{m+1}, \qquad N=\mathbb{R}^{n+1}, \qquad p = 0, \qquad \Psi(p) = 0.
\]
Keeping in mind Remark~\ref{remark:polar coordinates}, we have
		\[
		\Psi\circ\pi:S^m\times[0,\infty)\to\mathbb{R}^{n+1},\quad\Psi\circ\pi|_{S^m\times\{0\}}=0.
		\]
		Let $s$ denote the $[0,\infty)-$coordinate. Since $\Psi$ is cone-immersive, for $r \in S^m,$ the one-sided directional derivative $\pderiv[(\Psi\circ\pi)]{s}|_{(r,0)}$ is non-vanishing. So, there exist $\epsilon~>~0,$ an open half space $H \subset \R^{n+1},$ and a neighborhood $r \in U \subset S^m,$ such that $\pderiv[(\Psi\circ\pi)]{s}|_{(r',s')} \in H$ for $(r',s') \in U  \times [0,\epsilon).$ By the mean-value theorem,
$\Psi\circ\pi(r',s') \neq 0$
for $(r', s') \in U\times(0,\epsilon).$  Part \ref{sort of locally one-to-one} now follows from compactness of~$S^m.$
		
By definition of cone-immersion, the cone differential $d\Psi_p : T_pM \setminus \{0\} \to T_{\Psi(p)}N$ is an immersion, which necessarily takes no non-zero vector to zero. Thus, the oriented projectivization
\[
\orientedprojective(d\Psi_p):\orientedprojective\left(T_pM\right)\to \orientedprojective\left(T_{\Psi(p)}N\right)
\]
is a well-defined smooth immersion. The equality~\eqref{equality: strict transform} determines the desired $\widetilde{\Psi}$ of part \ref{existence of lifting} uniquely away from the exceptional sphere $E_p.$ Since the complement of $E_p$ is dense, a continuous extension to $E_p$ is unique if it exists. For $c \in E_p,$ we set
		\[
		\widetilde{\Psi}(c):= \orientedprojective(d\Psi_p)(c)  \in E_{\Psi(p)} = \orientedprojective(T_{\Psi(p)}N),
		\]
and prove that $\widetilde{\Psi}$ is smooth as follows. Let $\tilde p \in E_p$ and let $(U,\mathbf{X},\alpha)$ be cone coordinates at $\widetilde{p}$ as in Definition~\ref{definition: blowup}(\ref{cone coordinates}).
Then
\[
\Psi \circ \pi \circ \alpha : U \times [0,\epsilon) \to \R^{n+1} = N
\]
vanishes on $U \times \{0\}.$ So, by Lemma~\ref{lemma: milnor lemma}, there exists $g : U \times [0,\epsilon) \to \R^{n+1}$ such that
\[
\Psi \circ \pi \circ \alpha(u,s) = s \cdot g(u,s), \qquad u \in U, \quad s \in [0,\epsilon),
\]
and
\[
g(u,0) = \frac{\partial (\Psi\circ\pi \circ \alpha)}{\partial s}(u,0).
\]
For $v \in T_pM,$ we write $[v] \in \orientedprojective(T_pM) \simeq  E_p$ for the corresponding point in the oriented projectification. For $u = [v] \in U,$
it follows from the definition of the blowup and cone coordinates that there exists $\sigma = \sigma(u) \in \R$ such that $\frac{\partial (\pi \circ \alpha)}{\partial s}(u, 0) = \sigma v.$ So,
\[
g(u,0) = d \Psi_p \left(\frac{\partial \pi \circ \alpha}{\partial s}(u,0)\right) = \sigma d\Psi_p(v),
\]
and consequently
\[
[g(u,0)] = \orientedprojective(d\Psi_p)(u).
\]
Thus,
\[
\widetilde \Psi(u,s) = (s \cdot g(u,s),[g(u,s)]) \in \widetilde N_{\Psi(p)} = \left\{\left.(q,r)\in \R^{n+1}\times\orientedprojective(\R^{n+1})\;\right|\;q\in r\right\},
\]
so $\widetilde \Psi$ is smooth on the image of $\alpha.$ Since $\tilde p \in E_p$ was arbitrary, it follows that $\widetilde \Psi$ is smooth in a neighborhood of $E_p.$ Away from $E_p,$ the projection $\pi$ is a local diffeomorphism, so smoothness of $\widetilde \Psi$ follows from~\eqref{equality: strict transform}.

We show that the derivative $d\widetilde{\Psi}_{\tilde p}$ is one-to-one for all $\tilde p\in E_p.$ Indeed, since $\widetilde \Psi|_{E_p} = \orientedprojective(d\Psi_p)$ is an immersion,
\begin{equation}\label{eq:einj}
d\widetilde{\Psi}_{\tilde p}|_{T_{\tilde p} E_p}: T_{\tilde p}E_p \to T_{\widetilde{\Psi}(\tilde p)} E_{\Psi(p)}
\end{equation}
is injective. On the other hand, if $v \in T_pM$ is such that $\tilde p = [v],$ then
\[
d\pi : T_{\tilde p} \widetilde M_p/T_{\tilde p}E_p \to \R \cdot v \subset T_pM
\]
is an isomorphism. Additionally, if $w \in T_{\Psi(p)}N$ is such that
\[
[w] = \widetilde{\Psi}(\tilde p) = \orientedprojective(d\Psi_p)([v]),
\]
then
\[
d\Psi_p : \R \cdot v \to \R \cdot w \subset T_{\Psi(p)}N
\]
is an isomorphism. It follows from~\eqref{equality: strict transform} that the map
\begin{equation}\label{eq:qinj}
T_{\tilde p} \widetilde M_p/T_{\tilde p}E_p \to T_{\widetilde{\Psi}(\tilde p)} \widetilde N_{\Psi(p)}/T_{\widetilde{\Psi}(\tilde p)} E_{\Psi(p)}
\end{equation}
induced by $d\widetilde{\Psi}_{\tilde p}$ is an isomorphism. Combining the injectivity of~\eqref{eq:einj} and~\eqref{eq:qinj}, we deduce that
$d\widetilde{\Psi}_{\tilde p}$ is injective. Since injectivity of the differential is an open property, we can find $E_p \subset \widetilde{V}' \subset \widetilde V$ such that $\widetilde \Psi|_{\widetilde{V}'}$ is an immersion as claimed in~\ref{existence of lifting}. Finally,~\eqref{implication:notangent} follows from the injectivity of~\eqref{eq:qinj}.
	\end{proof}

	\begin{dfn}
		\label{definition: strict transform}
		The map $\widetilde{\Psi}$ of Lemma~\ref{lemma: cone-immersions} \ref{existence of lifting} is called the \emph{strict transform} of $\Psi|_V.$
	\end{dfn}
	
	\begin{dfn}\label{definition:cone immersed submanifold}
		Let $M$ and $N$ be smooth manifolds. Let $S\subset M$ be a finite subset and $\Psi:M\to N$ a continuous map.
		\begin{enumerate}[label=(\alph*)]
        \item
    We say the pair $(\Psi,S)$ is \emph{cone-smooth} if $\Psi$ is smooth away from $S$ and cone-smooth at every element of $S.$
			\item We say the pair $(\Psi, S)$ is a \emph{cone-immersion} from $(M,S)$ to $N$ if $\Psi$ is a smooth immersion away from $S$ and cone-immersive at every element of $S.$
    \item
    Suppose $N = M$ and $\Psi(p) = p$ for all $p \in S.$ The pair $(\Psi,S)$ is a \emph{cone-diffeomorphism} of $(M,S)$ if $\Psi$ is a smooth diffeomorphism away from $S$ and for $p \in S,$ the cone derivative $d\Psi_p|_{T_pM\setminus \{0\}} : T_pM\setminus \{0\} \to T_pM$ is a diffeomorphism onto $T_pM \setminus \{0\}.$ We let $\diff(M,S)$ denote the group of cone-diffeomorphisms of $(M,S)$ that act trivially on the set of connected components.
			\item Let the diffeomorphism group $\diff(M,S)$ act on cone-immersions from $(M,S)$ to $N$ by composition. A \emph{cone-immersed submanifold of $N$ of type $(M,S)$} is an orbit of the $\diff(M,S)$-action.
    \item\label{item:orientation} Suppose $M$ is orientable and let $\diff^+(M,S)\triangleleft \diff(M,S)$ denote the normal subgroup of orientation preserving cone-smooth diffeomorphisms. An \emph{orientation} on a cone-immersed submanifold $K$ of $N$ of type $(M,S)$ is an equivalence class of pairs $(O,C)$ where $O$ is an orientation of $M$ and $C$ is a $\diff^+(M,S)$ orbit inside the $\diff(M,S)$ orbit $K.$ There is a natural $\diff(M,S)/\diff^+(M,S)$ action on such pairs and this gives rise to the desired equivalence relation.
		\end{enumerate}
	\end{dfn}

\begin{rem}\label{remark: cone-diffeomorphism lifting}
It follows from Lemma~\ref{lemma: cone-immersions} that a map $\Psi : (M,S) \to (M,S)$ is a cone-diffeomorphism if and only if it lifts to a diffeomorphism $\widetilde\Psi : \widetilde{M}_S \to \widetilde{M}_S$ such that $\widetilde{\Psi}(E_p) = E_p$ for $p \in S.$
\end{rem}

\begin{rem}
		\label{remark: embedded cone locus}
		In this article, all cone-immersions are assumed to have embedded cone locus. That is, for a cone-immersion $(\Psi,S)$ and points $c\ne c'\in S$ we have $\Psi(c)\ne\Psi(c').$
	\end{rem}

\begin{cor}\label{cor:cone-immersion open mapping}
Let $M,N,$ be smooth manifolds of equal dimension, let $S \subset M$ be finite and let $(\Psi : M \to N,S)$ be a cone-immersion. Then $\Psi$ is an open map.
\end{cor}
\begin{proof}
It suffices to prove that for $p \in S$ and every open $p \in V \subset M$ as in Lemma~\ref{lemma: cone-immersions}~\ref{sort of locally one-to-one}, the image $\Psi(V)$ contains an open neighborhood of $\Psi(p).$ Indeed, consider the strict transform $\widetilde \Psi : \widetilde V \to \widetilde N_{\Psi(p)}$ as in Lemma~\ref{lemma: cone-immersions}~\ref{existence of lifting}. Then, $\widetilde \Psi(E_p) \subset E_{\Psi_p}$ is open and closed, so $\widetilde \Psi(E_p) = E_{\Psi_p}.$ By the inverse function theorem, there is an open $E_p \subset U$ with $E_{\Psi(p)} \subset \widetilde \Psi(U)$ and $\widetilde \Psi(U)$ open. So, the image under the blowup projection $\pi_N\left(\widetilde \Psi(U)\right)$ is open, which implies the claim.
\end{proof}

\begin{dfn}
	Let $K=[(\Psi:M\to N,S)]$ be a cone-immersed submanifold of type $(M,S).$ As in Definition~\ref{definition:points}, a \emph{point} $p$ in $K$ is an equivalence class of pairs $((\chi,S),q),$ where $(\chi: M \to N,S)$ is a representative of $K$ and $q\in M.$ We let $p_0$ denote the image of $p$ in $N.$ That is, $p_0 = \chi(q)$ for $((\chi,S),q)$ a representative of $p.$ The cone-immersed submanifold $K$ thus has a well-defined \emph{cone locus}
	\[
	K^C:=\{[((\Psi,S),c)]\;|\;c\in S\}.
	\]
A \emph{cone point} is an element of the cone locus. We define the \emph{tangent cone} of $K$ at a point $p = [((\Psi,S),q)]$ to be the cone-immersed submanifold
\[
TC_pK := [(d\Psi_c : T_cM \to T_{p_0}N,\{0\})]
\]
of $T_{p_0}N.$ The tangent cone $TC_pK$ is indeed a cone, that is, invariant under scalar multiplication. Moreover, it is independent of the choice of $\Psi.$ If $\Psi$ is smooth at $q,$ then $TC_pK$ is smoothly embedded and recovers the usual notion of tangent space. We define the \emph{projective tangent cone} of $K$ at $p$ by
	\[
	\orientedprojective(TC_pK):=\left[\orientedprojective(d\Psi_c) :\orientedprojective(T_cM)\to \orientedprojective(T_{p_0}N)\right].
	\]
	This is a smooth immersed sphere in $\orientedprojective(T_{p_0}N).$

A \emph{function} on $K$ is an equivalence class of pairs $((\chi,S),f),$ where $(\chi,S)$ is a representative of $K$ and $f$ is a function on $M.$ We say the function $h=[((\Psi,S),f)]$ is \emph{cone-smooth} at the point $p=[((\Psi,S),q)]$ if $f$ is cone-smooth at $q.$ In this case $h$ has a well-defined \emph{cone-derivative} $dh_p : TC_pK \to \R,$ which is a degree-1 homogeneous function.
\end{dfn}

	\begin{lemma}
		\label{lemma: induced tangent frame}
Let $M,S,\pi : \widetilde M_S \to M,p$ and $\tilde p$ be as in Definition~\ref{definition: blowup}(\ref{blowup at discrete set}) and~(\ref{cone coordinates}).
Let $(U,\mathbf{X},\alpha)$ be cone coordinates at $\widetilde{p}$ and abbreviate $V = \mathrm{Im}(\alpha) \subset \widetilde M_S.$ Define sections $e_i$ of the vector bundle $\pi^*TM|_V$ by
		\[
		e_i(\tilde q, s) :=
		\left\{
		\begin{array}{ll}
		\frac{1}{s}\pderiv[\pi]{x^i}(\tilde q, s),
		&s\neq 0,\\
		\nabla_s \pderiv[\pi]{x^i}(\tilde q, s),
		&s =0,
		\end{array}
		\right.\quad i=1,\ldots,m,
		\]
where $\nabla$ is an arbitrary connection.
		Define $e_0:= \pderiv[\pi]{s}.$ Then $e_0,\ldots,e_m,$ are independent of the choice of $\nabla,$ smooth, and everywhere linearly independent.
	\end{lemma}
	\begin{proof}
		As the blowup projection $\pi$ maps the exceptional sphere $E_{p}$ to the point $p,$ we have
		\[
		 \pderiv[\pi]{x^i}(\tilde q, 0) = 0, \qquad \tilde q \in U, \quad i=1,\ldots,m.
		\]
		The sections $e_i,\;i=1,\ldots,m,$ are thus well-defined on $U$ independently of the choice of connection $\nabla.$ They are smooth by Lemma~\ref{lemma: milnor lemma}. One shows the sections $e_0,\ldots,e_m,$ are everywhere linearly independent by using the definition of the blowup projection $\pi.$
	\end{proof}

Let $(\Psi : M \to N, S)$ be a cone-smooth map and let $p \in S.$ Recall that the cone derivative $d\Psi_p : T_pM \to T_{\Psi(p)}N$ is homogeneous of degree $1$ and $d\Psi_p|_{T_pM \setminus\{0\}}$ is smooth. It follows that for $0 \neq v \in T_pM$ and $\lambda > 0,$ we have $d(d\Psi_p)_{v} = d(d\Psi_p)_{\lambda v}$ under the canonical identification $T_v T_pM \simeq T_pM \simeq T_{\lambda v}T_pM.$ Thus, for $\widetilde p = [v] \in \orientedprojective(T_pM),$ we define
\[
d(d\Psi_p)_{\widetilde p} : = d(d\Psi_p)_v : T_pM \to T_{\Psi(p)}N.
\]
\begin{lemma}\label{lemma:blowupdifferential}
Let $(\Psi : M \to N, S)$ be a cone-smooth map. Consider the map $d\left(\Psi|_{M\setminus S}\right) : TM|_{M \setminus S} \to \Psi^* TN|_{M \setminus S}.$ Pulling back by $\pi$ gives a map
\[
\pi^* d\left(\Psi|_{M\setminus S}\right) : \pi^* TM|_{\widetilde M_S^\circ} \to \pi^* \Psi^* TN|_{\widetilde M_S^\circ}.
\]
This map extends uniquely to a map of bundles $\widetilde{d\Psi}:\pi^* TM \to \pi^* \Psi^* TN.$ Moreover, for $p \in S$ and $\widetilde p \in E_p,$ we have
\begin{equation}\label{equation:dderivatives}
\widetilde{d\Psi}_{\widetilde p} = d(d\Psi_p)_{\widetilde p}.
\end{equation}
In particular, if $\Psi$ is a cone-immersion, then $\widetilde{d\Psi}$ is an injective map of vector bundles.
\end{lemma}
\begin{proof}
By Lemma~\ref{lemma: induced tangent frame} and using the notation therein, it suffices for the first claim to show that the sections $\pi^* d\Psi(e_i|_{\{s > 0\}})$ extend to smooth sections of $\pi^*\Psi^* TN$ for $i = 0,\ldots,m$ and any choice of cone coordinates. Linearity of $\widetilde{d\Psi}$ follows by continuity. Consider the differential of the blowup projection $d\pi|_{\widetilde M_S\setminus E_S} : T\widetilde M_S|_{\widetilde M_S^\circ} \to \pi^* TM|_{\widetilde M_S^\circ}.$ We have
\[
\pi^* d\left(\Psi|_{M\setminus S}\right)  = \pi^* d\left(\Psi|_{M\setminus S}\right) \circ d\pi \circ d \pi^{-1}|_{\widetilde M_S^\circ}  = d(\Psi \circ \pi) \circ d \pi^{-1}|_{\widetilde M_S^\circ} .
\]
By definition of cone-smooth, $d(\Psi \circ \pi) : T\widetilde{M}_S \to \Psi^* \pi^* TN$ is a smooth bundle map. For $p \in S,$ since $\Psi\circ \pi$ maps $E_p$ to a point, it follows that $d(\Psi \circ \pi)|_{\{s = 0\}} = 0.$ On the other hand, $d \pi^{-1}(e_i|_{\{s > 0\}}) = \frac{1}{s} \pderiv{x^i}$ for $i = 1,\ldots, m,$ and $d \pi^{-1}(e_0|_{\{s > 0\}}) = \pderiv{s}.$ So, it follows from Lemma~\ref{lemma: milnor lemma} that $\pi^*\left( d\Psi(e_i|_{\{s>0\}})\right)$ extends to a smooth section of $\pi^*\Psi^* TN,$ which is $\widetilde{d\Psi}(e_i).$  Thus, the map $\widetilde{d\Psi}$ is well-defined.

To prove equation~\eqref{equation:dderivatives}, we claim that for $\widetilde q \in U,$
\begin{gather*}
\widetilde{d\Psi}(e_i)(\widetilde q,0) = \frac{\partial^2\Psi}{\partial x^i \partial s}(\widetilde q, 0), \qquad i = 1,\ldots, m, \\
\widetilde{d\Psi}(e_0)(\widetilde q,0) = \pderiv[\Psi]{s} (\widetilde q, 0).
\end{gather*}
Indeed,
\begin{equation*}
\frac{\partial^2\Psi}{\partial x^i \partial s}(\widetilde q, 0) = \frac{\partial^2\Psi}{\partial s \partial x^i}(\widetilde q, 0) =\pderiv{s}\left(\widetilde{d\Psi}\left(\pderiv[\pi]{x^i}\right)\right)(\widetilde q, 0) = \widetilde{d\Psi}(e_i)(\widetilde q, 0).
\end{equation*}
Similarly,
\begin{equation*}
\pderiv[\Psi]{s}(\widetilde q, 0) =\widetilde{d\Psi}\left(\pderiv[\pi]{s}\right)(\widetilde q, 0) = \widetilde{d\Psi}(e_0)(\widetilde q, 0).
\end{equation*}
On the other hand, writing $v(\widetilde q) = e_0(\widetilde q,0) \in T_pM,$ recalling the definition of the cone derivative $d\Psi_p,$ and identifying $T_vT_pM \simeq T_pM,$ we have
\begin{gather*}
\frac{\partial^2\Psi}{\partial x^i \partial s} (\widetilde q, 0) =  \pderiv{x^i}d\Psi_p(v(\widetilde q)) = d(d\Psi_p)_{v(\widetilde q)}(e_i(\widetilde q,0)), \quad i = 1,\ldots, m,\\
\pderiv[\Psi]{s}(\widetilde q, 0) = d\Psi_p(v) = d(d\Psi_p)_{v(\widetilde q)}(e_0(\widetilde q,0)).
\end{gather*}
Equation~\eqref{equation:dderivatives} follows.
\end{proof}

\begin{dfn}\label{definition: many tangent spaces}
Let $S \subset M$ be a finite subset. The \emph{blowup tangent bundle} of $(M,S)$ is the bundle $\widetilde{TM}_S : = \pi^*TM \to \widetilde{M}_S.$ When clear from the context, we may omit the subscript $S.$ Let $(\Psi : M \to N, S)$ be a cone-smooth map. The \emph{blowup differential} of $\Psi$ is the map
\[
\widetilde{d \Psi} : \widetilde{TM}_S \to \pi^* \Psi^* TN
\]
given by Lemma~\ref{lemma:blowupdifferential}.
	\end{dfn}

Let $M$ be a smooth manifold and let $S \subset M$ be a finite subset and let $\pi : \widetilde{M}_S \to M$ denote the blowup projection. Given a differential form $\alpha$ on $M$ we can pull-back $\alpha$ as a section of $\Lambda^*(T^*M)$ to obtain a section of $\pi^*\Lambda^*\left(T^*M\right) \simeq \Lambda^*\left(\widetilde{TM}^*_S\right).$ We denote this pull-back by $\pi^{-1}\alpha.$ Observe that this pull-back is different from the pull-back of $\alpha$ as a differential form, $\pi^*\alpha,$ which would be a section of $\Lambda^*(T^*\widetilde{M}_S).$ A similar distinction applies to pull-backs of metrics.
\begin{dfn}\label{definition: cone-smooth differential form etc}
A \emph{cone-smooth differential form} on $(M,S)$ is a smooth differential form $\alpha$ on $M \setminus S$ such that $\pi^{-1} \alpha$ extends to a smooth section $\widetilde{\alpha}$ of $\Lambda^*\left(\widetilde{TM}_S^*\right).$ We call $\widetilde{\alpha}$ the \emph{blowup form}. We say that $\alpha$ and $\widetilde{\alpha}$ are \emph{closed} if $\alpha$ is closed as a differential form on $M\setminus S.$

A \emph{cone-smooth Riemannian metric} on $(M,S)$ is a smooth Riemannian metric $g$ on $M \setminus S$ such that $\pi^{-1}g$ extends to a smooth metric $\widetilde g$ on the blowup tangent bundle $\widetilde{TM}_S.$ We call $\widetilde g$ the \emph{blowup metric}.

Let $(\Psi:M \to N,S)$ be a cone-smooth map. A \emph{cone-smooth vector field along} $\Psi$ is a smooth section $\xi$ of $\left(\Psi|_{M\setminus S}\right)^*TN$ such that $\left(\pi|_{\widetilde{M}_S^\circ}\right)^*\xi$ extends to a smooth section $\widetilde \xi$ of $(\Psi \circ \pi)^*TN.$ We call $\widetilde \xi$ the \emph{blowup vector field.} In the case $N = M$ and $\Psi = \id_M$, we call $\xi$ a cone-smooth vector field on $(M,S).$
\end{dfn}

\begin{rem}\label{remark:differential of cone-smooth function}
Let $S \subset M$ be a finite subset and let $(f : M \to \R,S)$ be a cone-smooth function. Then, the blowup differential $\widetilde{df}$ is a smooth section of the blowup cotangent bundle $\widetilde{TM}_S^* \to \widetilde{M}_S,$ so $df$ is a cone-smooth 1-form. For $p \in S$ and $\widetilde p \in E_p,$ we use the notation $df_{\widetilde p} = \widetilde{df}|_{\widetilde p}.$ Given a cone-smooth Riemannian metric on $(M,S),$ we can define the cone-smooth vector field $\nabla f$ on $(M,S)$ in the usual way.
\end{rem}

\begin{rem}\label{remark:cone-smooth de Rham}
By Remark~\ref{remark:differential of cone-smooth function} the differential of a cone-smooth function is a cone-smooth $1$-form. However, a cone-smooth $0$-form need not be a cone-smooth function as it may not extend continuously to $M.$ Furthermore, the differential of a cone-smooth $0$-form need not be a cone-smooth $1$-form. More generally, the exterior derivative of a cone-smooth differential form may not be cone-smooth. In order to define the cone-smooth de Rham complex of $(M,S),$ one should consider only cone-smooth differential forms with cone-smooth exterior derivative. Since this is not used in the present paper, we will not discuss it further.
\end{rem}
\begin{lemma}\label{lemma:cone-smooth pullback}
Let $(\Psi: M \to N, S)$ be a cone-smooth map, let $\alpha$ be a smooth differential form on $N$ and let $\xi$ be a cone-smooth vector field along $\Psi.$ Then $\left(\Psi|_{M\setminus S}\right)^*\alpha$ and $i_{\xi}\alpha$ are cone-smooth differential forms on $(M,S).$ If the blowup vector field $\widetilde \xi$ vanishes on $\partial \widetilde{M}_S,$ then so does the blowup form $\widetilde{i_{\xi}\alpha}.$ If $(\Psi,S)$ is a cone-immersion, and $g$ is a Riemannian metric on $N,$ then $\left(\Psi|_{M\setminus S}\right)^*g$ is a cone-smooth Riemannian metric on $(M,S).$
\end{lemma}
\begin{proof}
This follows from Lemma~\ref{lemma:blowupdifferential}.
\end{proof}
In light of the preceding lemma, we abbreviate
\begin{equation}\label{eq:pullbackabbrev}
\Psi^*\alpha : = \left(\Psi|_{M\setminus S}\right)^*\alpha, \qquad \Psi^*g : = \left(\Psi|_{M\setminus S}\right)^*g.
\end{equation}
\begin{lemma}\label{lemma:pullbackform}
Let $\alpha$ be a cone-smooth differential form on $(M,S).$ Then, the pull-back differential form $\left(\pi|_{\widetilde{M}_S^\circ}\right)^*\alpha$ extends to a smooth differential form on $\widetilde{M}_S.$ If the blowup form $\widetilde \alpha$ vanishes on $\partial\widetilde{M}_S,$ then so does the extension of $\left(\pi|_{\widetilde{M}_S^\circ}\right)^*\alpha$ considered as a section of $\Lambda^*\left(T^*\widetilde{M}_S\right)|_{\partial\widetilde{M}_S}.$
\end{lemma}
\begin{proof}
The dual of the differential of $\pi$ gives a map of vector bundles
\[
d \pi^* : \widetilde{TM}_S^* = \pi^*T^*\!M \to T^*\widetilde{M}_S,
\]
which induces a map $\Lambda^*(d\pi^*) : \Lambda^*(\widetilde{TM}_S^*) \to \Lambda^*(T^*\widetilde{M}_S).$
Since
\[
\left(\pi|_{\widetilde{M}_S^\circ}\right)^*\alpha = \left.\Lambda^*(d\pi^*)\circ \widetilde \alpha\right|_{\widetilde{M}_S^\circ},
\]
the required extension is given by $\Lambda^*(d\pi^*)\circ \widetilde \alpha.$ The vanishing claim is immediate.
\end{proof}
We write $\pi^*\alpha$ for the extension of $\left(\pi|_{\widetilde{M}_S^\circ}\right)^*\alpha$ given by the preceding lemma.
\begin{lemma}\label{lemma:exactness of extension}
Let $\alpha$ be a cone-smooth differential 1-form on $(M,S)$ such that there exists a smooth function $f^\circ : M\setminus S \to \R$ with  $\alpha|_{M\setminus S} = df^\circ$ and the blowup form $\widetilde{\alpha}$ vanishes on $\partial\widetilde{M}_S$. Then $f^\circ$ extends to a cone-smooth function $f$ on $(M,S),$ and $S$ is contained in the critical locus of $f.$
\end{lemma}
\begin{proof}
Since $d\left(\alpha|_{M\setminus S}\right) = d(df^\circ)= 0,$ it follows that $d(\pi^*\alpha) = 0.$ Lemma~\ref{lemma:pullbackform} gives $\pi^*\alpha|_{\partial \widetilde{M}_S}=0,$ so $\pi^*\alpha$ is exact. Let $\widetilde f: \widetilde{M}_S \to \R$ be smooth with $d\widetilde{f} = \pi^*\alpha.$ After possibly adding a constant to $\widetilde f,$ we may assume that $\widetilde f|_{\widetilde{M}_S^\circ} = f^\circ \circ \pi|_{\widetilde{M}_S^\circ}.$ Again invoking the vanishing of $\pi^*\alpha|_{\widetilde M_S},$ it follows that $\widetilde{f}|_{\partial\widetilde{M}_S}$ is locally constant. So, we take $f: M \to \R$ to be the unique function such that $f \circ \pi = \widetilde f,$ and $f$ is cone-smooth by definition. The vanishing of the cone-derivative of $f$ at $S$ follows from the vanishing of $\pi^*\alpha|_{\partial\widetilde{M}_S}$ as a section of $T^*\widetilde{M}_S|_{\partial\widetilde{M}_S}$ given by Lemma~\ref{lemma:pullbackform}.
\end{proof}
\begin{dfn}
Suppose $M$ is oriented and let $\alpha$ be a cone-smooth differential form on $(M,S)$ such that $\pi^*\alpha$ has compact support. We define the integral of $\alpha$ by
\[
\int_M \alpha := \int_{\widetilde M_S} \pi^*\alpha.
\]
\end{dfn}
\begin{rem}
\begin{enumerate}[label=(\alph*)]
\item~\label{it:intcoin}
For $S \subset M$ finite, a smooth differential form $\bar \alpha$ on $M$ gives rise to a cone-smooth differential form $\alpha = \bar \alpha |_{M \setminus S}$ on $(M,S)$. Moreover, $\int_M \alpha = \int_M \bar \alpha.$ Indeed, $\pi|_{\widetilde{M}_S^\circ} : \widetilde{M}_S^\circ \to M \setminus S$ is a diffeomorphism, while $S$ and $\partial \widetilde{M}_S = \widetilde{M}_S \setminus \widetilde{M}_S^\circ$ have measure zero. So,
\[
\int_M \bar \alpha = \int_{M\setminus S} \alpha =
\int_{\widetilde M^\circ_S} \pi^*\alpha =
\int_{\widetilde M_S} \pi^*\alpha = \int_M \alpha.
\]
\item
For $k \geq 0,$ one can define cone-$C^k$ differential forms analogously to cone-smooth differential forms and the definition of the integral remains valid. Indeed, Lemma~\ref{lemma:pullbackform} continues to hold if we replace cone-smooth with cone-$C^k.$ So, we have a $C^k$ differential form $\pi^*\alpha$ on $\widetilde{M}_S$ that extends $\left(\pi|_{\widetilde{M}_S^\circ}\right)^*\alpha,$ and the integral $\int_{\widetilde{M}_S} \pi^*\alpha$ is well-defined. Moreover, just as in~\ref{it:intcoin}, a $C^k$ differential form $\bar \alpha$ on $M$ gives rise to a cone-$C^k$ differential form $\alpha = \bar \alpha|_{M \setminus S}$ on $(M,S)$ and $\int_M \alpha = \int_M \bar \alpha$.
\end{enumerate}
\end{rem}
Recall the definition of a cone-smooth vector field $\xi$ and the associated blowup vector field $\widetilde \xi$ from Definition~\ref{definition: cone-smooth differential form etc}.
\begin{lemma}\label{lemma:cone smooth vanishing tangent}
Let $\xi$ be a cone-smooth vector field on $(M,S)$ with blowup vector field $\widetilde \xi$ vanishing on $\partial \widetilde{M}_S.$ Then, there exists a unique smooth vector field $\widehat \xi$ on $\widetilde{M}_S$ such that
$
d\pi\left(\widehat \xi|_{\widetilde M_S^\circ}\right) = \xi.
$
Moreover, $\widehat \xi$ is tangent to $\partial \widetilde{M}_S.$
\end{lemma}
\begin{proof}
Since $\pi|_{\widetilde M_S^\circ}$ is a diffeomorphism, there exists a unique vector field $\overline \xi$ on $\widetilde M_S^\circ$ such that
\begin{equation}\label{equation:oxi}
d\pi(\overline \xi) = \xi.
\end{equation}
To prove the lemma, it suffices to show that $\overline \xi$ extends smoothly to a vector field $\widehat \xi$ on $\widetilde M_S$ that is tangent to $\partial \widetilde M_S.$
To this end, we use Lemma~\ref{lemma: induced tangent frame} and the notation therein. Abbreviate $V^\circ = V \cap \widetilde M_S^\circ.$ It suffices to show that $\overline \xi|_{V^\circ}$ extends to $V.$ Write
$\widetilde \xi|_V = \sum_{i = 0}^m \widetilde \xi^i e_i.$ So,
\[
\widetilde \xi|_{V^\circ} = \widetilde \xi^0|_{V^\circ} \pderiv[\pi]{s} + \sum_{i = 1}^m \widetilde \xi^i|_{V^\circ} \frac{1}{s}\pderiv[\pi]{x^i}.
\]
Furthermore, write
\[
\overline \xi|_{V^\circ} = \overline\xi^0 \pderiv{s} + \sum_{i = 1}^m \overline \xi^i \pderiv{x^i}.
\]
Equation~\eqref{equation:oxi} gives
\[
\overline{\xi}^0 = \widetilde{\xi}^0|_{V^\circ}, \qquad\qquad \overline{\xi}^i =\frac{1}{s}\widetilde{\xi}^i|_{V^\circ}, \quad i = 1,\ldots,m.
\]
Since $\widetilde \xi$ vanishes on $\partial \widetilde M_S,$ it follows that the functions $\widetilde \xi^i$ vanish on $\partial \widetilde M_S.$ By Lemma~\ref{lemma: milnor lemma}, the functions $\overline \xi^i, \, i = 1,\ldots,m,$ extend smoothly to functions $\widehat \xi^i$ on~$V.$ Take $\widehat \xi^0 = \widetilde \xi^0.$ Define $\widehat \xi$ on $V$ by
\[
\widehat \xi|_V = \widehat \xi^0 \pderiv{s} + \sum_i \widehat \xi^i \pderiv{x^i}.
\]
Since $\widehat \xi^0 = \widetilde \xi^0$ vanishes on $\partial \widetilde M_S,$ it follows that $\widehat \xi$ is tangent to $\partial \widetilde{M}_S.$
\end{proof}

\begin{rem}\label{remark: blowup differential lifting}
Let $\Psi \in \diff(M,S).$ Then Remark~\ref{remark: cone-diffeomorphism lifting} and the fact that
\[
\pi^*\Psi^* TM = \widetilde \Psi^*\pi^*TM = \widetilde \Psi^*\widetilde{TM}_S
\]
imply that the blowup differential gives an isomorphism of vector bundles
\[
\widetilde{d\Psi} : \widetilde{TM}_S \overset{\sim}{\longrightarrow} \widetilde{\Psi}^*\widetilde{TM}_S.
\]
In particular, cone-smooth diffeomorphisms act by pull-back on cone-smooth differential forms.
\end{rem}

\begin{dfn}\label{definition:immersed cone smooth differential form}
Let $K = [(\Psi: M \to N,S)]$ be a cone-immersed submanifold of type $(M,S).$ A \emph{cone-smooth differential form} on $K$ is an equivalence class $\tau = [((\chi,S),\alpha)]$ where $(\chi,S)$ represents $K$ and $\alpha$ is a cone-smooth differential form on $(M,S).$ Two pairs are equivalent if they belong to the same orbit of the $\diff(M,S)$ action given by Remark~\ref{remark: blowup differential lifting}. We may write $\alpha = \Psi^*\tau.$ Given a smooth form $\eta$ on $N,$ the \emph{restriction} to $K$ is the cone-smooth form given by
\[
\eta|_K : = [((\Psi,S),\Psi^*\eta)].
\]
We say that $\tau$ is \emph{closed} if $\alpha$ is. We say that $\tau$ \emph{vanishes at the cone locus} if $\widetilde{\alpha}$ vanishes on $\partial \widetilde{M}_S.$ Given an orientation on $K,$ if $\pi^*\alpha$ has compact support, we define
\[
\int_K \tau := \int_M \alpha.
\]
\end{dfn}

\begin{dfn}
Let $K= [(\Psi:M \to N,S)]$ be a cone-immersed submanifold, let $p = [((\Psi,S),q)]$ be a cone point, and let $\widetilde{p} = \left[\left(\orientedprojective(d\Psi_q),\widetilde q\right)\right]\in \orientedprojective(TC_pK)$. The \emph{tangent space of $K$ at $\widetilde{p}$} is defined by
\[
T_{\widetilde p} K := \widetilde{d\Psi}_{\widetilde q}\left(\widetilde{TM}_{\widetilde q}\right)\subset T_{p_0}N,
\]
which is independent of the choice of representatives. At a smooth point $p$ of $K,$ we define the tangent space $T_pK$ as in Definition~\ref{definition:points}.

Let $h = [((\Psi,S),f)]$ be a cone-smooth function on $K.$ Then the differential of $h$ at $\widetilde p$ is defined by
\[
dh_{\widetilde p} := df_{\widetilde p}\circ \widetilde{d\Psi}^{-1}_{\widetilde q} : T_{\widetilde p} K \to \R,
\]
which is independent of choices of representatives.
\end{dfn}

\begin{rem}
An orientation on a cone immersed submanifold $K$ as in Definition~\ref{definition:cone immersed submanifold}~\ref{item:orientation} is equivalent to a continuously varying orientation on its tangent spaces.
\end{rem}

\begin{dfn}
		\label{definition: critical cone point}\mbox{}
\begin{enumerate}
\item\label{item:critical cone point intrinsic}
Let $(f : M \to \R,S)$ be a cone-smooth function, and let $p \in S.$ The point $p$ is said to be a \emph{critical point} of $f$ if the cone-derivative $df_p : T_pM \to \R$ vanishes identically.
\item\label{item:critical cone point submanifold}
Let $K = [(\Psi : M \to N,S)]$ be a cone-immersed submanifold, let $h = [((\Psi,S),f)]$ be a cone-smooth function on $K,$ and let $p  = [((\Psi,S),q)]$ be a cone point. The point $p$ is a \emph{critical point} of $h$ if $q$ is a critical point of $f.$
\end{enumerate}
	\end{dfn}
\begin{rem}\label{remark: critical cone point}
It follows from Lemma~\ref{lemma:blowupdifferential} that in the situation of part~(\ref{item:critical cone point intrinsic}) of the preceding definition, if $p$ is a critical point of $f,$ then $df_{\widetilde p} = 0$ for all $\widetilde p \in E_p.$ The analogous statement holds in the situation of part~(\ref{item:critical cone point submanifold}).
\end{rem}

\begin{dfn}\label{dfn:cone-Hessian}
Let $(f : M \to \R,S)$ be a cone-smooth function, and let $p \in S$ be a critical point of $f.$
\begin{enumerate}
\item
The \emph{cone-Hessian} of $f$ at $p$ is the map
\[
\nabla df : T_pM \to T_p^*M,
\]
smooth away from $0$ and homogeneous of degree $1$, defined as follows.
By Remark~\ref{remark: critical cone point}, the blowup differential $\widetilde{df}$ vanishes on the exceptional sphere $E_p \subset \widetilde{M}_S.$ So, the restriction of the second covariant derivative $\nabla \widetilde{df} \in Hom\left(T\widetilde{M}_S, \widetilde{TM}_S^*\right)$ to $E_p$ is independent of the choice of connection. Moreover, $\nabla \widetilde{df}$ vanishes on $TE_p \subset T\widetilde{M}_S|_{E_p}.$ Recall that a vector $0 \neq v \in T_pM$ gives rise to a point $[v] \in \orientedprojective(T_pM) \simeq E_p.$ For $v \in T_pM,$ and $\widetilde v \in T_{[v]} \widetilde{M}_S$ such that $d \pi_{[v]}(\widetilde v) = v,$  we define
\[
\nabla_v df := \nabla_{\widetilde v} \widetilde{df} \in \left(\widetilde{TM}^*_S\right)_{[v]} = T_p^*M.
\]
\item
\label{definition: degenerate critical cone point}
The critical point $p$ is said to be \emph{degenerate} if there exists a tangent vector $0\ne v\in T_pM$ with $\nabla_vdf=0.$
\end{enumerate}
\end{dfn}

\begin{lemma}\label{lemma: coordinate expression of cone-second derivative}
Continue with the notation of Definition~\ref{dfn:cone-Hessian}. Let $0 \neq v \in T_pM$ and write $\widetilde p = [v].$ Take local cone coordinates with $\pderiv[\pi]{s}\left(\widetilde{p},0\right)=v$ and let $e_0,\ldots,e_m,$ denote the induced local frame of the blowup tangent bundle as in Lemma~\ref{lemma: induced tangent frame}. The cone-Hessian is given by
	\begin{equation*}
	(\nabla_vdf)(e_0(\widetilde p, 0))=\frac{\partial^2f}{\partial s^2}\left(\widetilde{p},0\right),\quad(\nabla_vdf)(e_i(\widetilde p, 0))=\frac{1}{2}\frac{\partial^3f}{\partial s^2\partial x^i}\left(\widetilde{p},0\right), \quad 1 \leq i \leq m.
	\end{equation*}
\end{lemma}
\begin{proof}
Since
\[
\pderiv[f]{x^i}(\widetilde{q},0) = 0, \qquad \widetilde q \in U,
\]
Lemma~\ref{lemma: milnor lemma} gives
$\pderiv[f]{x^i} = s g_i,$ where $g_i$ is smooth and
\[
\pderiv[g_i]{s}(\widetilde q, 0) = \frac{1}{2}\frac{\partial^3 f}{\partial s^2 \partial x^i}(\widetilde q, 0).
\]
Observe that for $s > 0,$ we have
\[
df(e_i) =
\frac{1}{s}\pderiv[f]{x^i} = g_i.
\]
By continuity, we have $df(e_i) = g_i$ everywhere. Use the connection on $\widetilde{TM}_S$ with respect to which the frame $e_0,\ldots,e_m,$ is parallel.
Then, for $i = 1,\ldots,m,$
\[
(\nabla_vdf)(e_i(\widetilde p, 0)) = \left.\pderiv{s} df(e_i) \right|_{(\widetilde p, 0)}= \pderiv[g_i]{s}(\widetilde p, 0) =  \frac{1}{2}\frac{\partial^3 f}{\partial s^2 \partial x^i}(\widetilde p, 0).
\]
The proof of the left-hand equality is similar but easier.
\end{proof}

\begin{dfn}
Let $K = [(\Psi: M \to N,S)]$ be a cone-immersed submanifold, let $h = [((\Psi,S),f)]$ be a cone-smooth function and let $p = [((\Psi,S),q)]$ be a critical cone point. Let $\widehat TC_pK = TC_pK \setminus\{0\}$ denote the punctured tangent cone. The \emph{cone-Hessian} of $h$ at $p$ is the section $\nabla dh$ of the cotangent bundle of the punctured tangent cone
\[
T^*\widehat TC_pK \to \widehat TC_pK
\]
defined as follows. Let $v = [(d\Psi_q,w)]$ be a point of the punctured tangent cone. Observe that $T_v \widehat TC_pK \simeq T_{[v]} K$ by equation~\eqref{equation:dderivatives}. Recall that $(\widetilde {T M}_S)_{[w]} \simeq T_q M$ and by definition $T_{[v]}K= \widetilde{d\Psi}_{[w]}((\widetilde {T M}_S)_{[w]}).$ We define
\[
\nabla_v dh := \nabla_{w} df \circ \widetilde{d\Psi}^{-1}_{[w]} : T_{[v]}K \to \R.
\]
The critical cone point $p$ is \emph{degenerate} if $q$ is a degenerate critical point of $f.$
\end{dfn}

	The following lemma is well-known for extrema of smooth functions. We show it is true also for cone-smooth functions.
	
	\begin{lemma}
		\label{lemma: vanishing second derivative at extremum point}
		Let $M$ be a smooth manifold, let $p\in M,$ let $0\ne v\in T_pM$ and let $h:M\to\mathbb{R}$ be cone-smooth at $p$ with $dh_p=0.$ Assume further that $p$ is an extremum point of $h.$ If the equality $(\nabla_vdh)(v)=0$ holds, then we have $\nabla_vdh=0.$ In particular, $p$ is a degenerate critical point of $h$ in this case.
	\end{lemma}

	\begin{proof}
		Without loss of generality we assume $p$ is a minimum of $h.$ Let $\widetilde{p}:=[v]\in\orientedprojective(T_pM).$ Let $(U,\mathbf{X},\alpha)$ be local cone coordinates with $\pderiv[\pi]{s}\left(\widetilde{p},0\right)=v$ and let $e_0,\ldots,e_m,$ denote the induced local frame of the blowup tangent bundle as in Lemma~\ref{lemma: induced tangent frame}. By the assumption and Lemma~\ref{lemma: coordinate expression of cone-second derivative}, we have
		\begin{equation}
		\label{second s derivative vanishes}
		\frac{\partial^2h}{\partial s^2}\left(\widetilde{p},0\right)=0.
		\end{equation}
		Since $p$ is a critical point where $h$ attains a minimum, we have
		\begin{equation}
		\label{all second derivatives are non-negative}
		\frac{\partial^2h}{\partial s^2}\left(\widetilde{q},0\right)\ge0,\quad\widetilde{q}\in U.
		\end{equation}
		For $i=1,\ldots,m,$ Lemma~\ref{lemma: coordinate expression of cone-second derivative} gives
\[
		(\nabla_vdh)\left(e_i\left(\widetilde{p},0\right)\right)
		=\frac{1}{2}\frac{\partial^3h}{\partial s^2\partial x^i}\left(\widetilde{p},0\right)
		=\frac{1}{2}\frac{\partial^3h}{\partial x^i\partial s^2}\left(\widetilde{p},0\right).
\]
Since equations~\eqref{second s derivative vanishes} and~\eqref{all second derivatives are non-negative} imply that $\tilde p$ is a local minimum for $\frac{\partial^2h}{\partial s^2}(x_1,\ldots,x_m, 0),$ it follows that
		\begin{equation}
		\label{another derivative vanishes}
		(\nabla_vdh)\left(e_i\left(\widetilde{p},0\right)\right) =0.
		\end{equation}
		Since $e_0\left(\widetilde{p},0\right) = v,$ the assumption and~\eqref{another derivative vanishes} give $\nabla_vdh=0,$ as desired.
	\end{proof}

Recall the meaning of a polar coordinate map from Definition~\ref{definition:polar coordinates}.
	\begin{lemma}
		\label{lemma: nice level sets}
		Let $M$ be a smooth manifold of dimension $m+1$ and let $p\in M.$ Let $h:M\to\mathbb{R}$ be cone-smooth at $p$ such that $p$ is a non-degenerate critical point and an extremum point of $h.$ Then there exist a positive $\epsilon$ and a polar coordinate map $\kappa:S^m\times[0,\epsilon)\to M$ centered at $p$ such that
for each $s\in(0,\epsilon)$ the restricted map $\kappa|_{S^m\times\{s\}}$ parameterizes a level set of $h.$
	\end{lemma}

	\begin{proof}
		Without loss of generality we suppose $h(p)=0$ is a minimum. Let $\pi:\widetilde{M}_p\to M$ denote the blowup projection. For simplicity, we write $h$ instead of $h\circ\pi$ and think of $h$ as a function on $\widetilde{M}_p.$ Identify a neighborhood $E_p\subset V\subset\widetilde{M}_p$ with $\orientedprojective(T_pM)\times[0,\epsilon)$ and let $r$ denote the $[0,\epsilon)$-coordinate. Then we have
		\[
		\pderiv[h]{r}\left(\widetilde{p},0\right)=0,\quad\widetilde{p}\in\orientedprojective(T_pM).
		\]
		By Lemmas~\ref{lemma: coordinate expression of cone-second derivative} and~\ref{lemma: vanishing second derivative at extremum point}, we have
		\[
		\frac{\partial^2h}{\partial r^2}\left(\widetilde{p},0\right)>0,\quad\widetilde{p}\in\orientedprojective(T_pM).
		\]
		Applying Lemma~\ref{lemma: milnor lemma} twice and diminishing $\epsilon$ if necessary, we write
		\[
		h\left(\widetilde{p},r\right)=r^2f\left(\widetilde{p},r\right),
		\]
		where $f:\orientedprojective(T_pM)\times[0,\epsilon)\to\mathbb{R}$ is smooth and positive. It follows that the function $\sqrt{h}:\orientedprojective(T_pM)\times[0,\epsilon) \to \R$ is smooth. Diminishing $\epsilon$ again if necessary, $\sqrt{h}$ has no critical points. By the implicit function theorem, there exists a diffeomorphism
		\[
		\widetilde{\kappa}:S^m\times[0,\epsilon)\to\orientedprojective(T_pM)\times[0,\epsilon)
		\]
		satisfying
		\[
		\sqrt{h}\left(\widetilde{\kappa}(q,s)\right)=s,\quad(q,s)\in S^m\times[0,\epsilon).
		\]
We claim that the map $\kappa:=\pi\circ\widetilde{\kappa}$ has all the desired properties. Let
\[
\varpi : T_pM \setminus\{0\} \to \orientedprojective(T_pM)
\]
denote the projection. To show that $\kappa$ satisfies property~(\ref{item:sigma}) of a polar coordinate map, it suffices to show that $\pderiv[\kappa]{s}(q,0) \neq 0$ for $q \in S^m$ and the composition $\varpi \circ \pderiv[\kappa]{s}(\cdot,0) : S^m \to \orientedprojective(T_pM)$ is a diffeomorphism. Indeed,
\[
\pderiv[\kappa]{s}(q,0) = d\pi \circ \pderiv[\widetilde\kappa]{s}(q,0).
\]
Since $\widetilde \kappa$ is a diffeomorphism, $\pderiv[\widetilde{\kappa}]{s}(q,0)$ is not tangent to $E_p.$ So, $\pderiv[\kappa]{s}(q,0) \neq 0.$ Finally, $\varpi \circ \pderiv[\kappa]{s}(\cdot,0) = \widetilde{\kappa}(\cdot,0),$ which is a diffeomorphism by construction. The remaining properties of $\kappa$ are immediate.
	\end{proof}

	\subsection{Cone-immersed Lagrangians}
	\label{subsection: cone-immersed Lagrangians}
	
	Let $(X,\omega)$ be a $2n$-dimensional symplectic manifold and let $L$ be a connected $n$-dimensional smooth manifold. A cone-immersion $(\Psi:L\to X,S)$ is said to be \emph{Lagrangian} if it is Lagrangian away from its cone locus. The cone-immersed submanifold represented by a Lagrangian cone-immersion is also said to be Lagrangian. Suppose $\Lambda=[(\Psi:L\to X,S)]$ is Lagrangian, let $p\in\Lambda$ be a cone point, and let $\widetilde{p}\in\orientedprojective(TC_p\Lambda).$ As the Lagrangian Grassmannian bundle of $X$ is a closed subset of the $n$-Grassmannian bundle, it follows that $T_{\widetilde{p}}\Lambda$ is a Lagrangian subspace of $T_{p_0}X.$ Thus, we have a well-defined phase function $\theta_\Psi : \widetilde{L}_S \to S^1$ given by $\theta_\Psi(q) = \theta_{\widetilde{d\Psi}_q((\widetilde{TL}_S)_q)}.$
	
	We wish to study paths of cone-immersed Lagrangians with static cone locus. For a finite subset $C_0\subset X,$ we let $\mathcal{L}(X,L;S,C_0)$ denote the space of oriented cone-immersed Lagrangians in $X$ of type $(L,S)$ with cone locus image equal to $C_0.$ For a path $\Lambda=(\Lambda_t)_{t\in[0,1]}$ in $\mathcal{L}(X,L;S,C_0),$ a lifting of $\Lambda$ is a family of cone-immersions, $((\Psi_t:L\to X,S))_{t\in[0,1]},$ such that $(\Psi_t,S)$ represents $\Lambda_t$ for $t\in[0,1].$ The path $\Lambda$ is smooth if it admits a smooth lifting, that is, if the family of maps $\Psi_t\circ\pi$ is smooth, where $\pi:\widetilde{L}_S\to L$ is the blowup projection. Given a smooth path $\Lambda = (\Lambda_t)_{t \in [0,1]},$ we define a family of $1$-forms $\sigma_t$ on $\Lambda_t$ as follows. Let $((\Psi_t:L \to X,S))_{t}$ be a smooth lifting. We abbreviate
\[
\deriv{t}\Psi_t : = \deriv{t}\Psi_t|_{L\setminus S},
\]
which is a cone-smooth vector field along $\Psi_t.$ Moreover, the blowup vector field $\widetilde{\deriv{t}\Psi_t}$ vanishes on $\partial \widetilde{L}_S.$ We define
\[
\sigma_t := \left[\left((\Psi_t,S),i_{\deriv{t}\Psi_t}\omega\right)\right].
\]
\begin{lemma}\label{lemma:derivative 1-form of cone smooth path}
The form $\sigma_t$ is independent of the choice of $\Psi_t,$ closed, and vanishes at the cone locus of $\Lambda_t.$
\end{lemma}
\begin{proof}
The proof that $\sigma_t$ is independent of $\Psi_t$ and closed is analogous to the proof of~\cite[Lemma 2.1]{akveld-salamon}.   Lemma~\ref{lemma:cone-smooth pullback} implies that $\sigma_t$ vanishes at the cone locus of $\Lambda_t.$
\end{proof}
We call $\sigma_t$ the time derivative and write
\[
\deriv{t}\Lambda_t := \sigma_t.
\]
A path of cone-immersed Lagrangians is said to be exact if its time-derivative is the differential of a cone-smooth function $\deriv{t}\Lambda_t = dh_t.$ In this case, it follows that every cone point is a critical point of $h_t.$
	
	Suppose $(X,\omega,J,\Omega)$ is Calabi-Yau. An oriented cone-immersed Lagrangian $\Lambda\in\mathcal{L}(X,L;S,C_0)$ is \emph{positive} if the tangent space $T_p\Lambda$ is positive for each smooth point $p$, and for each cone point $p$ and $\widetilde p \in \orientedprojective(TC_p\Lambda)$ the tangent space $T_{\widetilde p}\Lambda$ is positive. Note that this is stronger than positivity at the smooth locus. Assume now that $L$ is closed. Let $\mathcal{O}\subset\mathcal{L}(X,L;S,C_0)$ be an exact isotopy class of positive cone-immersed Lagrangians and let $\Lambda = [(\Psi:L \to X,S)]\in\mathcal{O}.$ Recall Definition~\ref{definition:immersed cone smooth differential form}. As the blowup $\widetilde{L}_S$ is compact, the volume form $\real\Omega|_\Lambda$ is integrable. Let $\Cs(\Lambda)$ denote the space of cone-smooth functions on $\Lambda.$ Set
	\[
	\overline{\Cs}(\Lambda):=\left\{h\in \Cs(\Lambda)\left|\int_\Lambda h\real\Omega=0,\;\forall c\in \Lambda^C,\;dh_c=0\right.\right\}.
	\]
	Then the isomorphism~\eqref{equation: tangent space to O} and Riemannian metric~\eqref{equation: the metric} make sense as in the smooth case. Let $(\Lambda_t)_t$ be a smooth path in $\mathcal{O}$ and let $(\Psi_t)_t$ be a smooth lifting. We say $(\Psi_t)_t$ is \emph{horizontal} if it satisfies $i_{\deriv{t}\Psi_t}\real\Omega=0.$ It is shown in~\cite[Section 5.3]{solomon} that every compactly supported path of \emph{smooth} positive Lagrangians admits horizontal liftings. We show the same for cone-immersed positive Lagrangians.
	
	\begin{lemma}
		\label{lemma: path of cone-Lags admits horizontal liftings}
		Let $(\Lambda_t)_{t\in[0,1]}$ be a smooth path in $\mathcal{O},$ and let $(\Psi,S)$ be a representative of $\Lambda_0.$ Then $(\Psi,S)$ extends uniquely to a horizontal lifting of $(\Lambda_t)_t.$
	\end{lemma}

	\begin{proof}
		We imitate the argument presented in~\cite{solomon}. Let $\left(\left(\Psi_t,S\right)\right)_t$ be a smooth lifting of $(\Lambda_t)_t$ with $\Psi_0=\Psi.$ For $t\in[0,1],$ let $w_t$ denote the unique vector field on $L\setminus S$ satisfying
		\[ i_{w_t}\Psi_t^*\real\Omega=-i_{\deriv{t}\Psi_t}\real\Omega.
		\]
We claim that $w_t$ is cone-smooth and the blowup vector field $\widetilde w_t$ vanishes on $\partial \widetilde L_S.$ Indeed, Lemma~\ref{lemma:cone-smooth pullback}
implies that the differential forms $\Psi_t^*\real\Omega$ and $i_{\deriv{t}\Psi_t}\real\Omega$ are cone-smooth and the blowup differential form $\widetilde{i_{\deriv{t}\Psi_t}\real\Omega}$ vanishes on $\partial \widetilde L_S.$ Moreover, since $\Lambda_t$ is positive, it follows from Lemma~\ref{lemma:blowupdifferential} that the blowup form $\widetilde{\Psi_t^*\real \Omega}$ is non-vanishing. Let $\widetilde w_t$ be the unique section of $\widetilde{TL}_S$ such that
\[
i_{\widetilde w_t}\widetilde{\Psi_t^*\real\Omega}=-\widetilde{i_{\deriv{t}\Psi_t}\real\Omega}.
\]
Then, $\widetilde w_t|_{\widetilde L_S^\circ} = \pi|_{\widetilde L_S^\circ}^*w_t.$ So, $w_t$ is cone-smooth and $\widetilde w_t$ is the blowup vector field. Furthermore, $\widetilde w_t$ vanishes on $\partial \widetilde L_S$ since $\widetilde{i_{\deriv{t}\Psi_t}\real\Omega}$ vanishes on $\partial \widetilde L_S.$

By Lemma~\ref{lemma:cone smooth vanishing tangent} there exists a unique vector field $\widehat w_t$ on $\widetilde L_S$ such that
\[
d\pi\left(\widehat w_t|_{\widetilde L_S^\circ}\right) = w_t.
\]
Moreover, $\widehat w_t$ is tangent to $\partial \widetilde L_S.$
Let $(\widehat \varphi_t)_t$ denote the flow of $(\widehat w_t)_t.$ Then, $\widehat\varphi_t(\partial\widetilde L_S) = \partial \widetilde L_S,$  and $\widehat \varphi_t$ descends to a map $\varphi_t : M \to M.$
Remark~\ref{remark: cone-diffeomorphism lifting} implies that $\varphi_t$ is cone-smooth.
Thus, the family of compositions $(\Psi_t\circ\varphi_t)_t$ gives the desired horizontal lifting.
	\end{proof}

	By virtue of Lemma~\ref{lemma: path of cone-Lags admits horizontal liftings}, the Levi-Civita connection described in Section~\ref{subsection: calabi-yau manifolds} extends naturally to exact isotopy classes of cone-immersed positive Lagrangians. We use horizontal lifts to define geodesics in such classes. The definition below is more general than that used in previous works, as it allows the Lagrangians in question to be non-closed or non-smooth. Note, however, that it is equivalent to the old definition when the Lagrangians in question are smoothly embedded and closed.
	
	\begin{dfn}
		\label{definition: geodesic}
		Let $(X,\omega,J,\Omega)$ be a Calabi-Yau manifold, let $L$ be a connected smooth manifold, not necessarily closed, and let $S \subset L$ be a finite subset. Let $C_0\subset X$ be finite, let $\mathcal{O}\subset\mathcal{L}(X,L;S,C_0)$ be an exact isotopy class of cone-immersed Lagrangians, and let $(\Lambda_t)_{t\in[0,1]}$ be a path in $\mathcal{O}.$ The path $(\Lambda_t)_t$ is a \emph{geodesic} if it admits a horizontal lifting $((\Psi_t,S))_{t\in[0,1]}$ and a family of functions $h_t\in \Cs(\Lambda_t)$ satisfying
		\[
		\deriv{t}\Lambda_t=dh_t,\quad\deriv{t}(h_t\circ\Psi_t)=0.
		\]
		We call the family $(h_t)_{t\in[0,1]}$ the \emph{Hamiltonian} of the geodesic. We also call the time independent function $h = h_t \circ \Psi_t : L \to \R$ the \emph{Hamiltonian} with respect to the horizontal lifting $(\Psi_t)_t.$ Observe that $h_t = [(\Psi_t,h)].$  If $L$ is not compact, the Hamiltonian is only well-defined up to a time independent constant. If $C_0$ is empty, we say $(\Lambda_t)_t$ is a \emph{smooth geodesic} or \emph{geodesic of smooth Lagrangians}.
	\end{dfn}
From now on, unless otherwise specified, the term geodesic will be used in the sense of the preceding definition.
	
\begin{lemma}\label{lemma:extension of geodesic to cone point}
Let $(\Lambda_t)_t$ be a path in $\mathcal{L}(X,L;S,C_0)$ and let $(\Lambda_t^\circ)_t$ be the path in $\mathcal{L}(X,L\setminus S)$ obtained by removing the cone points. If $(\Lambda_t^\circ)_t$ is a geodesic, then so is~$(\Lambda_t)_t.$
\end{lemma}	
\begin{proof}
This follows from Lemma~\ref{lemma:exactness of extension} and Lemma~\ref{lemma:derivative 1-form of cone smooth path}.
\end{proof}
	
	\section{Lagrangian and special Lagrangian cylinders}
	\label{section: lagrangian and special lagrangian cylinders}
The present section collects results on Lagrangian cylinders. Section~\ref{subsection: the space of lagrangian cylinders} begins by applying Theorem~\ref{thm: frechet} to show that the space of Lagrangian cylinders with boundary in a pair of positive Lagrangian submanifolds is a Frechet manifold. Proposition~\ref{proposition: space of special lag cylinders is one dimensional} shows that the space of imaginary special Lagrangian cylinders with boundary in a pair of positive Lagrangians is a one-dimensional submanifold of this Frechet manifold. Theorem~\ref{theorem: space of ISL cylinders is 1-dim} follows. The proof of Proposition~\ref{proposition: space of special lag cylinders is one dimensional} is based on Lemma~\ref{lemma: linearized operator}, which computes the linearization of the imaginary special Lagrangian equation with positive Lagrangian boundary conditions.

Section~\ref{subsection: relative Lagrangian Flux} defines the relative Lagrangian flux of a path of Lagrangian cylinders with boundary in a fixed pair of Lagrangians. Lemma~\ref{lemma: relative Lagrangian flux} computes the relative flux in terms of the differences in the values of Hamiltonian functions at the two ends of the cylinders.

Section~\ref{subsection: regular families of special Lagrangian cylinders} begins with a discussion of the fundamental harmonic of an imaginary special Lagrangian cylinder, which is by definition the unique harmonic function taking the value $0$ on one end of the cylinder and $1$ on the other. The cylinder is said to have regular harmonics if its fundamental harmonic has no critical points. Lemma~\ref{lemma:immersive cylinder has regular harmonics} shows that a cylinder $Z$ has regular harmonics if and only if any family of immersions representing a small enough open interval around $Z$ in the space of imaginary special Lagrangian cylinders gives an immersion of the cylinder times the interval. Definition~\ref{definition: interior regularity}~\eqref{interior regularity first part} calls such a family of immersions an interior regular parameterization under an additional technical condition. Definition~\ref{definition: interior regularity}~\eqref{compatible with harmonics} formulates the condition for an interior regular parameterization to be compatible with harmonics, and Lemma~\ref{lemma: interior and end regular parameterization compatible with harmonics}~\ref{interior regular parameterization compatible with harmonics} shows that any interior regular parameterization can be modified to be so. There is a close connection between interior regular parameterizations compatible with harmonics and horizontal lifts of geodesics of positive Lagrangians, which plays an important role in the proof of Proposition~\ref{proposition: interior-regular family gives rise to geodesic}.

Components of the space of imaginary special Lagrangian cylinders with boundary in a pair of positive Lagrangian submanifolds are typically non-compact. Nonetheless, the perturbation arguments in the proof of Theorem~\ref{theorem: perturbation of geodesic} rely on a compactness argument.
For this purpose,
Definition~\ref{definition: interior regularity}~\eqref{regular convergence to intersection point} formulates the notion of a regular parameterization about an intersection point of the positive Lagrangian boundary conditions. Such a parameterization partially compactifies the domain of an interior regular parameterization by allowing a certain degeneration of the family of immersions. The regularity requirements near the degeneration are motivated by Lemma~\ref{lemma: cylinder with regular harmonics in Euclidean space}, which pertains after rescaling. Lemma~\ref{lemma: interior and end regular parameterization compatible with harmonics}~\ref{end regular parameterization compatible with harmonics} shows that regular parameterizations about an intersection point can be modified to be compatible with harmonics.

	\subsection{The space of Lagrangian cylinders}
	\label{subsection: the space of lagrangian cylinders}
	
	We start this section with the definition of its main objects.
	
	\begin{dfn}\label{definition: Lagrangian cylinders}$\;$
		\begin{enumerate}[label=(\alph*)]
			\item \label{item:just Lagrangian} Let $(X,\omega)$ be a symplectic manifold, and let $\Lambda_0,\Lambda_1\subset X$ be smoothly embedded Lagrangians. A \emph{Lagrangian cylinder} between $\Lambda_0$ and $\Lambda_1$ is a smooth immersed Lagrangian submanifold with boundary, $Z=[f:L\to X]\in\mathcal{L}(X,L;\Lambda_0,\Lambda_1)$, where $L=N\times[0,1]$ for some smooth manifold $N,$ and the restricted immersion $f|_{N\times\{i\}}$ is an embedding into $\Lambda_i$ for $i=0,1.$ We let $\mathcal{LC}(N;\Lambda_0,\Lambda_1)$ denote the space of Lagrangian cylinders between $\Lambda_0$ and $\Lambda_1$ of type $N\times[0,1]$. We let $\mathcal{LC}(\Lambda_0,\Lambda_1)$ denote the space of Lagrangian cylinders between $\Lambda_0$ and $\Lambda_1$ of general topological type.
			\item\label{item:special} If $X$ is Calabi-Yau and $\Lambda_0$ and $\Lambda_1$ are positive, we let $\mathcal{SLC}(N;\Lambda_0,\Lambda_1)\subset\mathcal{LC}(N;\Lambda_0,\Lambda_1)$ denote the subspace consisting of imaginary special Lagrangian cylinders. The space $\mathcal{SLC}(\Lambda_0,\Lambda_1)$ is defined analogously.
		\end{enumerate}
	\end{dfn}

\begin{rem}\label{remark:automatic}\mbox{}
\begin{enumerate}
  \item \label{item:free}
Let $Z = [f : N \times [0,1] \to X]$ as in Definition~\ref{definition: Lagrangian cylinders}~\ref{item:just Lagrangian}. By Lemma~\ref{lemma: embedded boundary point implies freeness} the requirement that $f|_{N \times \{i\}}$ is an embedding implies that $f$ is free as required in Notation~\ref{notation:space of Lagrangians}.
  \item \label{item:not tangent}
Let $X, \Lambda_0,\Lambda_1,$ be as in Definition~\ref{definition: Lagrangian cylinders}~\ref{item:special}. Let $Z = [f: N \times [0,1] \to X]$ be an immersed imaginary special Lagrangian with the boundary component corresponding to $N\times \{i\}$ in $\Lambda_i.$  Then $Z$ automatically satisfies condition~\ref{not tangent} in Notation~\ref{notation:space of Lagrangians}.
\end{enumerate}
\end{rem}

	It is well-known that special Lagrangian submanifolds can be described as solutions to an elliptic PDE (see~\cite[Section~ III.2]{harvey-lawson} and~\cite[Theorem 5]{caffarelli-nirenberg-spruck} or~\cite{mclean,salur} for a geometric approach). We show below that any pair of positive Lagrangian submanifolds, $\Lambda_0$ and $\Lambda_1$, provides an elliptic boundary condition to the imaginary special Lagrangian equation. In particular, we shall see in Proposition~\ref{proposition: space of special lag cylinders is one dimensional} that, if $N$ is compact, $\mathcal{SLC}(N;\Lambda_0,\Lambda_1)$ is a smooth $1$-dimensional submanifold of the Fr\'echet manifold $\mathcal{LC}(N;\Lambda_0,\Lambda_1)$. Our approach is similar to that of McLean~\cite[Section~3]{mclean} with some necessary adaptations.
	
	Fix two Lagrangian submanifolds, $\Lambda_0,\Lambda_1\subset X,$ and a compact $n-1$-dimensional manifold $N.$ By Theorem~\ref{thm: frechet}, the space $\mathcal{LC}(N;\Lambda_0,\Lambda_1)$ is a Fr\'echet manifold modeled locally on $\Omega_B^1(N\times[0,1]),$ the space of closed 1-forms on $N\times[0,1]$ annihilating the boundary. Moreover, by Lemma~\ref{lemma: tangent space} \ref{tangent space of lag with boundary}, for $Z\in\mathcal{LC}(N;\Lambda_0,\Lambda_1)$ we have a canonical isomorphism
	\[
	T_Z\mathcal{LC}(N;\Lambda_0,\Lambda_1)\cong\Omega_B^1(Z).
	\]
	The following observation allows us to replace the spaces $\Omega_B^1(N\times[0,1])$ and $\Omega_B^1(Z)$ with spaces of functions.

	\begin{lemma}
		\label{lemma: closed forms on cylinder}
		Let $N$ be a smooth manifold without boundary, and let $\sigma\in\Omega^1(N\times[0,1])$ be closed with pull-back to the boundary component $N\times\{0\}$ zero. Then $\sigma$ is exact.
	\end{lemma}

	\begin{proof}
		Let $\pi_0:N\times[0,1]\to N\times[0,1]$ be given by $(p,t)\mapsto(p,0)$. Let $\gamma:S^1\to N\times[0,1]$ be a smooth loop, and write $\gamma_0:=\pi_0\circ\gamma$. Then $\gamma_0$ is homotopic to $\gamma$. As $\sigma$ is closed and annihilates the boundary component $N\times\{0\}$, we have
		\[
		\int_{S^1}\gamma^*\sigma=\int_{S^1}\gamma_0^*\sigma=0,
		\]
		and the lemma follows.
	\end{proof}

	In the rest of this section, we assume $N$ is connected. By Lemma~\ref{lemma: closed forms on cylinder}, every differential form in $\Omega_B^1(N\times[0,1])$ is exact with primitive constant on each boundary component. So, by Theorem~\ref{thm: frechet} we obtain the following lemma. For a smooth cylinder $Z$ with boundary components $C_0$ and $C_1$, we let $\cob{\infty}(Z)$ denote the space of smooth functions on $Z$ which vanish on $C_0$ and are constant on $C_1$.
	
	\begin{lemma}
		\label{lemma: local description of space of cylinders}
		Let $(X,\omega)$ be a symplectic manifold, let $\Lambda_0,\Lambda_1\subset X$ be fixed smoothly embedded Lagrangians, let $N$ be a connected closed $n-1$-dimensional manifold and let $Z\in\mathcal{LC}(N;\Lambda_0,\Lambda_1)$. Then any immersed Weinstein neighborhood of $Z$ compatible with $\Lambda_0$ and $\Lambda_1$ gives rise to a local parameterization
		\[
		\mathbf{X}:\mathcal{U}\subset\cob{\infty}(Z)\to\widetilde{\mathcal{U}}\subset\mathcal{LC}(N;\Lambda_0,\Lambda_1).
		\]
		Moreover, we have a canonical isomorphism
		\[
		T_Z\mathcal{LC}(N;\Lambda_0,\Lambda_1)\cong\cob{\infty}(Z).
		\]
	\end{lemma}

	Continuing with the same setting, we now assume further that $(X,\omega,J,\Omega)$ is Calabi-Yau, $\Lambda_0$ and $\Lambda_1$ are positive and $Z\in\mathcal{SLC}(N;\Lambda_0,\Lambda_1).$ Abbreviate
\[
L = N \times [0,1].
\]
Fix an immersion
\[
f: L\to X
\]
representing $Z.$
By Lemma~\ref{lemma: Weinstein with boundary}, choose an immersed Weinstein neighborhood $(V,\varphi)$ of $Z$ compatible with $\Lambda_0$ and $\Lambda_1,$ where $V \subset T^*L$ and $\varphi : V\to X$ with $\varphi|_L = f.$ Let $\pi_L : T^*L \to L$ denote the projection. For $u\in\cob{\infty}(L),$ let $\mathrm{Graph}(du) \subset T^*L$ denote the graph. Write
	\[
	\mathcal{U}:=\left\{\left.u\in\cob{\infty}(L)\right|\mathrm{Graph}(du)\subset V\right\}.
	\]
For $u\in\mathcal{U},$ let $j_u : L \to X$ be given by
\[
j_u = \varphi \circ \left( \pi_L|_{\mathrm{Graph}(du)}\right)^{-1}.
\]

Define
	\begin{equation}
	\label{SL operator}
	F:\mathcal{U}\to C^\infty(L),\qquad u\mapsto*j_u^*\real\Omega,
	\end{equation}
	where $*$ denotes the Hodge star operator of $f^*g.$ Then, for $u\in\mathcal{U}$, the cylinder $[j_u]$ is imaginary special Lagrangian if and only if the function $u$ satisfies $F(u)=0$. We have thus established a local characterization of $\mathcal{SLC}(N;\Lambda_0,\Lambda_1)$ as the zero set of a differential operator. Since $Z\in\mathcal{SLC}(N;\Lambda_0,\Lambda_1),$ we have $F(0)=0$. We compute the linearization of $F$ in Lemma~\ref{lemma: linearized operator} below. By Lemma~\ref{lemma: Omega does not vanish on Lags} we have
	\begin{equation}
	\label{equation:rho on imaginary cylinder}
	f^*\imaginary\Omega=\rho_f\vol_f,
	\end{equation}
	where $\vol_f$ denotes the Riemannian volume form of $f^*g$ and $\rho_f := \rho \circ f$ with $\rho$ the positive function defined in~\eqref{equation: rho}.
	
	\begin{lemma}
		\label{lemma: linearized operator}
		Consider the above setting.
		\begin{enumerate}[label=(\alph*)]
			\item\label{more general lemma} Let $f_t:L\to X,\;t\in(-\epsilon,\epsilon),$ be a smooth family of Lagrangian immersions with $f_0=f.$ Write $v:=\left.\deriv{t}\right|_{t=0}f_t$ and suppose we have $i_v\omega=du$ for some $u\in C^\infty\left(L\right).$ Then
			\[ \left.\deriv{t}\right|_{t=0}\left(*f_t^*\real\Omega\right)=*d(\rho_f*du).
			\]
			\item\label{less general lemma} The linearization of the operator $F$ of~\eqref{SL operator} at $0 \in \mathcal{U}$ is given by
			\[
			dF_0(u)=*d(\rho_f *du).
			\]
		\end{enumerate}
	\end{lemma}

	\begin{proof}
		We prove part \ref{more general lemma}. Decompose $v$ into
		\[
		v=v^T+v^\perp,
		\]
		where $v^T$ is tangent to $Z$ and $v^\perp$ is orthogonal to $Z$ with respect to the K\"ahler metric $g$. As $Z$ is Lagrangian, we have
		\[
		i_{v^\perp}\omega=i_{v}\omega=du,
		\]
		which implies
		\begin{equation}
		\label{orthogonal component}
		v^\perp=-Jdf(\nabla u),
		\end{equation}
		where $\nabla$ denotes the gradient with respect to $f^*g$. By~\eqref{equation:rho on imaginary cylinder} and~\eqref{orthogonal component}, as $Z$ is imaginary special Lagrangian and $\Omega$ is complex-linear, we have
		\begin{align}
		\label{gradient contraction}
		i_{v}\real\Omega
		&=i_{v^\perp}\real\Omega\\
		&=i_{\nabla u}f^*\imaginary\Omega\nonumber\\
		&=\rho_f \,i_{\nabla u}\vol_f\nonumber\\
		&=\rho_f*du.\nonumber
		\end{align}
		As $\Omega$ is closed, from the Cartan formula and~\eqref{gradient contraction} we deduce
		\begin{align*}
		*\left(\left.\deriv{t}\right|_{t=0}f_t^*\real\Omega\right)
		&=*di_{v}\real\Omega\\
		&=*d(\rho_f*du).
		\end{align*}
Finally, since $*$ is the Hodge star operator associated with the metric $f^*g = f_0^*g$ and thus independent of $t,$ it follows that
\begin{equation*}
\left.\deriv{t}\right|_{t=0}\left(*f_t^*\real\Omega\right) =  *\left(\left.\deriv{t}\right|_{t=0}f_t^*\real\Omega\right) = *d(\rho_f*du)
\end{equation*}
		as desired. Part \ref{less general lemma} is a particular case of \ref{more general lemma}.
	\end{proof}

	In light of Lemma~\ref{lemma: linearized operator} we define the linear second-order operator
	\begin{equation}
	\label{equation: definition of rho-Laplacian}
	\Delta_\rho:\cob{\infty}(L)\to C^\infty(L),\qquad u\mapsto*d(\rho_f*du).
	\end{equation}
For $Z  \in \mathcal{SLC}(N;\Lambda_0,\Lambda_1),$ and $f : L \to X$ an immersion with $Z = [f],$ we have a canonical identification $C^\infty(Z) = C^\infty(L),$ so we obtain an operator
\[
\Delta_\rho : \cob{\infty}(Z) \to C^\infty(Z).
\]
This operator does not depend on the choice of $f.$
	The operator $\Delta_\rho$ is similar to the usual Riemannian Laplacian in a manner made precise in Lemma~\ref{lemma: Delta rho is a cool operator} below.

For $k\ge0$ and $\alpha\in(0,1)$ we let $\cob{k,\alpha}(L)$ denote the completion of $\cob{\infty}(L)$ with respect to the H\"older $C^{k,\alpha}$-norm. We let $C^{k,\alpha}(L;\partial L)$ denote the space of functions on $L$ of regularity $C^{k,\alpha}$ which vanish on the boundary. For $P,Q,$ smooth manifolds, we let $C^{k,\alpha}(P,Q)$ denote the smooth Banach manifold of maps $P \to Q$ of regularity $C^{k,\alpha}.$

	\begin{lemma}
		\label{lemma: Delta rho is a cool operator}
		Let $k$ be a non-negative integer and let $\alpha\in(0,1)$. Then the naturally extended linear operator $\Delta_\rho:\cob{k+2,\alpha}(L)\to C^{k,\alpha}(L)$ is surjective with a 1-dimensional kernel consisting of $C^\infty$ functions.
	\end{lemma}

	\begin{proof}
		We note first that the principal symbol of $\Delta_\rho$ differs from that of the usual Riemannian Laplacian by the positive function $\rho_f$. Hence $\Delta_\rho$ is elliptic. Also, $\Delta_\rho$ annihilates constants. Equivalently, expressing $\Delta_\rho$ in local coordinates as
		\[
		\Delta_\rho u=a^{ij}u_{ij}+b^iu_i+cu,
		\]
		the coefficient $c$ vanishes. As $L$ is compact, it now follows from standard arguments (see, for example,~\cite[Chapter~6]{gilbarg-trudinger} and~\cite[Section~5.1]{taylor1}) that
\[
\Delta_\rho:C^{k+2,\alpha}(L;\partial L)\to C^{k,\alpha}(L)
\]
is one-to-one and surjective. Finally, for $a\in\R,$ there exists a unique $u\in\cob{k+2,\alpha}(L)$ satisfying
		\[
		\Delta_\rho u=0,\quad u|_{N\times \{1\}}=a,
		\]
		which is in fact $C^\infty$. The lemma follows.
	\end{proof}

	\begin{prop}
		\label{proposition: space of special lag cylinders is one dimensional}
		Let $\Lambda_0,\Lambda_1\subset X$ be smoothly embedded positive Lagrangians, let $N$ be a connected closed smooth manifold, and abbreviate $L = N\times [0,1].$ Then the space $\mathcal{SLC}(N;\Lambda_0,\Lambda_1)$ is a smoothly embedded $1$-dimensional submanifold of $\mathcal{LC}(N;\Lambda_0,\Lambda_1)$. Moreover, for $Z = [f : L \to X]\in\mathcal{SLC}(N;\Lambda_0,\Lambda_1)$, we have
		\[
		T_Z\mathcal{SLC}(N;\Lambda_0,\Lambda_1)=\ker\Delta_\rho,
		\]
		where $\Delta_\rho:\cob{\infty}(L)\to C^\infty(L)$ is defined as in~\eqref{equation: definition of rho-Laplacian}.
	\end{prop}

	\begin{proof}
		Let $Z = [f : L \to X]\in\mathcal{SLC}(N;\Lambda_0,\Lambda_1)$. Recall the definition of the operator $F : \mathcal{U} \to C^\infty(L)$ from~\eqref{SL operator}. Pick $\alpha\in(0,1)$ and let $\mathcal{U}^{2,\alpha} \subset \cob{2,\alpha}(L)$ be an open set with $\mathcal{U}^{2,\alpha} \cap \cob{\infty}(L) = \mathcal{U}.$ Extend the operator $F$ to an operator
\[
F : \mathcal{U}^{2,\alpha} \to C^{\alpha}(L).
\]
Since $\Omega$ is smooth, and the map
\[
\mathcal{U}^{2,\alpha} \to C^{1,\alpha}(L,X), \qquad u \mapsto j_u,
\]
is smooth, it follows that
the extended operator $F$ is smooth. Since $Z$ is imaginary special Lagrangian, $F(0)=0$. By Lemmas~\ref{lemma: linearized operator} and~\ref{lemma: Delta rho is a cool operator}, the linearization
\[
dF_0=\Delta_\rho:\cob{2,\alpha}(L)\to C^\alpha(L)
\]
is onto with a $1$-dimensional kernel. By the implicit function theorem, there exist $\epsilon>0$, an open $0\in\mathcal{W}\subset\cob{2,\alpha}(L)$ and a smooth embedding $\gamma:(-\epsilon,\epsilon)\to\mathcal{W}$ with $\gamma(0)=0,$ such that for $f\in\mathcal{W}$, we have $F(f)=0$ if and only if $f=\gamma(t)$ for some $t\in(-\epsilon,\epsilon)$. The path $\gamma$ satisfies
		\[
		\Delta_\rho(\dot{\gamma}(0))=\left.\deriv{t}\right|_{t=0}F(\gamma(t))=0.
		\]
		By elliptic regularity (e.g.\ \cite[Chapter 17]{gilbarg-trudinger}), for $t\in(-\epsilon,\epsilon)$ the function $\gamma(t)$ is in fact in $\cob{\infty}(L)$. Moreover, for every $k\ge2$, the above argument with $2$ replaced by $k$ shows that $\gamma$ is smooth as an embedding into the space $\cob{k,\alpha}(L)$. It follows that $\gamma$ is smooth as a map into $\cob{\infty}(L)$.
	\end{proof}

	\begin{proof}[Proof of Theorem~\ref{theorem: space of ISL cylinders is 1-dim}]
		Follows from Proposition~\ref{proposition: space of special lag cylinders is one dimensional}.
	\end{proof}

	For a Lagrangian cylinder $Z \in\mathcal{LC}(N;\Lambda_0,\Lambda_1)$, let $Z^\mathcal{H}\subset\mathcal{LC}(N;\Lambda_0,\Lambda_1)$ denote the space of Lagrangian cylinders exact isotopic to $Z$ relative to the boundary. Then we have
	\[
	T_ZZ^\mathcal{H}=C^\infty(Z;\partial Z)\subset\cob{\infty}(Z),
	\]
	the space of smooth functions vanishing on the boundary. As we have
	\[
	\cob{\infty}(Z)=C^\infty(Z;\partial Z)\oplus\ker\Delta_\rho,
	\]
	the following is a consequence of Proposition~\ref{proposition: space of special lag cylinders is one dimensional}.
	
	\begin{cor}
		\label{corollary: Hamiltonian isotopy class is descrete}
		Let $\Lambda_0,\Lambda_1\subset X$ be smoothly embedded positive Lagrangians and let $Z\in\mathcal{LC}(N;\Lambda_0,\Lambda_1)$. Then the intersection $Z^\mathcal{H}\cap\mathcal{SLC}(N;\Lambda_0,\Lambda_1)$ is a discrete subset of $\mathcal{LC}(N;\Lambda_0,\Lambda_1)$.
	\end{cor}

	\begin{rem}
		\label{remark: not only cylinders}
		Using the above technique, one can generalize Proposition~\ref{proposition: space of special lag cylinders is one dimensional} to spaces of imaginary special Lagrangians modeled by an arbitrary manifold $L$ with any number of boundary components. Namely, recalling Notation~\ref{notation:space of Lagrangians}, for $\Lambda_1,\ldots,\Lambda_k \subset X$ positive Lagrangians, let
\[
\mathcal{SL}(X,L;\Lambda_1,\ldots,\Lambda_k) \subset \mathcal{L}(X,L;\Lambda_1,\ldots,\Lambda_k)
\]
denote the subspace of imaginary special Lagrangians. As Lemma~\ref{lemma: closed forms on cylinder} fails to hold for arbitrary manifolds $L,$ one considers the space $\Omega_B^1(L)$ as a local model of $\mathcal{L}(X,L;\Lambda_1,\ldots,\Lambda_k)$ as in Theorem~\ref{thm: frechet}. Special Lagrangian submanifolds near $Z \in \mathcal{SL}(X,L;\Lambda_1,\ldots,\Lambda_k)$ are then parameterized by $\Delta_\rho$-harmonic $1$-forms on $L$, where $\Delta_\rho$ in this case is a modified Hodge Laplacian. It follows from the Hodge decomposition for manifolds with boundary (see~\cite{gunterschwartz},~\cite[Section~5.9]{taylor1} and the references therein) that $\mathcal{SL}(X,L;\Lambda_1,\ldots,\Lambda_k)$ has the dimension of the real relative cohomology space $H^1(L,\partial L)$. Similarly to Corollary~\ref{corollary: Hamiltonian isotopy class is descrete}, the intersection $Z^\mathcal{H} \cap \mathcal{SL}(X,L;\Lambda_1,\ldots,\Lambda_k)$ is discrete. Since the first version of this paper appeared, this line of investigation has been taken up in~\cite{Pi25a,Pi25b}.	\end{rem}

	\subsection{Relative Lagrangian flux}
	\label{subsection: relative Lagrangian Flux}
	
	Let $(X,\omega)$ be a symplectic manifold. Recall the notion of \emph{Lagrangian flux}~\cite[Section~1]{fukaya}, \cite[Section~6]{solomon}, generalizing the flux of a symplectomorphism~\cite{calabi}. Roughly speaking, given a path of closed Lagrangian submanifolds, $\Lambda_t \subset X,\;t\in[0,1],$ its Lagrangian flux is a linear functional on $H_1(\Lambda_0)$ which measures the path's deviation from being exact. Suppose now that $\Lambda_0,\Lambda_1\subset X$ are fixed Lagrangians and $Z_s,\;s\in[0,1],$ is a path in $\mathcal{LC}(N;\Lambda_0,\Lambda_1).$ Then the \emph{relative} Lagrangian flux of the path is a linear functional on $H_1(Z_0,\partial Z_0)$ which measures the deviation of the path from being exact relative to the boundary. As the relative homology group $H_1(Z_0,\partial Z_0)$ is generated by a single element, one can think of the relative Lagrangian flux of a path of cylinders as a number. The precise definition is as follows.
	
Let $N$ be a closed connected $n-1$-dimensional smooth manifold and let $Z_s,\;s\in[s_0,s_1],$ be a smooth path in $\mathcal{LC}(N;\Lambda_0,\Lambda_1)$. Let $\Phi:N\times[0,1]\times[s_0,s_1]\to X$ be a smooth parameterization of the path $(Z_s)$. Namely, for $s\in[s_0,s_1],$ the restricted map $\Phi_s:=\Phi|_{N\times[0,1]\times\{s\}}$ is a parameterization of the cylinder $Z_s$, and for $(p,s)\in N\times[s_0,s_1]$ we have
\[
	\Phi(p,i,s)\in\Lambda_i,\quad i=0,1.
\]
For $s\in[s_0,s_1],$ write
\[
	h_s:=\deriv{s}Z_s\in\cob{\infty}(Z_s).
\]
	For $i = 0,1,$ let $C_{i,s}$ denote the boundary component of the cylinder $Z_s$ corresponding to $N \times \{i\},$ let $A_s \in \R$ satisfy
	\[
	h_s|_{C_{1,s}} \equiv A_s.
	\]
	Let $\gamma:[0,1]\to N\times[0,1]$ be a smooth path representing the fundamental class in the relative homology group $H_1(N\times[0,1],\partial(N\times[0,1])).$ Namely, $\gamma$ satisfies
	\[
	\gamma(i)\in N\times\{i\},\quad i=0,1.
	\]
	Set
	\[
	\overline{\gamma}:[0,1]\times[s_0,s_1]\to X,\quad(t,s)\mapsto\Phi(\gamma(t),s).
	\]

	\begin{lemma}
		\label{lemma: relative Lagrangian flux}
		In the above setting we have
		\begin{equation}
		\label{equality: relative Lagrangian flux}
		\int_{[0,1]\times[s_0,s_1]}\overline{\gamma}^*\omega=-\int_{s_0}^{s_1}A_sds.
		\end{equation}
	\end{lemma}

	\begin{proof}
		For $s\in[s_0,s_1],$ let
		\[
		Y_s:=\deriv{s}(\Phi_s\circ\gamma)\in\Gamma([0,1],(\Phi_s\circ\gamma)^*TX).
		\]
		By definition of the derivative of a Lagrangian path (see Remark~\ref{remark: tangent space hands on}), we have
		\[
		i_{Y_s}\omega=d(h_s\circ\Phi_s\circ\gamma).
		\]
		By Fubini's theorem,
		\begin{align*}
		\int_{[0,1]\times[s_0,s_1]}\overline{\gamma}^*\omega
		&=-\int_{s_0}^{s_1}\left(\int_{[0,1]}i_{Y_s}\omega\right)ds\\
		&=-\int_{s_0}^{s_1}A_sds,
		\end{align*}
		as desired.
	\end{proof}

	\begin{dfn}
		\label{definition:relative Lagrangian flux}
		The quantity in equality~\eqref{equality: relative Lagrangian flux} is called the \emph{relative Lagrangian flux} of the path $(Z_s)_{s\in[s_0,s_1]}.$  We let $\mathrm{RelFlux}\left((Z_s)_{s\in[s_0,s_1]}\right)$ denote the relative Lagrangian flux. More generally, if $I \subset \R$ is an interval, possibly open or half open, with endpoints $a < b,$ and $(Z_s)_{s \in I}$ is a path in $\mathcal{LC}(N;\Lambda_0,\Lambda_1),$ we write
\[
\mathrm{RelFlux}\left((Z_s)_{s\in I}\right) = \lim_{s_0 \to a}\lim_{s_1 \to b}\mathrm{RelFlux}\left((Z_s)_{s\in[s_0,s_1]}\right)
\]
whenever the limit exists.
	\end{dfn}

\begin{rem}\label{remark:invariance of relative flux}
	A straightforward modification of the arguments in~\cite[Section~6]{solomon} shows that $\mathrm{RelFlux}\left((Z_s)_{s\in I}\right)$ is independent of the choices of the parameterization $\Phi$ and the path $\gamma.$ Moreover, when $I$ is a closed interval, $\mathrm{RelFlux}\left((Z_s)_{s\in I}\right)$ depends only on the homotopy class of the path $(Z_s)_s$ relative to its endpoints.
\end{rem}

	\subsection{Regular families of special Lagrangian cylinders}
	
	\label{subsection: regular families of special Lagrangian cylinders}
	
	Let $(X,\omega,J,\Omega)$ be a Calabi-Yau manifold of real dimension $2n,$ let $\Lambda_0,\Lambda_1\subset X$ be smoothly embedded positive Lagrangians, and let $N$ be a closed smooth manifold of dimension~$n-1.$
	
	\begin{dfn}
		\label{definition: fundamental harmonic and regular harmonics}
		Let $Z\in\mathcal{SLC}(N;\Lambda_0,\Lambda_1),$ and let $C_i,\;i=0,1,$ denote the boundary component of $Z$ corresponding to $N \times \{i\}.$ The \emph{fundamental harmonic} on $Z$ is the unique function $\sigma\in\cob{\infty}(Z)$ satisfying
		\[
		\Delta_\rho(\sigma)=0,\quad\sigma|_{C_1}\equiv 1.
		\]
		If $\sigma$ has no critical points, we say $Z$ has \emph{regular harmonics}.
	\end{dfn}

For the remainder of this section, we assume that $N$ is connected. Let
\[
\gamma:(-\epsilon,\epsilon)\to\mathcal{SLC}(N;\Lambda_0,\Lambda_1)
\]
be a smooth path. A smooth lifting of $\gamma$ is a smooth  map $\Phi:N\times[0,1]\times(-\epsilon,\epsilon)\to X$ such that for $s \in (-\epsilon,\epsilon)$ the restriction
\[
\Phi_s : = \Phi|_{N\times [0,1]\times\{s\}} : N \times [0,1] \to X
\]
represents the immersed submanifold $\gamma(s)\in \mathcal{SLC}(N;\Lambda_0,\Lambda_1).$
	\begin{lemma}
		\label{lemma:immersive cylinder has regular harmonics}
		Let $Z\in\mathcal{SLC}\left(N;\Lambda_0,\Lambda_1\right).$ Then the following are equivalent.
\begin{enumerate}
\item\label{item:regular harmonics}
$Z$ has regular harmonics.
\item\label{item:all}
For every embedded curve $\gamma : (-\epsilon,\epsilon) \to \mathcal{SLC}\left(N;\Lambda_0,\Lambda_1\right)$ with $\gamma(0) = Z,$ and for every smooth lifting of $\gamma,$
		\[
		\Phi:N\times[0,1]\times(-\epsilon,\epsilon)\to X,
		\]
after possibly diminishing $\epsilon,$ the map $\Phi$ is an immersion.
\item\label{item:one}
For one embedded curve $\gamma : (-\epsilon,\epsilon) \to \mathcal{SLC}\left(N;\Lambda_0,\Lambda_1\right)$ with $\gamma(0) = Z,$ and one smooth lifting of $\gamma,$
\[
		\Phi:N\times[0,1]\times(-\epsilon,\epsilon)\to X,
		\]
the map $\Phi$ is an immersion.
\end{enumerate}
\end{lemma}
	
	\begin{proof}
First, we prove that condition~(\ref{item:one}) implies condition~(\ref{item:regular harmonics}).		Let $\gamma$ and $\Phi$ be as in condition~(\ref{item:one}).
		Let $\sigma$ denote the fundamental harmonic on $Z$. By Proposition~\ref{proposition: space of special lag cylinders is one dimensional}, we have
		\[
		\dot\gamma(0)=a\sigma
		\]
		for some $a\in\R.$ Let
		\begin{equation}
		\label{this is Y}
		Y:=\left.\deriv{s}\right|_{s=0}\Phi(\cdot,\cdot,s)\in\Gamma\left(N\times[0,1], \Phi_0^*TX\right).
		\end{equation}
		By Remark~\ref{remark: tangent space hands on} we have
		\begin{equation}
		\label{equation: Y and a and sigma}
		d(a\sigma\circ\Phi_0)=i_Y\omega.
		\end{equation}
		As $\Phi$ is an immersion, the section $Y$ is nowhere tangent to $Z$. As $Z$ is Lagrangian, the right-hand side of~\eqref{equation: Y and a and sigma} is nowhere vanishing. It follows that $a\sigma$ does not have critical points and neither does $\sigma.$
		
		Next, we prove that condition~(\ref{item:regular harmonics}) implies condition~(\ref{item:all}). Indeed, suppose $Z$ has regular harmonics. Let $\gamma$ and $\Phi$ be as in condition~(\ref{item:all}). By Proposition~\ref{proposition: space of special lag cylinders is one dimensional} we have
		\[
		\dot{\gamma}(0)=a\sigma
		\]
		for some $0\ne a\in\mathbb{R}.$ Defining $Y$ as in~\eqref{this is Y}, the equality~\eqref{equation: Y and a and sigma} continues to hold true. By assumption, the left-hand side of~\eqref{equation: Y and a and sigma} is non-vanishing. It follows that $Y$ is nowhere tangent to $Z,$ which implies that, diminishing $\epsilon$ if necessary, $\Phi$ is an immersion.

It is immediate that condition~(\ref{item:all}) implies condition~(\ref{item:one}).
	\end{proof}
	
	The following lemma is, in fact, a particular case of Lemma~\ref{lemma:immersive cylinder has regular harmonics}. We provide an independent proof for clarity.
	
	\begin{lemma}
		\label{lemma: cylinder with regular harmonics in Euclidean space}
		Equip $\mathbb{C}^n$ with the standard Calabi-Yau structure, let $\Lambda_0,\Lambda_1\subset\mathbb{C}^n$ be positive Lagrangian linear subspaces, and let $Z\in\mathcal{SLC}\left(S^{n-1};\Lambda_0,\Lambda_1\right).$ Then $Z$ has regular harmonics if and only if $Z$ is nowhere tangent to the Euler vector field.
	\end{lemma}
	
	\begin{proof}
		Let $\Phi_1:S^{n-1}\times[0,1]\to\mathbb{C}^n$ be a representative of $Z,$ and set
		\[
		\Phi:S^{n-1}\times[0,1]\times(0,\infty)\to\mathbb{C}^n,\quad(p,t,s)\mapsto s\cdot\Phi_1(p,t).
		\]
		Then for $s\in(0,\infty),$ the map $\Phi_s:=\Phi|_{S^{n-1}\times[0,1]\times\{s\}}$ is an immersion representing an element in $\mathcal{SLC}\left(S^{n-1};\Lambda_0,\Lambda_1\right).$ Hence, we have
		\[
		\left.\deriv{s}\right|_{s=1}[\Phi_s]=a\sigma
		\]
		for some $0\ne a\in\mathbb{R},$ where $\sigma$ denotes the fundamental harmonic on $Z.$ By Remark~\ref{remark: tangent space hands on}, a point $q\in Z$ is a critical point of $\sigma$ if and only if the Euler vector at $q$ is tangent to $Z.$ The lemma follows.
	\end{proof}

	The following lemma, which is in fact an elementary observation in functional analysis, shows that the fundamental harmonic depends smoothly on the geometry of a cylinder. In particular, the property of having regular harmonics is stable under small perturbations.
	
	\begin{lemma}
		\label{lemma: regular harmonics is open property}
		Let $N$ be a closed connected smooth manifold, and write $Z:=N\times[0,1].$ Let $C_i:=N \times\{i\},i=0,1,$ denote the boundary components of $Z.$ Let $\alpha\in(0,1),$ let $k\in\mathbb{N},$ and let $L:\cob{k+2,\alpha}(Z)\to C^{k,\alpha}(Z)$ be a bounded linear operator, such that $L_0:=L|_{C^{k+2,\alpha}(Z;\partial Z)}:C^{k+2,\alpha}(Z;\partial Z)\to C^{k,\alpha}(Z)$ is an isomorphism. Let $u\in\cob{k+2,\alpha}(Z)$ be the unique function satisfying
		\[
		Lu=0,\quad u|_{C_1}\equiv 1.
		\]
		Then, letting $\mathcal{B}\left(\cob{k+2,\alpha}(Z),C^{k,\alpha}(Z)\right)$ denote the space of bounded linear operators $\cob{k+2,\alpha}(Z)\to C^{k,\alpha}(Z)$ endowed with the operator norm, there is an open neighborhood
\[
L\in \mathcal{V}\subset\mathcal{B}\left(\cob{k+2,\alpha}(Z),C^{k,\alpha}(Z)\right)
\]
such that for $\widetilde{L}\in\mathcal{V},$ there exists a unique $\widetilde{u}=\widetilde{u}_{\widetilde{L}}$ satisfying
		\[
		\widetilde{L}\widetilde{u}=0,\quad\widetilde{u}|_{C_1}\equiv 1.
		\]
		Moreover, the assignment $\widetilde{L}\mapsto\widetilde{u}_{\widetilde{L}}$ is smooth.
	\end{lemma}

	\begin{proof}
		Let $L\in\mathcal{V}\subset\mathcal{B}\left(\cob{k+2,\alpha}(Z),C^{k,\alpha}(Z)\right)$ be open such that, for $\widetilde{L}\in\mathcal{V},$ the restricted operator $\widetilde{L}_0:=\widetilde{L}|_{C^{k+2,\alpha}(Z;\partial Z)}$ is an isomorphism. The assignment $\widetilde{L}\mapsto\widetilde{L}_0^{-1}$ is smooth in $\mathcal{V}.$ For $\widetilde{L}\in\mathcal{V},$ the unique $\widetilde{u}$ with the desired properties is given by
		\[
		\widetilde{u}_{\widetilde{L}}=u-\widetilde{L}_0^{-1}\widetilde{L}u.
		\]
		The lemma follows.
	\end{proof}
	
	Next, we define interior regularity and regular convergence for families of special Lagrangian cylinders. As a preliminary, we recall the notion of ends (see~\cite{freudenthal}). Let $A$ be a topological space. The set of ends $\mathcal{E}(A)$ is given by
\[
\mathcal{E}(A) := \varprojlim_{K \subset\subset A} \pi_0(A\setminus K).
\]
Thus, an end $E \in \mathcal{E}(A)$ determines for every compact $K\subset A$ a connected component $E(K)$ of the complement $A\setminus K$, such that for two compact subsets $K\subset K'$ we have $E(K')\subset E(K).$ For our purposes the space $A$ will always be a connected component of $\mathcal{SLC}\left(S^{n-1};\Lambda_0,\Lambda_1\right),$ which is 1-dimensional. We thus use the following terminology, which is somewhat more elementary and equivalent in the case at hand. Let $C$ be a connected non-compact $1$-dimensional manifold. That is, $C$ is a curve diffeomorphic to the real line. A \emph{ray} in $C$ is a connected open proper subset $U\subsetneq C$ with non-compact closure. Two rays, $U,V\subset C$ are said to be equivalent if $U\subset V$ or $V\subset U.$ Finally, an \emph{end} is an equivalence class of rays. Every curve $C$ as above has exactly two ends.

	\begin{dfn}
		\label{definition: interior regularity}
		$\;$
		\begin{enumerate}
			\item\label{interior regularity first part} Let $U\subset\mathcal{SLC}\left(N;\Lambda_0,\Lambda_1\right)$ be open and connected. An \emph{interior-regular parameterization} of $U$ is a smooth immersion $\Phi:N\times[0,1]\times(a,b)\to X$ satisfying the following conditions:
			\begin{enumerate}[ref=(\alph*)]
				\item\label{item: boundary embedding} The restriction of $\Phi$ to a boundary component $\Phi|_{N\times\{i\}\times(a,b)}$ is an embedding for $i = 0,1$.
				\item\label{item: rep} For $s\in(a,b),$ the restricted immersion $\Phi_s:=\Phi|_{N\times[0,1]\times\{s\}}$ represents an element of $U$.
				\item\label{item: diffeo s} The map $\chi:(a,b)\to U,\quad s\mapsto[\Phi_s],$ is a diffeomorphism.
			\end{enumerate}
			The subset $U$ is said to be \emph{interior-regular} if it admits an interior-regular parameterization.
			\item \label{compatible with harmonics}
Let $U$ be as in~\eqref{interior regularity first part} and let $\Phi$ be an interior-regular parameterization of $U$. For $s\in(a,b),$ write $Z_s:=[\Phi_s]$ and let $\sigma_s$ denote the fundamental harmonic on $Z_s.$ The parameterization $\Phi$ is \emph{compatible with the harmonics of $U$} if the following conditions hold.
			\begin{enumerate}[ref=(\alph*)]
				\item\label{t compatible}
For $(p,t,s)\in N\times[0,1]\times(a,b)$ we have
				\[
				\sigma_s(\Phi(p,t,s))=t.
				\]
				\item\label{orthogonal} For $(p,t,s)\in N\times[0,1]\times(a,b)$, the derivative $\pderiv{t}\Phi(p,t,s)$ is orthogonal to the $t$-level set of $\sigma_s.$
			\end{enumerate}
			\item\label{regular convergence to intersection point} Let $\mathcal{Z}\subset\mathcal{SLC}\left(S^{n-1};\Lambda_0,\Lambda_1\right)$ be a connected component, let $E$ be an end of $\mathcal{Z},$ and let $q\in\Lambda_0\cap\Lambda_1.$ A \emph{regular parameterization} of $E$ about $q$ is a smooth map $\Phi:S^{n-1}\times[0,1]\times[0,\epsilon)\to X$ satisfying the following conditions.
			\begin{enumerate}[ref=(\alph*)]
				\item\label{critical point} For $(p,t)\in S^{n-1}\times[0,1]$ we have $\Phi(p,t,0)=q.$
				\item\label{item:interior regular} The restricted map $\Phi|_{S^{n-1}\times[0,1]\times(0,\epsilon)}$ is an interior-regular parameterization of $U,$ for some ray $U\subset\mathcal{Z}$ representing $E.$
				\item \label{item:derivative Phi immersion} The derivative
				\[
				\left.\pderiv{s}\right|_{s=0}\Phi(\cdot,\cdot,s):S^{n-1}\times[0,1]\to T_qX
				\]
				is an immersion and the restriction $\left.\left.\pderiv{s}\right|_{s=0}\Phi(\cdot,\cdot,s)\right|_{S^{n-1}\times\{i\}}$ is an embedding for $i = 0,1.$
				\item \label{item:nowhere tangent Euler} The Euler vector field on $T_qX$ is nowhere tangent to the immersion $\left.\pderiv{s}\right|_{s=0}\Phi(\cdot,\cdot,s).$
			\end{enumerate}
In this case, we also say that $\Phi$ is a regular parameterization of $U$ about~$q.$ A regular parameterization $\Phi$ as above is said to be compatible with the harmonics of $E$ if the interior-regular restricted map $\Phi|_{S^{n-1}\times[0,1]\times(0,\epsilon)}$ is. We say the end $E$ \emph{converges regularly} to the intersection point $q$ if it admits a regular parameterization about $q.$ We may use a half-open interval with arbitrary endpoints, open either from below or above, in place of the half-open interval $[0,\epsilon).$
		\end{enumerate}
	\end{dfn}
The following lemmas make it easier to check regularity.
\begin{lemma}\label{lemma: immersion nowhere tangent to Euler}
Let $N$ be a compact manifold possibly with boundary. Let $G : N \times [0,\epsilon) \to \R^m$ be a smooth map satisfying the following.
\begin{enumerate}
  \item
The restriction
$G|_{N \times \{0\}}$
is an immersion.
  \item
The Euler vector field on $\R^m$ is nowhere tangent to the immersion $G|_{N \times \{0\}}.$
\end{enumerate}
Let $F : N \times [0,\epsilon) \to \R^m$ be given by
\[
F(p,s) = sG(p,s).
\]
Then, possibly after shrinking $\epsilon,$ the restriction $F|_{N \times (0,\epsilon)}$ is an immersion. Moreover, if $N = S^{m-1},$ and $G|_{N \times \{0\}}$ is an embedding, then $F|_{N \times (0,\epsilon)}$ is an embedding.

Furthermore, there exists a $C^1$ neighborhood $G \in U$ in the space of smooth maps $C^\infty(N \times [0,\epsilon),\R^m)$ such that the conclusions of the lemma continue to hold with a uniform choice of $\epsilon$ if we replace $G$ with any map in $U.$
\end{lemma}
\begin{proof}
Possibly after shrinking $\epsilon,$ the restriction $G|_{N \times \{s\}}$ is an immersion to which the Euler vector field is nowhere tangent for $s \in [0,\epsilon).$ Letting $\varepsilon$ denote the Euler vector field, we have
\[
\pderiv[F]{s}(p,s) = \varepsilon(G(p,s)) + s \pderiv[G]{s}(p,s).
\]
So, possibly shrinking $\epsilon$ again, we may assume that
\[
\pderiv[F]{s}(p,s) \notin \mathrm{Im}\left(d\left(G|_{N \times \{s\}}\right)_{(p,s)}\right), \qquad (p,s) \in N \times [0,\epsilon).
\]
On the other hand, for $v \in T_pN$ and $w \in T_s [0,\epsilon)\simeq \R,$ we have
\[
dF_{(p,s)}(v,w) =  sd(G|_{N \times \{s\}})_{(p,s)}(v) + w\pderiv[F]{s}(p,s).
\]
It follows that $dF_{(p,s)}$ is injective for $(p,s) \in N \times (0,\epsilon),$ so $F|_{N \times (0,\epsilon)}$ is an immersion.

We turn to the case that $N = S^{m-1}$ and $G|_{N \times \{0\}}$ is an embedding. Let $\pi : \R^m \to \orientedprojective(\R^m)$ denote the quotient map. We first show that $\pi \circ G|_{N \times \{s\}}$ is a diffeomorphism for $s \in [0,\epsilon).$ Indeed, since the Euler vector field is nowhere tangent to $G|_{N \times \{s\}},$ it follows that $\pi \circ G|_{N \times \{s\}}$ is a local diffeomorphism. The compactness of $N$ implies that $\pi \circ G|_{N \times \{s\}}$ is a covering map. When $m \geq 3,$ since $S^{m-1}$ is simply connected, $\pi \circ G|_{N \times \{s\}}$ must be a diffeomorphism. On the other hand, when $m = 2,$ since $G|_{N \times \{0\}}$ is an embedding nowhere tangent to the Euler vector field, the winding number about zero is $\pm 1.$ The winding number is preserved under isotopy. So, again, $\pi \circ G|_{N \times \{s\}}$ must be a diffeomorphism.

Next, we show that $F|_{N \times (0,\epsilon)}$ is injective. Define a diffeomorphism $\varphi : N \times [0,\epsilon) \to N \times [0,\epsilon)$ by
\[
\varphi(p,s) = (\pi \circ G|_{N \times \{s\}})^{-1}(\pi \circ G|_{N \times \{0\}}(p,0)).
\]
It suffices to show $F \circ \varphi$ is injective, so we may replace $F$ with $F \circ \varphi$ and $G$ with $G \circ \varphi,$ and thus we may assume $\pi \circ F(p,s)$ is independent of $s.$
Since $\pi \circ G|_{N \times \{s\}}$ is injective, so is $G|_{N \times \{s\}}$ for $s \in [0,\epsilon).$
Thus, we need only show that for $p \in N$ and $s < s' \in [0,\epsilon)$ we have $|F(p,s)| < |F(p,s')|.$ But
\[
\left.\pderiv[|F(p,s)|]{s}\right|_{s = 0^+} = |G(p,0)|.
\]
By compactness of $N,$ after possibly shrinking $\epsilon,$ we have $\pderiv[|F(p,s)|]{s} > 0$ for $(p,s) \in N \times [0,\epsilon).$ So, $F|_{N \times (0,\epsilon)}$ is injective.

To complete the proof, we show that $F|_{N \times (0,\epsilon)}$ is a homeomorphism onto its image. Let $B_\epsilon(0)$ denote the ball of radius $\epsilon$ in $\R^m$ and identify $S^{m-1}$ with the unit sphere in $\R^m.$ Define $\widetilde F : B_{\epsilon}(0) \to \R^m$ by $\widetilde F(sp) = F(p,s).$ Then, $\widetilde F$ is continuous and injective, so for $0 < \epsilon' < \epsilon,$ the restriction of $\widetilde F$ to the compact space $\overline{B_{\epsilon'}(0)}$ is a homeomorphism onto its image. So, we conclude by replacing $\epsilon$ with $\epsilon'.$

The choice of $\epsilon$ throughout the proof depends only on $G$ and its first derivatives, so the conclusions of the lemma continue to hold with uniform $\epsilon$ for any map that is $C^1$ close to~$G.$
\end{proof}

\begin{lemma}\label{lemma:easy regularity}
Let $M$ and $N$ be manifolds and let $q \in M.$ Let $F : N \times [0,\epsilon) \to M$ be a smooth map satisfying the following.
\begin{enumerate}
  \item
$F(p,0) = q$ for all $p \in N.$
  \item
The derivative
\[
\left.\pderiv[F]{s}\right|_{s = 0} : N \to T_qM
\]
is an immersion.
  \item
The Euler vector field on $T_qX$ is nowhere tangent to the immersion $\left.\pderiv[F]{s}\right|_{s = 0}.$
\end{enumerate}
Then, possibly after shrinking $\epsilon,$ the restriction $F|_{N \times (0,\epsilon)}$ is an immersion. Moreover, writing $m = \dim M,$ if $N = S^{m-1},$ and $\left.\pderiv[F]{s}\right|_{s = 0}$ is an embedding, then $F|_{N \times (0,\epsilon)}$ is an embedding.
\end{lemma}
\begin{proof}
This follows by choosing coordinates on $M$ about $q$ and applying Lemmas~\ref{lemma: milnor lemma} and~\ref{lemma: immersion nowhere tangent to Euler}.
\end{proof}
The following is an immediate consequence of Lemma~\ref{lemma:easy regularity}.
\begin{cor}\label{rem: easy regularity}
If $\Phi:S^{n-1}\times[0,1]\times[0,\epsilon)\to X$ satisfies conditions~\eqref{regular convergence to intersection point}\ref{critical point}, \eqref{regular convergence to intersection point}\ref{item:derivative Phi immersion}, \eqref{regular convergence to intersection point}\ref{item:nowhere tangent Euler} and~\eqref{interior regularity first part}\ref{item: rep} of Definition~\ref{definition: interior regularity}, then possibly after diminishing $\epsilon,$ it also satisfies condition~\eqref{regular convergence to intersection point}\ref{item:interior regular} and thus it is a regular parameterization.
\end{cor}

	\begin{rem}
		\label{remark: every cylinder in interior-regular family has regular harmonics}
		It follows from Lemma~\ref{lemma:immersive cylinder has regular harmonics} that if a connected open subset $U\subset\mathcal{SLC}\left(N;\Lambda_0,\Lambda_1\right)$ is interior-regular, every element of $U$ has regular harmonics.
	\end{rem}
	
	\begin{rem}
		\label{remark: regular parameterization gives rise to family of cone-immersions}
		Let $\mathcal{Z},E$ and $q,$ be as in Definition~\ref{definition: interior regularity}~\eqref{regular convergence to intersection point}, and suppose $\Phi$ is a regular parameterization of $E$ about $q.$ Recalling Definition~\ref{definition:polar coordinates}, let
\[
\kappa: S^{n-1}\times[0,\epsilon) \to \R^n
\]
be a polar coordinate map centered at zero with image $U.$ For a fixed $t\in[0,1],$ denote by $\Psi_t: U \to X$ the unique map such that
\[
\Phi|_{S^{n-1}\times\{t\}\times[0,\epsilon)} = \Psi_t \circ \kappa.
\]
Then, Lemma~\ref{lemma:immersion criterion} together with properties~\ref{item:derivative Phi immersion} and~\ref{item:nowhere tangent Euler} of a regular parameterization in Definition~\ref{definition: interior regularity}~(\ref{regular convergence to intersection point}) shows that $\Psi_t$ is cone-immersive at zero. It follows from property~\ref{item:interior regular} of a regular parameterization that $\Psi_t$ is a smooth immersion away from zero. Thus $(\Psi_t,0)$ is a cone-immersion.

This remark is the reason that cone-smooth maps enter the present work. It will be used in Lemma~\ref{lemma: geodesic of small Lagrangians} below, where the cone-immersions $\Psi_t$ parameterize cone-smooth immersed Lagrangians that comprise a geodesic. Lemma~\ref{lemma: geodesic of small Lagrangians} is an important step in showing that the cylindrical transform is surjective as claimed in Theorem~\ref{theorem: geodesic-cylinder correspondence}. Namely, for every regular connected component $\mathcal{Z}\subset\mathcal{SLC}\left(S^{n-1};\Lambda_0,\Lambda_1\right),$ there exists a geodesic between $\Lambda_0$ and $\Lambda_1$ with cylindrical transform $\mathcal{Z}.$ Cone-smooth Lagrangians are necessary to make sure there are enough geodesics for the cylindrical transform to be surjective.
	\end{rem}

Lemma~\ref{lm:chu} below clarifies the meaning of a parameterization being compatible with harmonics as in Definition~\ref{definition: interior regularity}(\ref{compatible with harmonics}). It is used to prove uniqueness of parameterizations under appropriate conditions.
\begin{rem}\label{rm:one slice compatible}
	The notion of compatibility with the harmonics could be defined for a parameterization of a single cylinder $Z \in \mathcal{SLC} \left( N; \Lambda_0, \Lambda_1 \right).$ Indeed, for a parameterization $\Phi: N \times [0, 1] \times (a, b) \to X$ as in Definition~\ref{definition: interior regularity}, the conditions specified in part~(\ref{compatible with harmonics}) make sense for a fixed $s \in (a, b).$
\end{rem}
\begin{lemma}\label{lm:chu}
Let $U\subset\mathcal{SLC}\left(N;\Lambda_0,\Lambda_1\right)$ be interior-regular, let $Z \in U$ be a cylinder, and let $p$ be a point of $Z.$ There exists a unique vector $v \in T_pZ$ such that for any interior regular parameterization of $U$ compatible with harmonics,
\[
\Phi: N \times [0,1] \times (0,1) \to X,
\]
with $[\Phi_s] = Z,$ and $(q,t) \in N \times [0,1]$ such that  $p = [\Phi_s,(q,t)],$ there holds
\[
\pderiv[\Phi]{t}(q,t,s) = v.
\]
\end{lemma}
\begin{proof}
Let $\sigma : Z \to \R$ denote the fundamental harmonic of $Z.$
By condition~\ref{orthogonal} of Definition~\ref{definition: interior regularity}(\ref{compatible with harmonics}) concerning compability with harmonics, $v$ must be orthogonal to the hypersurface $\sigma^{-1}(t) \subset Z$ and thus a multiple of $\nabla \sigma(q,t,s).$ Taking the derivative of condition~\ref{t compatible} of the same definition shows that $\langle\nabla \sigma(q,t,s),v\rangle = 1.$ There exists a unique $v$ satisfying these two constraints.
\end{proof}
	
	\begin{lemma}
		\label{lemma: interior and end regular parameterization compatible with harmonics}
		$\:$
		\begin{enumerate}[label=(\alph*)]
			\item \label{interior regular parameterization compatible with harmonics}
			Let $U\subset\mathcal{SLC}\left(N;\Lambda_0,\Lambda_1\right)$ be interior-regular and let \[
\Phi: N \times [0,1] \times (0,1) \to X
\]
be an interior regular parameterization. Then $U$ admits a unique interior-regular parameterization $\widehat \Phi : N \times [0,1]\times (0,1) \to X$ compatible with its harmonics such that $\widehat \Phi|_{N \times \{0\} \times (0,1)} = \Phi|_{N \times \{0\} \times (0,1)}.$
			\item\label{end regular parameterization compatible with harmonics} Let $\mathcal{Z}\subset\mathcal{SLC}\left(S^{n-1};\Lambda_0,\Lambda_1\right)$ be a connected component and let $E$ be an end of $\mathcal{Z}.$ Let $q\in\Lambda_0\cap\Lambda_1$ be a transverse intersection point, suppose $E$ converges regularly to $q$ and suppose $\Phi : S^{n-1} \times [0,1] \times [0,\epsilon) \to X$ is a regular parameterization of $E$ about $q.$ Then $E$ admits a unique regular parameterization $\widehat \Phi : S^{n-1} \times [0,1] \times [0,\epsilon) \to X$ about $q$ which is compatible with the harmonics of $E$ such that $\widehat \Phi|_{N \times \{0\} \times [0,\epsilon)} = \Phi|_{N \times \{0\} \times [0,\epsilon)}.$
  		\end{enumerate}
	\end{lemma}
	
	\begin{proof}
		We prove~\ref{end regular parameterization compatible with harmonics}. The proof of~\ref{interior regular parameterization compatible with harmonics} is similar and less involved.
		
		We identify a neighborhood of $q$ in $X$ with a ball $V\subset\mathbb{C}^n$ via Darboux coordinates carrying $q$ to the origin and $\Lambda_i,\;i=0,1,$ to linear subspaces. For $s \in [0,1),$ write
\[
V_s:= \{z \in \C^n| s z \in V\}.
\]
For $s\in(0,1)$ let $M_s:V_s\to V$ denote rescaling by $s.$ Define a family of complex structures and $n$-forms on $V_s$ by
		\begin{equation}
		\label{equation:rescaled J and Omega}
		J_s:=M_s^*J,\quad\Omega_s:=s^{-n}M_s^*\Omega,\quad s\in(0,1),
		\end{equation}
		where $J$ and $\Omega$ are the complex structure and Calabi-Yau form of $X,$ respectively. Then as $s\searrow0,$ $J_s$ and $\Omega_s$ converge in the $C^\infty$-topology on compact subsets to a constant complex structure and a constant $n$-form on $V_0 = \C^n$ denoted by $J_0$ and $\Omega_0$ respectively. For $s\in[0,1),$ the quadruple $(V_s,\omega,J_s,\Omega_s)$ is a Calabi-Yau manifold. We let $g_s$ denote the associated K\"ahler metric. We define $\rho_s$ as in~\eqref{equation: rho}.
		
		Let $\Phi:S^{n-1}\times[0,1]\times[0,1)\to V$ be a regular parameterization of $E$ about the origin $q.$ By Lemma~\ref{lemma: milnor lemma}, we have
		\[
		\Phi(p,t,s)=s\cdot\chi(p,t,s),\quad(p,t,s)\in S^{n-1}\times[0,1]\times[0,1),
		\]
		where $\chi:S^{n-1}\times[0,1]\times[0,1)\to \C^n$ is smooth with
		\[
		\chi(p,t,0)=\pderiv[\Phi]{s}(p,t,0).
		\]
		For $s\in[0,1),$ write $\chi_s:=\chi|_{S^{n-1}\times[0,1]\times\{s\}}.$ Then $\chi_s$ is an immersion representing an $\Omega_s$-imaginary special Lagrangian cylinder with boundary components in $\Lambda_i,\;i=0,1.$ Define a smooth family of elliptic linear differential operators
		\[
		\Delta_s:C^\infty\left(S^{n-1}\times[0,1]\right)\to C^\infty\left(S^{n-1}\times[0,1]\right),\quad u\mapsto*_sd((\rho_s\circ\chi_s)*_sdu),
		\]
		where $*_s$ denotes the Hodge star operator of the metric $\chi_s^*g_s.$ Let $\widetilde{\sigma}_s$ denote the fundamental harmonic on $S^{n-1}\times[0,1]$ with respect to $\Delta_s.$ Namely, $\widetilde{\sigma}_s:S^{n-1}\times[0,1]\to\mathbb{R}$ is the unique function satisfying
		\[
		\Delta_s(\widetilde{\sigma}_s)=0,\quad\widetilde{\sigma}_s|_{S^{n-1}\times\{i\}}=i,\quad i=0,1.
		\]
		Note that all the cylinders $[\chi_s]$ have regular harmonics. Indeed, for $s\in(0,1),$ the cylinder $[\chi_s]$ is merely a rescaling of $[\Phi_s],$ which has regular harmonics by Remark~\ref{remark: every cylinder in interior-regular family has regular harmonics}, whereas $[\chi_0]$ has regular harmonics by Lemma~\ref{lemma: cylinder with regular harmonics in Euclidean space} and property~\ref{item:nowhere tangent Euler} of Definition~\ref{definition: interior regularity}~(\ref{regular convergence to intersection point}). Hence, all the functions $\widetilde{\sigma}_s$ have no critical points. By Lemma~\ref{lemma: regular harmonics is open property}, the function $\widetilde{\sigma}_s$ depends smoothly on $s.$ Define a smooth family of vector fields on $S^{n-1}\times[0,1]$ by
		\[
		Y_s:=\frac{\nabla\widetilde{\sigma}_s}{|\nabla\widetilde{\sigma}_s|^2},\quad s\in[0,1),
		\]
		where the gradient and modulus are taken with respect to $\chi_s^*g_s.$ Then $Y_s$ is $\chi_s^*g_s$-orthogonal to the level sets of $\widetilde{\sigma}_s$ and satisfies $Y_s(\widetilde{\sigma}_s)\equiv1.$ Let
\[
\varphi_s : W_s \subset \R \times \left(S^{n-1}\times [0,1]\right) \to S^{n-1}\times [0,1]
\]
denote the flow of $Y_s.$ Note that for $(p,t,s)\in S^{n-1}\times[0,1]\times[0,1)$ we have
		\[
(t,(p,0)) \in W_s, \qquad
		\widetilde{\sigma}_s(\varphi_s(t,(p,0)))=t.
		\]
It follows that also
\[
(-\widetilde{\sigma}_s(p,t),(p,t)) \in W_s, \qquad \varphi_s(-\widetilde{\sigma}_s(p,t),(p,t)) = (p,0).
\]
Thus for every $s$ the map
\[
\mu_s:S^{n-1} \times [0,1] \to S^{n-1}\times [0,1], \qquad (p,t) \mapsto \varphi_s(t,(p,0)),
\]
is a diffeomorphism. Indeed, writing $\pi : S^{n-1} \times [0,1] \to S^{n-1}$ for the projection to the first component, $\mu_s^{-1}$ is given by $(p,t) \mapsto (\pi\circ\varphi_s(-\widetilde{\sigma}_s(p,t),(p,t)),\widetilde{\sigma}_s(p,t)).$
Finally, we set
		\[
		\widehat{\chi}:S^{n-1}\times[0,1]\times[0,1)\to V,\qquad(p,t,s)\mapsto\chi_s\circ \mu_s(p,t),
		\]
		and
		\[
		\widehat{\Phi}:S^{n-1}\times[0,1]\times[0,1)\to V,\qquad(p,t,s)\mapsto s\cdot\widehat{\chi}(p,t,s).
		\]
Since $\mu_s$ is a diffeomorphism, it follows from Definition~\ref{definition: interior regularity}~(\ref{regular convergence to intersection point}) that $\widehat{\Phi}$ is a regular parameterization of $E$ about $q$. Since $\mu_s|_{N \times \{0\}} = \id_{N \times \{0\}},$ it follows that $\widehat \Phi|_{N \times \{0\} \times [0,\epsilon)} = \Phi|_{N \times \{0\} \times [0,\epsilon)}.$ Since $\widetilde{\sigma}_s(\mu_s(p,t))=t,$ and $Y_s$ is orthogonal to the level sets of $\widetilde{\sigma},$ it follows from Definition~\ref{definition: interior regularity}~\eqref{compatible with harmonics} that $\widehat \Phi$ is compatible with the harmonics of $E.$

Suppose $\widehat \Phi'$ is another such regular parameterization. Since
\begin{equation}\label{equation:restriction to t=0}
\widehat \Phi|_{N \times \{0\} \times [0,\epsilon)} = \Phi|_{N \times \{0\} \times [0,\epsilon)} = \widehat \Phi'|_{N \times \{0\} \times [0,\epsilon)},
\end{equation}
it follows from the definition of regular parameterization that $\left[\widehat \Phi'_s\right] = \left[\widehat \Phi_s\right]$ for $s \in (0,\epsilon).$ So, there exists a diffeomorphism $\psi_s : N \times [0,1] \to N \times [0,1]$ such that $\widehat \Phi'_s = \widehat \Phi_s \circ \psi_s.$ By the chain rule,
\begin{equation}\label{eq:chru}
d\widehat\Phi_s\left(\pderiv[\psi_s]{t}\right) = \pderiv[\widehat\Phi_s']{t}.
\end{equation}
On the other hand, by compatibility with harmonics and Lemma~\ref{lm:chu},
\[
\pderiv[\widehat\Phi_s]{t}\circ \psi_s = \pderiv[\widehat\Phi_s']{t}.
\]
That is,
\begin{equation}\label{eq:uniqv}
d\widehat\Phi_s\left(\pderiv{t}\circ \psi_s\right) = \pderiv[\widehat\Phi_s']{t}.
\end{equation}
Since $d\widehat\Phi_s$ is injective, it follows from equations~\eqref{eq:chru} and~\eqref{eq:uniqv} that
\[
\pderiv[\psi_s]{t}(q,t) = \pderiv{t}\circ \psi_s(q,t).
\]
By~\eqref{equation:restriction to t=0}, we have also $\psi_s(q,0) = (q,0).$
Consequently, $\psi_s(q,t) = (q,t)$ and thus $\widehat \Phi' = \widehat \Phi$ proving uniqueness.
\end{proof}
Let $\diff(N \times (0,1) \to (0,1))\subset \diff(N \times (0,1))$ denote the subgroup consisting of diffeomorphisms that carry fibers of the projection $N \times (0,1) \to (0,1)$ to other such fibers.	For $\Phi : N \times [0,1] \times (0,1) \to X$ an interior regular parameterization of $U$ and $\phi \in \diff(N \times (0,1) \to (0,1)),$ write
\[
\Phi_\phi = \Phi \circ \left(\phi \times \id_{[0,1]}\right): N \times [0,1]\times (0,1) \to X.
\]
\begin{cor}\label{corollary:action}
Let $U\subset\mathcal{SLC}\left(N;\Lambda_0,\Lambda_1\right)$ be interior-regular. Then, the map $(\Phi,\phi) \mapsto \Phi_\phi$ gives a free transitive action of the group $\diff(N \times (0,1) \to (0,1))$ on the set of interior regular parameterizations of $U$ compatible with the harmonics.
\end{cor}
\begin{proof}
It follows from the definitions that the action preserves regular parameterizations compatible with the harmonics. Freeness is immediate. Transitivity follows from the uniqueness claim of Lemma~\ref{lemma: interior and end regular parameterization compatible with harmonics}~\ref{interior regular parameterization compatible with harmonics}.
\end{proof}

	\section{The relation between cylinders and geodesics}
	\label{section: the relation between cylinders and geodesics}

	In this section we establish the relation between geodesics of positive Lagrangians and families of imaginary special Lagrangian cylinders. The section starts with the proof of Theorem~\ref{theorem: family} divided into two lemmas. Lemma~\ref{geodesic yields unique cylinder} shows how imaginary special Lagrangian cylinders arise from geodesics of positive Lagrangian submanifolds. Lemma~\ref{lemma: time function} identifies the fundamental harmonics of the cylinders with the time parameter of the geodesic.

The remainder of the section is devoted to the proof of Theorem~\ref{theorem: geodesic-cylinder correspondence}, which establishes a bijective correspondence between geodesics of positive Lagrangian spheres and regular components of the space of imaginary special Lagrangian cylinders. According to Definition~\ref{definition: regularity} and Remark~\ref{remark: regular connected component}, a regular component is one that admits an interior regular parameterization and regular parameterizations about intersection points at each end. The map from regular components to geodesics is constructed in two steps. Proposition~\ref{proposition: interior-regular family gives rise to geodesic} constructs a geodesic of open positive Lagrangians from an interior regular parameterization.
Lemma~\ref{lemma: geodesic of small Lagrangians} constructs a geodesic of open positive Lagrangians from a regular parameterization about an intersection point. These geodesics of open positive Lagrangians glue together to give a geodesic of Lagrangian spheres. Conversely, it is shown in two steps that the cylindrical transform maps geodesics of positive Lagrangian spheres to regular components of the space of imaginary special Lagrangian cylinders. Lemma~\ref{lemma: non-singular cylindrical transform} constructs the requisite interior regular parameterization, while Lemma~\ref{lemma:rocking lemma} constructs the regular parameterizations about intersection points.

We fix an ambient Calabi-Yau manifold $(X,\omega,J,\Omega),$ a smooth manifold $L,$ a finite subset $S \subset L$ and a finite subset $C_0 \subset X.$ In the following, all geodesics are geodesics of cone-immersed Lagrangians of type $(L,S)$ with cone locus image $C_0$ unless otherwise mentioned.

	\begin{lemma}
		\label{geodesic yields unique cylinder}
		Let $(\Lambda_t)_{t\in[0,1]}$ be a geodesic of positive Lagrangians of type $(L,S)$, and let $\Psi_t:L\to\Lambda_t,\;t\in[0,1],$ be a horizontal lifting of $(\Lambda_t)$. Let $(h_t)_t$ be the Hamiltonian of $(\Lambda_t)_t$ and let $h = h_t\circ \Psi_t:L\to\mathbb{R}$ be the Hamiltonian of $(\Lambda_t)_t$ with respect to the lifting $(\Psi_t)_t$.
		For $c$ in the image of $h,$ define
		\[
\Phi_c:\left(h^{-1}(c)\setminus\crit(h)\right)\times[0,1]\to X,\qquad (p,t)\mapsto\Psi_t(p).
		\]
		Then $\Phi_c$ is an imaginary special Lagrangian immersion.
	\end{lemma}

	\begin{proof}
		Fix $t_0\in[0,1]$ and $p\in h^{-1}(c)\setminus\crit(h)$. Let $u_1,\ldots,u_{n-1}\in T_ph^{-1}(c)$ be a basis and write
		\[
		w_i:=d(\Psi_{t_0})_p(u_i)\in T_{\Psi_{t_0}(p)}\Lambda_{t_0},\qquad i=1,\ldots,n-1.
		\]
		The tangent vectors $w_1,\ldots,w_{n-1}$ are linearly independent as $\Psi_{t_0}$ is a smooth immersion away from critical points of $h$. Let $v\in T_{\Psi_{t_0}(p)}X$ be the $t$-derivative of $\Phi_c$. That is,
		\[
		v:=\pderiv[\Phi_c]{t}(p,t_0)=\deriv{t}\Psi_t(p).
		\]
		Since $p$ is a regular point of $h$, we have $v\not\in T_{\Psi_{t_0}(p)}\Lambda_{t_0}$. It follows that the tangent vectors $v,w_1,\ldots,w_{n-1},$ are linearly independent and $\Phi_c$ is indeed an immersion.
		
		As the immersion $\Psi_{t_0}$ is Lagrangian, we have
		\[
		\omega(w_i,w_j)=0,\qquad i,j=1,\ldots,n-1.
		\]
		By definition of the derivative of a Lagrangian path (see Remark~\ref{remark: tangent space hands on}), we have
		\[
		\omega(v,w_i)=dh_p(u_i)=0,\qquad i=1,\ldots,n-1.
		\]
		The immersion $\Phi_c$ is thus Lagrangian. Finally, by horizontality of the lifting $(\Psi_t)$, we have
		\[
		\real\Omega(v,w_1,\ldots,w_{n-1})=0,
		\]
		and $\Phi_c$ is indeed an imaginary special Lagrangian immersion.
	\end{proof}

	\begin{dfn}
		\label{definition: associated cylinders}
		In the setting of Lemma~\ref{geodesic yields unique cylinder}, for $c$ in the image of the Hamiltonian $h,$ we call the immersed special Lagrangian cylinder $[\Phi_c]$ the \emph{associated cylinder of $c$ level sets.} Note that $[\Phi_c]$ is independent of the horizontal lifting $(\Psi_t)_t.$ It depends only on the choice of additive constant in the intrinsic Hamiltonian $(h_t)_t.$
	\end{dfn}

\begin{rem}\label{rem:compact cylinder}
In the setting of Lemma~\ref{geodesic yields unique cylinder}, for $c$ in the image of the Hamiltonian~$h,$ the associated cylinder of $c$ level sets $Z_c$ is compact and non-empty only if $c$ is a regular value. Indeed,
$Z_c$ is represented by the immersion
\[
\Phi_c:\left(h^{-1}(c)\setminus\crit(h)\right)\times[0,1]\to X,
\]
and $h^{-1}(c)\setminus \crit(h)$ is compact and non-empty only if $h^{-1}(c) \cap \crit(h) = \emptyset.$
\end{rem}

	\begin{lemma}
		\label{lemma: time function}
		Let $(\Lambda_t)_t,(\Psi_t)_t,h$ and $(\Phi_c)_c,$ be as in Lemma~\ref{geodesic yields unique cylinder}. For $c$ in the image of $h,$ let $Z_c$ denote the associated cylinder $[\Phi_c]$ and define a function
		\[
		\sigma_c:Z_c\to[0,1],
\]
by
\[
\sigma_c(\Phi_c(q,t))=t.
		\]
		\begin{enumerate}[label=(\alph*)]
			\item\label{time function is the derivative} Let $N$ be an $(n-1)$-dimensional smooth manifold, let $c_0<c_1\in\mathbb{R}$ and let $\beta:N\times(c_0,c_1)\to L$ be a smooth embedding with
			\begin{equation}
			\label{local level slicing}
			h(\beta(q,c))=c,\quad(q,c)\in N\times(c_0,c_1).
			\end{equation}
			For $c\in(c_0,c_1)$ define a smooth immersion
			\[
			\phi_c:N\times[0,1]\to X,\quad(q,t)\mapsto\Phi_c(\beta(q,c),t)=\Psi_t(\beta(q,c)),
			\]
and let
\[
v = \deriv{c}\phi_c.
\]
Then,
\[
i_v \omega = -d(\sigma_c\circ \phi_c).
\]
\item\label{time function is harmonic} For $c$ in the image of $h$ we have $\Delta_\rho(\sigma_c)=0.$
		\end{enumerate}
	\end{lemma}

	\begin{proof}
		Let $N$ and $\beta$ be as in~\ref{time function is the derivative}.
		Let $t$ denote the $[0,1]$-coordinate on $N\times[0,1].$ By Remark~\ref{remark: tangent space hands on} and the definition of a geodesic, we have
		\begin{align}
		\label{same time derivative}
i_v\omega\left(\pderiv{t}\right) &=-\omega\left(\deriv{t}\Psi_t,\deriv{c}(\Psi_t\circ\beta)\right)\\
		&=-\deriv{c}(h\circ\beta)\nonumber\\
        &=-1 \nonumber \\
		&=-\pderiv[(\sigma_c\circ\phi_c)]{t}.\nonumber
		\end{align}
Let $w$ be a vector field on $N \times [0,1]$ tangent to $N.$ Then, since $\Psi_t$ is a Lagrangian immersion,
\begin{align}\label{same tangent derivative}
i_v\omega\left(w\right) &=-\omega\left(d\Psi_t(w),\deriv{c}(\Psi_t\circ\beta)\right)\\
		&=0\nonumber\\
		&= -w(\sigma_c\circ\phi_c).\nonumber
\end{align}
		Part~\ref{time function is the derivative} now follows from~\eqref{same time derivative} and~\eqref{same tangent derivative}.
		
		Let $p\in L$ be a regular point of $h$ and let $c_0<c_1$ with $h(p)\in(c_0,c_1).$ Let $B$ denote the standard open ball of dimension $n-1.$ Let $p\in W\subset L$ be a ball containing only regular points of $h,$ and let $\beta:B\times(c_0,c_1)\to W$ be a diffeomorphism with
		\[
		h\circ\beta(q,c)=c,\quad(q,c)\in B\times(c_0,c_1).
		\]
		For $c\in(c_0,c_1),$ define $\phi_c$ as in~\ref{time function is the derivative}. Part~\ref{time function is harmonic} now follows from~\ref{time function is the derivative}, Lemma~\ref{lemma: linearized operator}~\ref{more general lemma} and Lemma~\ref{geodesic yields unique cylinder}.
	\end{proof}

	\begin{rem}
		\label{remark: time function is the fundamental harmonic}
		Let $(\Lambda_t)_{t\in[0,1]},\,(h_t)_t,\,(Z_c)_c,$ and $(\sigma_c)_c,$ be as in Lemma~\ref{lemma: time function}. Let $c$ be such that $Z_c$ is compact and non-empty. In particular, $c$ must be a regular value of $(h_t)_t$ by Remark~\ref{rem:compact cylinder}. It follows from Lemma~\ref{lemma: time function}~\ref{time function is harmonic} that $\sigma_c$ is the fundamental harmonic on $Z_c.$ Also, by its definition, the function $\sigma_c$ has no critical points. The cylinder $Z_c$ thus has regular harmonics.
	\end{rem}

	\begin{proof}[Proof of Theorem~\ref{theorem: family}]
		Follows from Lemmas~\ref{geodesic yields unique cylinder} and~\ref{lemma: time function}.
	\end{proof}

	Let $(\Lambda_t)_t$ be a geodesic of positive Lagrangians with Hamiltonian $(h_t:\Lambda_t\to\mathbb{R})_t.$ For $c$ in the image of $h_t,$ let $Z_c$ denote the associated cylinder of $c$ level sets. Recall from Definition~\ref{definition: cylindrical transform} that the family of cylinders $(Z_c)_c$ is called the \emph{cylindrical transform} of the geodesic $(\Lambda_t)_t$. The subset $\{Z_c\;|\;c\;\mathrm{is\;a\;regular\;value\;of}\; h\}$ is the \emph{non-singular cylindrical transform}. We show that every component in the non-singular cylindrical transform is interior-regular with relative Lagrangian flux given by the Hamiltonian of the geodesic.

	\begin{lemma}
		\label{lemma: non-singular cylindrical transform}
		Let $(\Lambda_t)_t$ be a geodesic with Hamiltonian $(h_t)_t$ and assume the functions $h_t$ are proper and $\Lambda_0,\Lambda_1,$ are embedded. For $c$ in the image of $h_t,$ let $Z_c$ denote the associated cylinder of $c$ level sets. Since the functions $h_t$ are proper, the cylinder $Z_c$ is compact when $c$ is a regular value of $(h_t)_t.$ Let $c_0<c_1\in\mathbb{R}$ be such that the interval $(c_0,c_1)$ consists of regular values of $(h_t)_t.$
		\begin{enumerate}[label=(\alph*)]
			\item\label{cylindrical transform is interior-regular} The family of cylinders $U:=\{Z_c\;|\;c\in(c_0,c_1)\}$ is interior-regular.
			\item\label{flux equals Hamiltonian} For $b_0<b_1\in(c_0,c_1)$ we have
			\[
			\mathrm{RelFlux}\left((Z_c)_{c\in[b_0,b_1]}\right)=b_1-b_0.
			\]
		\end{enumerate}
	\end{lemma}

	\begin{proof}
		Let $(\Psi_t:L\to X)_t$ be a horizontal lifting of $(\Lambda_t)_t,$ and let $h:L\to\mathbb{R}$ be the corresponding Hamiltonian,
		\[
		h_t=[(\Psi_t,h)].
		\]
		By assumption there exist an $(n-1)$-dimensional smooth manifold $N$ and a diffeomorphism
		\[
		\beta:N\times(c_0,c_1)\to h^{-1}(c_0,c_1)
		\]
		with
		\[
		h\circ\beta(q,c)=c,\quad (q,c)\in N\times(c_0,c_1).
		\]
		Define
		\[
		\Phi:N\times[0,1]\times(c_0,c_1)\to X,\quad(q,t,c)\mapsto\Psi_t(\beta(q,c)).
		\]
Since $\Lambda_i$ is embedded, $\Phi|_{N\times\{i\}\times(a,b)}$ is an embedding for $i = 0,1.$
It follows that $\Phi$ is an interior-regular parameterization of $U$ verifying~\ref{cylindrical transform is interior-regular}. Part~\ref{flux equals Hamiltonian} follows from Lemma~\ref{lemma: time function}~\ref{time function is the derivative} and Definition~\ref{definition:relative Lagrangian flux}.
	\end{proof}

	Proposition~\ref{proposition: interior-regular family gives rise to geodesic} below is a converse to Lemma~\ref{lemma: non-singular cylindrical transform}~\ref{cylindrical transform is interior-regular} showing that an interior-regular family of special Lagrangian cylinders gives rise to a geodesic of positive Lagrangians. We use the following notion.
	
	\begin{dfn}
		\label{definition: swept submanifold}
		Let $\Lambda_0,\Lambda_1\subset X$ be smooth Lagrangians and let $N$ be a closed connected smooth manifold of dimension $n-1.$ Let $U\subset\mathcal{SLC}(N;\Lambda_0,\Lambda_1)$ be open, connected and interior-regular. Let $\Phi:N\times[0,1]\times(0,1)\to X$ be an interior regular parameterization of $U.$ For $i=0,1,$ the submanifold of $\Lambda_i$ \emph{swept by $U$} is the image of the embedding $\Phi|_{N\times\{i\}\times(0,1)}.$ This is independent of $\Phi.$ Similarly, suppose $U\subset\mathcal{SLC}\left(S^{n-1};\Lambda_0,\Lambda_1\right)$ is a ray and $\Phi : S^{n-1}\times [0,1]\times [0,\epsilon) \to X$ is a regular parameterization of $U$ about an intersection point $q \in \Lambda_0 \cap \Lambda_1.$ For $i = 0,1,$ the \emph{unpunctured} submanifold of $\Lambda_i$ swept by $U$ is the image of the restricted map~$\Phi|_{S^{n-1}\times\{i\}\times[0,\epsilon)}.$
	\end{dfn}

	\begin{prop}
		\label{proposition: interior-regular family gives rise to geodesic}
		Let $\Lambda_0,\Lambda_1\subset X$ be smoothly embedded positive Lagrangians, let $N$ be a closed connected smooth manifold of dimension $n-1$ and let
\[
U\subset\mathcal{SLC}(N;\Lambda_0,\Lambda_1)
\]
be open, connected and interior regular. For $i=0,1,$ let $\Lambda_i^U$ denote the submanifold of $\Lambda_i$ swept by $U.$ Then, there exists a geodesic of positive Lagrangians between $\Lambda_0^U$ and $\Lambda_1^U$ with cylindrical transform $U.$ This geodesic is unique up to reparameterization and has empty critical locus.
	\end{prop}

	\begin{proof}
Let $\Phi:N\times[0,1]\times(0,1)\to X$ be an interior-regular parameterization of~$U.$ Since the points of $U$ are imaginary special Lagrangians, $N$ must be orientable. Orient $N$ so that the open embedding
\[
\Phi|_{N \times \{0\}\times(0,1)} : N \times \{0\} \times (0,1) \to \Lambda_0
\]
is orientation preserving. By Lemma~\ref{lemma: interior and end regular parameterization compatible with harmonics}, we may assume that $\Phi$ is compatible with the harmonics of $U.$ For $s\in(0,1)$ write $\Phi_s:=\Phi|_{N\times[0,1]\times\{s\}}$ and $Z_s:=[\Phi_s].$ Let $\sigma_s$ denote the fundamental harmonic on $Z_s.$ Let $\xi$ denote the $t$-derivative of $\Phi,$
		\[
		\xi=\xi(p,t,s)=\pderiv{t}\Phi(p,t,s).
		\]
		For $t\in[0,1]$ write
		\[
		\Psi_t:=\Phi|_{N\times\{t\}\times(0,1)}.
		\]
		Then $\Psi_t$ is a smooth immersion by the definition of an interior regular parameterization. Let $\Lambda_t^U:=[\Psi_t]$ with orientation given by the orientation on $N \times (0,1).$ We show in three steps that $\left(\Lambda_t^U\right)_{t\in[0,1]}$ is a geodesic of positive Lagrangians.
		
		\textbf{Step 1:} \emph{For $t\in[0,1],$ the immersed submanifold $\Lambda_t^U$ is positive Lagrangian.} To show this, pick $(p,t,s)\in N\times[0,1]\times(0,1).$ As $\Phi$ is compatible with the harmonics of $U,$ the immersed submanifold
		\[
		K_{t,s}:=\left[\Phi|_{N\times\{t\}\times\{s\}}\right]
		\]
		is the $t$-level set of $\sigma_{s}$ in the Lagrangian cylinder $Z_s.$ In particular, $K_{t,s}$ is $\omega$-isotropic. By Proposition~\ref{proposition: space of special lag cylinders is one dimensional}, there exists $a(s)\in\R$ such that
		\begin{equation}
		\label{specific harmonic}
		\deriv{s}Z_s=a(s)\sigma_{s}.
		\end{equation}
		Hence, by Remark~\ref{remark: tangent space hands on}, for $v\in T_{\Psi_t(p,s)}K_{t,s}$ we have
		\begin{align*}
		\omega\left(\pderiv{s}\Psi_t(p,s),v\right)
		&=a(s)d(\sigma_s)_{\Psi_t(p,s)}(v)\\
		&=0,
		\end{align*}
		showing that $\Lambda_t^U$ is Lagrangian.

It remains to establish positivity. Fix a basis $v_1,\ldots,v_{n-1}\in T_{\Psi_t(p,s)}K_{t,s}.$ Let $0\ne w\in T_{\Psi_t(p,s)}\Lambda_t^U$ be orthogonal to $K_{t,s}.$ Define $E_{p,t,s}\subset T_{\Psi_t(p,s)}X$ by
		\begin{align*}
		E_{p,t,s}:
		&=(T_{\Psi_t(p,s)}K_{t,s})^g\cap(T_{\Psi_t(p,s)}K_{t,s})^\omega\\
		&=\{u\in T_{\Psi_t(p,s)}X\;|\;g(u,v_i)=0 = \omega(u,v_i),\;i=1,\ldots,n-1\}.
		\end{align*}
		Then, $E_{p,t,s}$ is a $J$-invariant linear subspace of real dimension $2.$ It follows from Definition~\ref{definition: interior regularity}\eqref{compatible with harmonics}\ref{orthogonal} and the fact that $Z_s$ is Lagrangian that $\xi(p,t,s) \in E_{p,t,s}.$ Since $\Lambda_t^U$ is Lagrangian, it follows that $w\in E_{p,t,s}.$ Since $\Phi$ is an immersion, $\xi\ne0$, so
		\begin{equation}
		\label{complex dependence}
		w=\alpha\xi+\beta J\xi
		\end{equation}
		for some $\alpha,\beta\in\R.$ Also, as $\Phi$ is an immersion, we have $\beta\ne0.$ By Lemma~\ref{lemma: Omega does not vanish on Lags} and since $Z_s$ is imaginary special Lagrangian, we have
		\[
		0\ne\Omega(v_1,\ldots,v_{n-1},\xi)\in\ii\R.
		\]
		Finally, as $\Omega$ is of type $(n,0),$ we have
		\begin{align*}
		\real\Omega(v_1,\ldots,v_{n-1},w)
		&=\real\left(\left(\alpha+\ii\beta\right)\Omega(v_1,\ldots,v_{n-1},\xi)\right)\\
		&=\ii\beta\Omega(v_1,\ldots,v_{n-1},\xi)\\
		&\ne0.
		\end{align*}
Since $\Lambda_0^U$ is positive by assumption, the positivity of $\Lambda_t^U$ follows by continuity.
		
		\textbf{Step 2:} \emph{The family of immersions $(\Psi_t)_{t\in[0,1]}$ is horizontal.} To prove this, we need to verify
		\begin{equation}
		\label{horizontal 1}
		\real\Omega(\xi,v_1,\ldots,v_{n-1})=0
		\end{equation}
		and
		\begin{equation}
		\label{horizontal 2}
		\real\Omega\left(\xi,v_1,\ldots\widehat{v}_i\ldots,v_{n-1},w\right)=0,\quad i=1,\ldots,n-1,
		\end{equation}
		where $v_1,\ldots,v_{n-1}$ and $w$ are as above. Equality~\eqref{horizontal 1} holds true as all the cylinders in $U$ are imaginary special Lagrangian. Equality~\eqref{horizontal 2} holds true by virtue of equality~\eqref{complex dependence}, as $\Omega$ is of type $(n,0).$ This completes the proof of Step 2.
		
		\textbf{Step 3:} \emph{Pick $s_0\in(0,1)$. Define a family of functions $h_t^U \in C^\infty(\Lambda_t^U)$ by
			\[
			h_t^U\circ\Psi_t(p,s_1)=\mathrm{RelFlux}\left((Z_s)_{s\in[s_0,s_1]}\right),\quad s_1\in(0,1).
			\]
			In particular, $h_t^U\circ\Psi_t(p,s_1)$ is independent of $t$ and the point $p\in N.$ Then,
\[
\deriv{t}\Lambda_t^U = dh_t^U.
\]
}

To show this, recall the function $a:(0,1)\to\R$ defined in equality~\eqref{specific harmonic}. By Definition~\ref{definition:relative Lagrangian flux} we have
		\begin{equation*}
		\label{derivative of RelFlux}
		\deriv{s_1}\mathrm{RelFlux}\left((Z_s)_{s\in[s_0,s_1]}\right)=-a(s_1).
		\end{equation*}
So, recalling from Definition~\ref{definition: interior regularity}\eqref{compatible with harmonics}\ref{t compatible} the meaning of $\Phi$ being compatible with the harmonics of $U,$
		\begin{align*}
		\pderiv{s}h_t^U\circ\Psi_t(p,s)
		&=-a(s)		\\
		&=-a(s)d(\sigma_s)_{\Psi_t(p,s)}\left(\xi(p,t,s)\right)\\
&=-\omega\left(\pderiv{s}\Psi_t(p,s),\xi(p,t,s)\right)\\
&=\omega\left(\xi(p,t,s),\pderiv{s}\Psi_t(p,s)\right),
		\end{align*}
On the other hand, for $v$ a vector field on $N \times [0,1]$ tangent to $N,$ since $h_t^U\circ\Psi_t(p,s)$ is independent of $p,$ we have
\begin{equation*}
v\left(h_t^U\circ\Psi_t\right) = 0 = \omega(\xi, d\Psi_t(v)).
\end{equation*}
By Remark~\ref{remark: tangent space hands on}, it follows that $\deriv{t}\Lambda^U_t = dh_t^U.$ This completes the proof of Step 3.

Since $h_t^U\circ \Psi_t$ is independent of $t,$ it follows that $\left(\Lambda_t^U\right)_{t\in[0,1]}$ is a geodesic. By construction, the cylindrical transform of this geodesic is $U.$ Finally, since by Definition~\ref{definition: interior regularity}~\eqref{interior regularity first part} the parameterization $\Phi$ is an immersion, the time derivative $\deriv{t}\Psi_t=\pderiv{t}\Phi = \xi$ is nowhere vanishing, implying that $\left(\Lambda_t^U\right)_t$ has empty critical locus.

It remains to prove uniqueness. Let $\left(\Lambda_t^{U'}\right)_{t\in[0,1]}$ be another geodesic with cylindrical transform $U.$ Let $(\Psi_t' : N \times (0,1) \to X)_t$ be a horizontal lifting with $\Psi_0' = \Psi_0,$ and let $h' : N \times (0,1) \to \R$ denote the associated Hamiltonian. Let $h$ denote the Hamiltonian of $\left(\Lambda_t^U\right)_t$ with respect to the horizontal lifting $(\Psi_t)_t.$ By Lemma~\ref{lemma: non-singular cylindrical transform}~\ref{flux equals Hamiltonian}, after possibly adding a constant to $h'$, we have $h' = h.$ Let $t_0 \in [0,1],$ let $c$ be a value of $h,$ and let $s \in (0,1)$ be such that $h^{-1}(c) = N \times \{s\}.$ Since the cylindrical transforms of $\left(\Lambda_t^U\right)_t$ and $\left(\Lambda_t^{U'}\right)_t$ coincide, the cylinder of $c$ level sets $Z_c$ associated with $\left(\Lambda_t^U\right)_t$ coincides with the cylinder of $c$ level sets $Z_c'$ associated with $\left(\Lambda_t^{U'}\right)_t.$ By Remark~\ref{remark: time function is the fundamental harmonic}, the immersed submanifolds $\left[\Psi_{t_0}|_{N \times \{s\}}\right]$ and $\left[\Psi_{t_0}'|_{N \times \{s\}}\right]$ both coincide with the $t_0$ level set of the fundamental harmonic of the cylinder $Z_c = Z_c'.$ Since $c$ was arbitrary, it follows that $\Lambda_{t_0}^U = \Lambda_{t_0}^{U'}.$ Since $t_0$ was arbitrary, the claim follows.
	\end{proof}

	\begin{lemma}
		\label{lemma: geodesic of small Lagrangians}
		Let $\Lambda_0,\Lambda_1\subset X$ be smoothly embedded positive Lagrangians intersecting transversally at a point $q.$ Suppose there exists a connected component $\mathcal{Z}\subset\mathcal{SLC}\left(S^{n-1};\Lambda_0,\Lambda_1\right)$ with an end $E$ converging regularly to $q.$ Let $U \subset \mathcal{Z}$ be a ray representing $E$ admitting a regular parameterization about $q.$ For $i = 0,1,$ let $\Lambda_i^U$ denote the unpunctured submanifold of $\Lambda_i$ swept by $U$. Then, there exists a geodesic of positive Lagrangians between $\Lambda_0^U$ and $\Lambda_1^U$ with cylindrical transform~$U.$ This geodesic is unique up to reparameterization and has critical locus $\{q\}.$
	\end{lemma}

	\begin{proof}
		Let $\Phi:S^{n-1}\times[0,1]\times[0,\epsilon)\to X$ be a regular parameterization of $U$ about $q$ compatible with the harmonics. For $t\in[0,1],$ define
		\[
		\widetilde{\Psi}_t:S^{n-1}\times[0,\epsilon)\to X,\qquad(c,s)\mapsto\Phi(c,t,s).
		\]
Recalling Definition~\ref{definition:polar coordinates}, let
\[
\kappa: S^{n-1}\times[0,\epsilon) \to \R^n
\]
be a polar coordinate map centered at zero with image $W.$
		For $t\in[0,1],$ let
\[
\Psi_t:W\to X
\]
be the unique map such that
\[
		\widetilde{\Psi}_t=\Psi_t\circ\kappa.
\]
		By Remark~\ref{remark: regular parameterization gives rise to family of cone-immersions}, the pair $(\Psi_t,0)$ is a cone-immersion for $t\in[0,1].$ Write
\[
\Psi_t^\circ:=\Psi_t|_{W\setminus\{0\}},\qquad \Lambda_t^U:=[(\Psi_t,0)], \qquad \Lambda_t^{U^\circ}:=\left[\Psi_t^\circ\right].
\]
By Proposition~\ref{proposition: interior-regular family gives rise to geodesic}, the family $\left(\Lambda_t^{U^\circ}\right)_{t\in[0,1]}$ is a geodesic with empty critical locus and cylindrical transform $U.$ Recalling the meaning of a regular parameterization about $q$ from Definition~\ref{definition: interior regularity}\eqref{regular convergence to intersection point}\ref{critical point}, it follows that $\Psi_t(0) = q$ for $t \in [0,1].$ Thus $\Lambda_t^U \in \mathcal{L}(X,W;\{0\},\{q\})$ as in Definition~\ref{definition: geodesic}. It follows from Lemma~\ref{lemma:extension of geodesic to cone point} that $\left(\Lambda_t^U\right)_t$ is a geodesic with critical locus $\{q\}.$ The uniqueness of $\left(\Lambda_t^U\right)_t$ follows from the uniqueness of $\left(\Lambda_t^{U^\circ}\right)_{t\in[0,1]}$ given by Proposition~\ref{proposition: interior-regular family gives rise to geodesic}.
	\end{proof}

	The following lemmas are preliminary to the proof of Theorem~\ref{theorem: geodesic-cylinder correspondence}. We first establish relevant notation. Let $(\Lambda_t)_{t\in[0,1]} \in \mathcal{L}(X,L;S,C_0)$ be a geodesic. Let $(h_t:\Lambda_t\to\mathbb{R})_{t\in[0,1]}$ denote the Hamiltonian of $(\Lambda_t)_t.$ Let $\pi:\widetilde{L}_S\to L$ denote the blowup projection. Let $((\Psi_t:L\to X,S))_{t\in[0,1]}$ be a horizontal lifting of $(\Lambda_t)_t,$ and let $h:L\to\mathbb{R}$ be the Hamiltonian of $(\Lambda_t)_t$ with respect to $((\Psi_t,S))_t.$ That is,
		\[
		h_t=[((\Psi_t,S),h)].
		\]
		Then $(h,S)$ is cone-smooth and $S$ is contained in its critical locus by Lemma~\ref{lemma:derivative 1-form of cone smooth path}. Recalling Definition~\ref{definition: cone-smooth differential form etc}, Remark~\ref{remark:differential of cone-smooth function}, Lemma~\ref{lemma:cone-smooth pullback} and notation~\eqref{eq:pullbackabbrev}, let $\nabla^t h$ denote the gradient of $h$ with respect to the pull-back metric $\Psi_t^*g.$ So, $\nabla^t h$ is a cone-smooth vector field on~$(L,S).$  Composing with the differential $d\Psi_t,$ we consider $\nabla^th$ as a cone-smooth vector field along $\Psi_t.$ Recall that we denote by $\widetilde{\nabla^t h}$ the blowup vector field, which is a section of $\pi^*\Psi_t^*TX.$ Let $\theta_t = \theta_{\Psi_t}:\widetilde{L}_S \to S^1$ denote the phase function.

\begin{lemma}\label{lemma: time derivative of cone derivative}
Let $p_0 \in S,$ let $0\neq v \in T_{p_0}L$ and let $\widetilde p: = [v] \in E_{p_0}.$ Then,
\[
\deriv{t}d\Psi_t(v) =-J\nabla_v\widetilde{\nabla^th}-\tan(\theta_t(\widetilde p))\nabla_v\widetilde{\nabla^th},
\]
where $\nabla$ is any connection on $TX.$
\end{lemma}
\begin{proof}
  By~\cite[Remark~5.6]{solomon}, we have
		\begin{equation}
		\label{horizontal vector formula}
		\deriv{t}\Psi_t\circ\pi(p)=-J\widetilde{\nabla^th}(p)-\tan(\theta_t(p))\widetilde{\nabla^th}(p),\quad p\in\widetilde{L}_S,t\in[0,1].
		\end{equation}
Let $s,x^1,\ldots,x^{n-1},$ be local cone-coordinates on $L$ around $\widetilde{p}$ with $\pderiv[\pi]{s}|_{\widetilde p}=v.$ As $p_0$ is a critical point of $h$ and by Remark~\ref{remark: critical cone point}, we have
		\[
		\widetilde{\nabla^th}\left(\widetilde{p}\right)=0,\quad t\in[0,1].
		\]
		Hence, covariantly differentiating~\eqref{horizontal vector formula} at $\widetilde{p}$ in the direction of $v$, we obtain
		\begin{align*}
\deriv{t}d\Psi_t(v) &=
		\frac{D}{\partial t}\frac{\partial\Psi_t}{\partial s}\left(\widetilde{p}\right) \\
		&=\frac{D}{\partial s}\frac{\partial\Psi_t}{\partial t}\left(\widetilde{p}\right)\\
		&=\nabla_v\left(-J\widetilde{\nabla^th}-\tan(\theta_t(\widetilde p))\widetilde{\nabla^th}\right)\\
		&=-J\nabla_v\widetilde{\nabla^th}-\tan(\theta_t(\widetilde p))\nabla_v\widetilde{\nabla^th}.
		\end{align*}
\end{proof}	
	
	\begin{lemma}
		\label{lemma: critical points are non-degenerate}
		Suppose $\Lambda_0$ and $\Lambda_1$ intersect transversally at $q \in C_0.$  Then, $q$ is a non-degenerate critical point of $h_t,\, t \in [0,1].$
	\end{lemma}
	
	\begin{proof}
Let $p \in S$ with $\Psi_t(p) = q$ for $t \in [0,1].$ By way of contradiction, suppose that $p$ is a degenerate critical point of $h$ and let $0\ne v\in T_{p}L$ with $\nabla_vdh=0.$ It follows from Lemma~\ref{lemma: time derivative of cone derivative} that
		\[
		d\Psi_0\left(v\right)=d\Psi_1\left(v\right)
		\]
		contradicting the transversality of $\Lambda_0$ and $\Lambda_1.$
	\end{proof}
Suppose $\Lambda_0$ and $\Lambda_1$ intersect transversally at a single point, so $C_0 = \{q\}.$ Moreover, for $t \in [0,1]$ assume $q$ is an absolute minimum or maximum of $h_t$. Write $S = \{p\},$ so $\Psi_t(p) = q.$ By Lemma~\ref{lemma: critical points are non-degenerate}, $p$ is a non-degenerate critical point of $h.$ By Lemma~\ref{lemma: nice level sets},
choose a positive $\epsilon$ and a polar coordinate map $\kappa:S^{n-1}\times[0,\epsilon)\to L$ centered at $p$ such that
for each $s\in(0,\epsilon)$ the restricted map $\kappa|_{S^{n-1}\times\{s\}}$ parameterizes a level set of $h.$
\begin{lemma}\label{lemma:rocking lemma}
Let $\mathcal{Z}$ denote the cylindrical transform of $(\Lambda_t)_t.$ Then, one end of $\mathcal{Z}$ converges regularly to $q.$
In fact, the map
		\[
		\Phi:S^{n-1}\times[0,1]\times[0,\epsilon)\to X,\qquad(c,t,s)\mapsto\Psi_t(\kappa(c,s))
		\]
is a regular parameterization of $\mathcal{Z}$ about $q.$
\end{lemma}
\begin{proof}
We verify the conditions of Definition~\ref{definition: interior regularity}~\eqref{regular convergence to intersection point} one by one. Condition~\ref{critical point} is immediate. By Lemma~\ref{lemma: non-singular cylindrical transform}~\ref{cylindrical transform is interior-regular}, $\Phi|_{S^{n-1}\times [0,1]\times (0,\epsilon)}$ is interior regular verifying condition~\ref{item:interior regular}. To verify conditions~\ref{item:derivative Phi immersion} and~\ref{item:nowhere tangent Euler}, we show that
\[
\left.\pderiv[\Phi]{s}\right|_{S^{n-1} \times [0,1]\times \{0\}} : S^{n-1} \times [0,1]\times \{0\} \to T_qX
\]
is an immersion nowhere tangent to the Euler vector field. Indeed, by Lemma~\ref{lemma:immersion criterion}, for $t \in [0,1]$ the restriction
\[
\left.\pderiv[\Phi]{s}\right|_{S^{n-1} \times \{t\}\times \{0\}} : S^{n-1} \times \{t\}\times \{0\} \to T_qX
\]
is an immersion nowhere tangent to the Euler vector field. Let $(c,t) \in S^{n-1} \times [0,1],$ and write
\[
v = \pderiv[\kappa]{s}(c,0) \in T_pL, \qquad \widetilde p = [v] \in E_p \simeq \orientedprojective(T_pL).
\]
Let
\[
\widetilde q = [(\orientedprojective((d\Psi_t)_p),\widetilde p)] \in \orientedprojective(TC_q\Lambda_t).
\]
By Lemma~\ref{lemma: time derivative of cone derivative} we have
\begin{equation}\label{equation:Jw}
\pderiv{t}\pderiv[\Phi]{s}(c,t,0) = \deriv{t} d\Psi_t(v) = -J\nabla_v\widetilde{\nabla^th}-\tan(\theta_t(\widetilde p))\nabla_v\widetilde{\nabla^th}.
\end{equation}
Write $w = \nabla_v\widetilde{\nabla^th}.$ As $p$ is a critical point of $h,$ Remark~\ref{remark: critical cone point} gives $\widetilde{\nabla^th}(\widetilde p) = 0,$ and it follows that $w \in T_{\widetilde q}\Lambda_t \subset T_qX.$ Moreover, $w \neq 0$ since $p$ is a non-degenerate critical point of $h.$ Write
\[
\Upsilon := d\left(\left.\pderiv[\Psi_t\circ\kappa]{s}\right|_{s=0}\right)\left(T_cS^{n-1}\right) = d \left(\left.\pderiv[\Phi]{s}\right|_{S^{n-1} \times \{t\}\times \{0\}} \right)\left(T_cS^{n-1}\right)
\subset T_q X.
\]
Recalling Definition~\ref{definition:polar coordinates}, let $\sigma = \left.\pderiv[\kappa]{s}\right|_{s=0}: \orientedprojective(T_pL) \to T_pL \setminus \{0\}$ denote the section associated to $\kappa.$ Let $\widehat p \subset T_pL$ and $\widehat q \subset T_q X$ denote the $1$-dimensional subspaces generated by the rays $\widetilde p, \widetilde q,$ respectively.
Then,
\[
\widetilde{TL}_S|_{\widetilde p} = T_pL = \widehat p \oplus d\sigma\left(T_{\widetilde p}\orientedprojective(T_pL)\right).
\]
So, we obtain
\[
T_{\widetilde q}\Lambda_t = \widetilde{d\Psi}_t\left(\left(\widetilde{TL}_S\right)_{\widetilde p}\right) = \widehat q \oplus \Upsilon.
\]
Since $T_{\widetilde q}\Lambda_t$ is Lagrangian, we have $T_{\widetilde q}\Lambda_t \cap JT_{\widetilde q}\Lambda_t = \{0\}.$ It follows from equation~\eqref{equation:Jw} that
\[
\pderiv{t}\pderiv[\Phi]{s}(c,t,0) \notin T_{\widetilde q}\Lambda_t.
\]
So, the linear subspace
\begin{multline*}
\R\left\langle\pderiv{t}\pderiv[\Phi]{s}(c,t,0)\right\rangle \oplus T_{\widetilde q}\Lambda_t =\\
 =\R\left\langle\pderiv{t}\pderiv[\Phi]{s}(c,t,0)\right\rangle \oplus \widehat q\oplus d \left(\left.\pderiv[\Phi]{s}\right|_{S^{n-1} \times \{t\}\times \{0\}} \right)\left(T_cS^{n-1}\right) \subset T_qX
\end{multline*}
is $n+1$ dimensional. Since the Euler vector field $\varepsilon$ on $T_qX$ satisfies
\[
0 \neq \varepsilon\left(\pderiv[\Phi]{s}(c,t,0)\right) \in \widehat q,
\]
the lemma follows.
\end{proof}

	We henceforth restrict the discussion to the setting of Theorem~\ref{theorem: geodesic-cylinder correspondence}. Namely, we assume $\Lambda_0,\Lambda_1\subset X$ are smoothly embedded positive Lagrangian spheres intersecting transversally at the points $q_0$ and $q_1$ and nowhere else.
	
	\begin{rem}
		\label{rem:simple topology}
		Note that, under the above assumption, if there exists a geodesic connecting $\Lambda_0$ and $\Lambda_1$, its critical locus necessarily consists of exactly two points. In particular, letting $h$ denote the Hamiltonian of the geodesic and $[m,M]$ the image of $h,$ every $c\in(m,M)$ is a regular value.
	\end{rem}

	We define regularity for a connected component in $\mathcal{SLC}\left(S^{n-1};\Lambda_0,\Lambda_1\right).$
	
	\begin{dfn}
		\label{definition: regularity}
		Let $\mathcal{Z}\subset\mathcal{SLC}\left(S^{n-1};\Lambda_0,\Lambda_1\right)$ be a connected component. A \emph{regular parameterization} of $\mathcal{Z}$ is a smooth map $\Phi:S^{n-1}\times[0,1]\times[0,1]\to X$ satisfying the following conditions:
		\begin{enumerate}
			\item The restricted map $\Phi|_{S^{n-1}\times[0,1]\times(0,1)}$ is an interior-regular parameterization of $Z.$
			\item The restricted maps $\Phi|_{S^{n-1}\times[0,1]\times[0,1/2)}$ and $\Phi|_{S^{n-1}\times[0,1]\times(1/2,1]}$ are regular parameterizations of the two ends of $\mathcal{Z}$ about the intersection points $q_0$ and $q_1,$ respectively.
		\end{enumerate}
		We say $\mathcal{Z}$ is \emph{regular} if it admits a regular parameterization.
	\end{dfn}
	
	\begin{rem}
		\label{remark: regular connected component}
		$\;$
		\begin{enumerate}[label=(\alph*)]
			\item \label{equivalent definition of regularity}	A connected component $\mathcal{Z}\subset\mathcal{SLC}\left(S^{n-1};\Lambda_0,\Lambda_1\right)$ is regular if and only if each $Z \in \mathcal{Z}$ has an interior-regular neighborhood and the ends of $\mathcal{Z}$ converge regularly to the intersection points~$q_0$ and $q_1$ respectively. This follows from Lemma~\ref{lemma: interior and end regular parameterization compatible with harmonics} and Corollary~\ref{corollary:action}. Indeed, we use the compactness of $[0,1]$ and the uniqueness claim to glue together parameterizations of the ends of~$\mathcal{Z}$ with finitely many interior regular parameterizations of open intervals in~$\mathcal{Z}.$
			\item Suppose $\mathcal{Z}\subset\mathcal{SLC}\left(S^{n-1};\Lambda_0,\Lambda_1\right)$ is a regular connected component. Then $\mathcal{Z}$ admits a regular parameterization compatible with the harmonics. This again follows from Lemma~\ref{lemma: interior and end regular parameterization compatible with harmonics}.
		\end{enumerate}
	\end{rem}

	\begin{proof}[Proof of Theorem~\ref{theorem: geodesic-cylinder correspondence}]
		Suppose there exists a geodesic $(\Lambda_t)_{t\in[0,1]}$ between the positive Lagrangian spheres $\Lambda_0$ and~$\Lambda_1.$ Let $\mathcal{Z}\subset\mathcal{SLC}\left(S^{n-1};\Lambda_0,\Lambda_1\right)$ denote the cylindrical transform. Then Lemma~\ref{lemma: non-singular cylindrical transform}, Lemma~\ref{lemma:rocking lemma} and Remark~\ref{remark: regular connected component}\ref{equivalent definition of regularity}, imply that $\mathcal{Z}$ is regular.
		
		Conversely, let $\mathcal{Z}\subset\mathcal{SLC}\left(S^{n-1};\Lambda_0,\Lambda_1\right)$ be a regular connected component. Let $\Phi : S^{n-1} \times [0,1]\times [0,1]$ be a regular parameterization of $\mathcal{Z}.$ It follows from Remark~\ref{remark: regular parameterization gives rise to family of cone-immersions} and Corollary~\ref{cor:cone-immersion open mapping} that the image of $\Phi|_{S^{n-1}\times \{i\} \times [0,1]} \subset \Lambda_i$ is open. Since $S^{n-1} \times [0,1]$ is compact, the image is also closed. Thus, the submanifolds swept by $\mathcal{Z}$ are $\Lambda_i\setminus\{q_0,q_1\}=:\Lambda_i^\circ,\;i=0,1.$ By Proposition~\ref{proposition: interior-regular family gives rise to geodesic}, there exists a geodesic $\left(\Lambda_t^\circ\right)_{t\in[0,1]}$ between $\Lambda_0^\circ$ and $\Lambda_1^\circ$ with cylindrical transform $\mathcal{Z}.$ For $j = 0,1,$ let $U_j$ be a ray representing the end of $\mathcal{Z}$ that converges regularly to $q_j.$ Let $\Lambda_i^{U_j}$ denote the unpunctured submanifold of $\Lambda_i$ swept by $U_j.$ By Lemma~\ref{lemma: geodesic of small Lagrangians}, there exists a geodesic $\left(\Lambda^{U_j}_t\right)_{t\in[0,1]}$ between $\Lambda_0^{U_j}$ and $\Lambda_1^{U_j}$ with cylindrical transform $U_j.$  By the uniqueness claim of Proposition~\ref{proposition: interior-regular family gives rise to geodesic}, the three above geodesics glue together to the desired geodesic between $\Lambda_0$ and $\Lambda_1.$
	\end{proof}

	\section{Proof of the perturbation theorem}
	\label{section: proof of perturbation theorem}
	
	In this section we prove Theorem~\ref{theorem: perturbation of geodesic}, which shows that geodesics of positive Lagrangian spheres deform uniquely and continuously under small perturbations of the endpoints. The relevant topologies are the subject of Definition~\ref{dfn:topology on geodesics}. By Theorem~\ref{theorem: geodesic-cylinder correspondence}, it suffices to show that regular components of the space of imaginary special Lagrangian cylinders deform uniquely and continuously under small pertubations of their positive Lagrangian boundary conditions. Proposition~\ref{proposition: interior big family of cylinders} shows that a small interval around each cylinder in a regular component deforms uniquely and continuously. Proposition~\ref{proposition: end big family of cylinders} shows that a small interval at each end of a regular component deforms continuously. By a compactness argument and the uniqueness part of Proposition~\ref{proposition: interior big family of cylinders}, a finite number of these deformed intervals are glued together to show that the entire regular component deforms uniquely and continuously.

We assume the setting of Theorem~\ref{theorem: perturbation of geodesic}. That is, $\mathcal{O}$ is a Hamiltonian isotopy class of smoothly embedded positive Lagrangian spheres in $X,$ the spheres $\Lambda_0, \Lambda_1 \in \mathcal{O}$ intersect transversally at $\Lambda_0 \cap \Lambda_1 = \{q_0,q_1\},$ and $(\Lambda_t)_{t\in[0,1]}$ is a geodesic between $\Lambda_0$ and $\Lambda_1.$ Let $\mathcal{Z}$ denote the cylindrical transform of $(\Lambda_t)_t.$ By Theorem~\ref{theorem: geodesic-cylinder correspondence} the cylindrical transform $\mathcal{Z}$ is a regular connected component of the 1-dimensional manifold $\mathcal{SLC}\left(S^{n-1};\Lambda_0,\Lambda_1\right).$ Let $\Phi : S^{n-1} \times [0,1] \times [0,1] \to X$ be a regular parameterization of $\mathcal{Z}.$ For $s \in (0,1),$ we write
	\[
	\Phi_s := \Phi|_{S^{n-1} \times [0,1] \times \{s\}},\quad Z_s := [\Phi_s] \in \mathcal{Z}.
	\]
Fix a Weinstein neighborhood $\Lambda_1 \subset W \subset X$ identified with an open neighborhood of the zero section $T^*\Lambda_1$ in such a way that the projection $\pi : W \to \Lambda_1$ induced from the projection $T^*\Lambda_1 \to \Lambda_1$ satisfies
	\[
	W \cap \Lambda_0 = \pi^{-1}(\{q_0, q_1\}).
	\]
In particular, spheres in $\mathcal{O}$ that are close to $\Lambda_1$ are identified with exact 1-forms on $\Lambda_1.$ For a function $h \in C^\infty(\Lambda_1)$ small enough in the $C^1$-sense that the graph of $dh$ is contained in $W,$ we let $\Lambda_{1, h}$ denote the corresponding element of $\mathcal{O}.$ We write $\Lambda_0 \cap \Lambda_{1, h} = \{q_{0, h}, q_{1, h}\}.$ Thus a sufficiently small $C^{k+1,\alpha}$ open set $0 \in \mathcal{W} \subset C^\infty(\Lambda_1)$ corresponds to a $C^{k,\alpha}$ open neighborhood of $\Lambda_1$ in $\mathcal{O}.$

	\begin{ntn}
		\label{notation: Hamiltonian flow with cutoff}
		Let $\chi : X \to [0,1]$ be smooth with compact support contained in $W,$ and let $h \in C^{k, \alpha}(\Lambda_1)$ for some $k \ge 2$ and $\alpha \in [0,1].$ Then the function $\chi \cdot h \circ \pi$ is well-defined and of regularity $C^{k, \alpha}$ on $X$ with compact support. We let $\varphi_{h, \chi}$ denote the time-1 Hamiltonian flow of $\chi \cdot h \circ \pi.$
	\end{ntn}

	\begin{rem}
		\label{remark: cutoff}
		Let $\chi$ and $h$ be as in Notation~\ref{notation: Hamiltonian flow with cutoff}. Then the map $\varphi_{h, \chi} : X \to X$ is a $C^{k - 1, \alpha}$ symplectomorphism of $X$ restricting to the identity on $\Lambda_0.$
	\end{rem}
In the proofs below, we will use the following well-known technical lemma. Let $\Omega^*_{C^{k,\alpha}}$ denote differential forms of regularity $C^{k,\alpha}.$
\begin{lemma}\label{lemma: Holder pullback}
Let $P,Q,$ be smooth manifolds. For $r,s,l \geq 1,$ the map
\[
C^{r,\alpha}(P,Q) \times \Omega^*_{C^{s,\alpha}}(Q) \to \Omega^*_{C^{\min\{r-1,s-l\},\alpha}}(P), \qquad (g,\eta) \mapsto g^*\eta,
\]
is of regularity $C^l.$
\end{lemma}

Abbreviate $L = S^{n-1}\times [0,1].$	

	\begin{prop}
		\label{proposition: interior big family of cylinders}
		Fix $\alpha \in (0,1).$ Let $s_0  \in (0, 1).$ Then there exist a positive $\epsilon,$ a $C^{3, \alpha}$-open $0 \in \mathcal{W} \subset C^\infty(\Lambda_1)$ and a family of smooth immersions
\[
f_{s,h} : L \to X, \qquad (s, h) \in (s_0 - \epsilon, s_0 + \epsilon) \times \mathcal{W},
\]
smooth in $s$ and continuous with respect to the $C^{3,\alpha}$ topology on $\mathcal{W}$ and the $C^{1,\alpha}$ topology on $C^{\infty}(L,X)$			with the following properties:
			\begin{enumerate}[label=(\arabic*)]
				\item \label{item: special} For $s$ and $h$ as above the immersion $f_{s,h}$ represents an imaginary special Lagrangian cylinder $Z_{s, h} \in \mathcal{SLC}(S^{n-1};\Lambda_0,\Lambda_{1,h}).$
\item \label{item:h = 0}
We have $f_{s,0} = \Phi_s$ for $s \in (s_0 - \epsilon, s_0 + \epsilon).$
				\item \label{item: interior regular} For $h \in \mathcal{W}$ the family of cylinders $(Z_{s, h})_s$ is interior regular.
			\end{enumerate}
Moreover, there exists a $C^{1,\alpha}$ open set $\mathcal V \subset C^\infty(L,X)$ with
\begin{equation}\label{equation: fshinW}
f_{s,h} \in \mathcal V, \qquad (s,h) \in (s_0-\epsilon,s_0 + \epsilon)\times \mathcal{W},
\end{equation}
such that if $f \in \mathcal V$ represents an imaginary special Lagrangian cylinder
\[
[f] \in \mathcal{SLC}(S^{n-1};\Lambda_0,\Lambda_{1,h})
\]
with $h \in \mathcal W,$ then $[f] = [f_{s,h}]$ for some $s \in (s_0-\epsilon,s_0 + \epsilon).$
	\end{prop}
	
	\begin{proof}
By Lemma~\ref{lemma: Weinstein with boundary}, choose an immersed Weinstein neighborhood $(Y,\psi)$ of $Z_{s_0}$ compatible with $\Lambda_0$ and $\Lambda_1,$ where $Y \subset T^*L$ and $\psi : Y \to X$ with $\psi|_L = \Phi_{s_0}.$ Let $\pi_L : T^*L \to L$ denote the projection. For $u\in\cob{2,\alpha}(L),$ let $\mathrm{Graph}(du) \subset T^*L$ denote the graph. For $u$ small enough that $\mathrm{Graph}(du) \subset Y,$ let $j_u : L \to X$ be given by
\[
j_u = \psi \circ \left( \pi_L|_{\mathrm{Graph}(du)}\right)^{-1}.
\]
Let $U \subset X$ be an open set such that
\[
\overline U \subset W \setminus \Lambda_0, \qquad \Phi_{s_0}({S^{n-1}\times \{1\}}) \subset U.
\]
Let $\chi : X \to [0,1]$ be a smooth function such that $\chi|_U \equiv 1$ and $\mathrm{supp}(\chi) \subset W \setminus \Lambda_0.$
Choose an open set $\Phi_{s_0}(S^{n-1}\times\{1\}) \subset U' \subset\subset U$ and an open set $0 \in \mathcal{A} \subset C^{3, \alpha}(\Lambda_1)$ such that for $h \in \mathcal{A}$ we have
		\[
		\varphi_{h, \chi}(U' \cap \Lambda_1) \subset \Lambda_{1, h}.
		\]
		Let $0 \in \mathcal{U} \subset \cob{2,\alpha}(L)$ be open such that for $u\in\mathcal{U},$ we have $j_u(S^{n-1}\times \{1\}) \subset U'.$ Consider the differential operator
		\[
		\mathcal{F} : \mathcal{U} \times \mathcal{A} \to C^\alpha\left(L\right),\quad(u,h) \mapsto *j_u^* \varphi_{h, \chi}^* \real \Omega.
		\]
For $(u, h) \in \mathcal{U} \times \mathcal{A},$ the immersion $\varphi_{h,\chi} \circ j_u$ represents a Lagrangian cylinder in $\mathcal{LC}\left(S^{n-1}; \Lambda_0, \Lambda_{1, h}\right)$, and this cylinder is imaginary special Lagrangian if and only if $\mathcal{F}(u, h) = 0.$

We claim $\mathcal{F}$ is continuously differentiable, and for a fixed $h \in \mathcal{A} \cap C^\infty(\Lambda_1),$ the map $u \mapsto \mathcal{F}(u,h)$ is smooth. Indeed, recalling Remark~\ref{remark: cutoff}, since the map
\[
\mathcal{A} \to C^{2,\alpha}(X,X), \qquad h \mapsto \varphi_{h,\chi},
\]
is continuously differentiable, it follows from Lemma~\ref{lemma: Holder pullback} that the map
\[
\mathcal{A} \to \Omega^*_{C^{1,\alpha}}(X), \qquad h \mapsto \varphi^*_{h,\chi}\real \Omega,
\]
is continuously differentiable. Similarly, since the map
\[
\mathcal{U} \to C^{1,\alpha}(L,X), \qquad u \mapsto j_u,
\]
is smooth, it follows from Lemma~\ref{lemma: Holder pullback} that the map
\[
\mathcal{U} \times \Omega^*_{C^{s,\alpha}}(X) \to \Omega^*_{C^{\alpha}}(L), \qquad (u,\eta) \mapsto j_u^* \eta,
\]
is of regularity $C^{s}.$
So, the map $\mathcal{F}$ is continuously differentiable as the composition of two continuously differentiable  maps. Moreover, for fixed $h \in \mathcal{A}\cap C^\infty(\Lambda_1),$ we have $\varphi^*_{h,\chi}\real \Omega \in \Omega^*_{C^{\infty}}(X),$ so the map $u \mapsto \mathcal{F}(u,h)$ is smooth.
		
		Consider the linearization of $\mathcal{F},$
		\[
		d \mathcal{F}_{(0, 0)} : \cob{2, \alpha}(L) \times C^{3, \alpha}(\Lambda_1) \to C^\alpha(L).
		\]
		By Lemma~\ref{lemma: linearized operator} and Lemma~\ref{lemma: Delta rho is a cool operator},  the restriction of $d \mathcal{F}_{(0, 0)}$ to the subspace
		\[
		C^{2, \alpha}(L; \partial L) \times \{0\} \subset \cob{2, \alpha}(L) \times C^{3, \alpha}(\Lambda_1)
		\]
		is an isomorphism onto $C^\alpha(L).$ Abbreviate $\mathbb{V} := C^{2, \alpha}(L; \partial L).$ Recall that
		\[
		\codim\left(\mathbb{V} \subset \cob{2, \alpha}(L)\right) = 1.
		\]
		Let $\ell \subset \cob{2, \alpha}(L)$ be a line consisting of smooth functions such that
		\[
		\cob{2, \alpha}(L) = \mathbb{V} \oplus \ell.
		\]
		By the implicit function theorem, there exist open neighborhoods
		\[
		0 \in \mathbb{V}_0 \subset \mathbb{V}, \quad 0 \in \ell_0 \subset \ell, \quad 0 \in \mathcal{A}_0 \subset \mathcal{A},
		\]
		such that for $(l, h) \in \ell_0 \times \mathcal{A}_0$ there exists a unique $v = v(l, h) \in \mathbb{V}_0$ with
		\[
		\mathcal{F}(v + l, h) = 0.
		\]
		By elliptic regularity (e.g.\ \cite[Chapter 17]{gilbarg-trudinger}), the function $v(l, h)$ is smooth if $h$ is. Write
		\[
		\mathcal{W} := C^\infty(\Lambda_1) \cap \mathcal{A}_0.
		\]
		Define a family of immersions by
\begin{equation}
		\label{big family of cylinders}
		\widetilde f_{l,h} := \varphi_{h, \chi} \circ j_{v(l, h) + l}, \qquad (l, h) \in \ell_0 \times \mathcal{W}.		
		\end{equation}
By construction, the immersions $\widetilde f_{l,h}$ represent special Lagrangian cylinders in  $\mathcal{SLC}\left(S^{n-1}; \Lambda_0, \Lambda_{1, h}\right)$ for $(l, h) \in \ell_0 \times \mathcal{W}.$ Moreover, the map
\[
\ell_0 \times \mathcal{W} \to C^{\infty}(L,X), \qquad (l,h) \mapsto \widetilde f_{l,h},
\]
is continuous with respect to the $C^{3,\alpha}$ topology on $\mathcal{W}$ and the $C^{1,\alpha}$ topology on $C^\infty(L,X).$ Also, for fixed $h \in \mathcal{W},$ the map $l \mapsto \widetilde f_{l,h}$ is smooth.

		For a small $\epsilon > 0,$ there is a unique open embedding $i : (s_0 - \epsilon, s_0 + \epsilon) \hookrightarrow \ell_0$ such that
		\[
		\left[\widetilde f_{i(s),0}\right] = Z_s, \qquad s \in (s_0 - \epsilon, s_0 + \epsilon).
		\]
For $s \in (s_0 - \epsilon, s_0 + \epsilon),$ let $\zeta_s \in \diff(L)$ be the diffeomorphism such that
\[
\widetilde f_{i(s),0}\circ \zeta_s = \Phi_s.
\]
Observe that the map
\[
(s_0 - \epsilon, s_0 + \epsilon) \to C^\infty(L,L), \qquad s \mapsto \zeta_s
\] 		
is smooth.
Take
\[
f_{s,h} := \widetilde f_{i(s),h} \circ \zeta_s.
\]
Since the maps $(l,h) \mapsto \widetilde f_{l,h}$ and $s \mapsto \zeta_s$ are continuous in the topologies specified above, it follows that the map $(s,h) \mapsto f_{s,h}$ is continuous as desired. Moreover, for fixed $h \in \mathcal{W},$ the map $s \mapsto f_{s,h}$ is smooth.

Write $Z_{s,h} := [f_{s,h}].$ Properties~\ref{item: special} and~\ref{item:h = 0} claimed in the proposition are immediate from the construction. To establish property~\ref{item: interior regular}, we claim that after possibly shrinking $\mathcal{W},$ for $h \in \mathcal{W},$ the map
\[
\Phi^h : L \times (s_0 - \epsilon, s_0 + \epsilon) \to X,
\]
given by
\[
\Phi^h(p,s) = f_{s,h}(p)
\]
is a regular parameterization of the family $(Z_{s,h})_{s \in (s_0 - \epsilon, s_0 + \epsilon)}.$ Indeed, the map $\Phi^h$ is smooth because the map $s \mapsto f_{s,h}$ is smooth. Conditions~\ref{item: rep} and~\ref{item: diffeo s} of Definition~\ref{definition: interior regularity}~\eqref{interior regularity first part} hold by construction. The remaining conditions in Definition~\ref{definition: interior regularity}~\eqref{interior regularity first part} are open, and $\Phi^0 = \Phi|_{S^{n-1}\times [0,1]\times (s_0 - \epsilon, s_0 + \epsilon)}$, which is interior regular. So, possibly after shrinking $\mathcal W$ and $\epsilon,$ the map $\Phi^h$ is also an interior regular parameterization for $h \in \mathcal W.$

Finally, we construct the open set $\mathcal{V}.$ Given $\mathcal{Q} \subset \mathcal{W}$ and $\mathcal{P}\subset C^\infty(L,X),$ let
\[
\mathcal{B}(\mathcal{Q},\mathcal{P}) := \{(h,f) \in \mathcal{Q}\times\mathcal{P}\,|\, f \text{ represents a cylinder } [f]\in \mathcal{LC}(S^{n-1};\Lambda_0,\Lambda_{1,h})\}.
\]
First, we claim that perhaps after shrinking~$\mathcal{W},$ there exists a $C^{1,\alpha}$ open set $\mathcal{V}_1 \subset C^{\infty}(L,X)$ such that
\[
\Phi_s \in \mathcal{V}_1, \qquad s \in (s_0-\epsilon,s_0 + \epsilon),
\]
and if $(h,f) \in \mathcal{B}(\mathcal{W},\mathcal{V}_1),$ then
\begin{equation*}
\varphi_{h,\chi}^{-1}\circ f(L) \subset \psi(Y).
\end{equation*}
Here, the set $\psi(Y)$ is not open in $X$ because $Y$ is a manifold with boundary. Indeed, $Y$ is an open neighborhood of the zero section of $T^*L,$ and $L$ is a manifold with boundary.
Nonetheless, for $i = 0,1,$ we have $\varphi_{h,\chi}^{-1}\circ f(S^{n-1} \times \{i\}) \subset \Lambda_i.$  Moreover, $\varphi_{h,\chi}^{-1}\circ f$ is close in the $C^{1,\alpha}$ topology to $\Phi_s$ for some $s \in (s_0-\epsilon,s_0 +\epsilon),$ and $\Phi_s(L) \subset \psi(Y).$
So, the claim follows.

For $(h,f) \in \mathcal{B}(\mathcal{W},\mathcal{V}_1),$  let
\[
\kappa_{h,f}: = \pi_L \circ \psi^{-1} \circ \varphi_{h,\chi}^{-1}\circ f : L \to L.
\]
The map
\[
\mathcal{B}(\mathcal{W},\mathcal{V}_1) \to C^\infty(L,L), \qquad (h,f) \mapsto \kappa_{h,f},
\]
is continuous with respect to the $C^{3,\alpha}$ topology on~$\mathcal{W}$ and the $C^{1,\alpha}$ topologies on~$\mathcal{V}_1$ and $C^\infty(L,L).$ Moreover, diffeomorphisms are open in $C^\infty(L,L)$ in the $C^{1,\alpha}$ topology and $\kappa_{0,\Phi_s} = \zeta_s$ is a diffeomorphism for $s \in (s_0-\epsilon,s_0+\epsilon).$
So, possibly shrinking $\mathcal{W},$ we choose an open $\mathcal{V}_2 \subset \mathcal{V}_1$ such that $\Phi_s \in \mathcal{V}_2$ for $s \in (s_0 - \epsilon, s_0 + \epsilon)$ and if $(h,f) \in \mathcal{B}(\mathcal{W},\mathcal{V}_2),$ then $\kappa_{h,f}$ is a diffeomorphism.

For $(h,f) \in \mathcal{B}(\mathcal{W},\mathcal{V}_2),$ by Lemma~\ref{lemma: local description of space of cylinders} there exists $u_{h,f} \in \cob{\infty}(L)$ such that
\[
j_{u_{h,f}} = \varphi_{h,\chi}^{-1}\circ f \circ \kappa_{h, f}^{-1}.
\]
The map
\[
\mathcal{B}(\mathcal{W},\mathcal{V}_2) \to \cob{\infty}(L), \qquad (h,f) \mapsto u_{h,f},
\]
is continuous with respect to the $C^{3,\alpha}$ topology on $\mathcal{W},$ the $C^{1,\alpha}$ topology on $\mathcal{V}_2$ and the $C^{2,\alpha}$ topology on $\cob{\infty}(L).$ Hence, possibly shrinking $\mathcal{W},$ we choose an open $\mathcal{V} \subset \mathcal{V}_2$ such that $\Phi_s \in \mathcal{V}$ for $s \in (s_0 - \epsilon, s_0 + \epsilon)$ and if $(h,f) \in \mathcal{B}(\mathcal{W},\mathcal{V}),$ then
\[
u_{h,f} \in \mathbb{V}_0 + \ell_0.
\]
For $(h,f) \in \mathcal{B}(\mathcal{W},\mathcal{V}),$ let $l_{h,f}$ be the projection of $u_{h,f}$ along $\mathbb{V}$ to $\ell.$ If $[f]$ is imaginary special Lagrangian, then $\mathcal{F}(u_{h,f},h) = 0.$ So, by uniqueness of $v_{l,h},$ it follows that $u_{h,f} = l_{h,f} + v_{l_{h,f},h}.$ Thus, $[f] = [f_{s,h}]$ for $s = i^{-1}(l_{h,f}).$ By continuity of $f_{s,h}$ and property~\ref{item:h = 0} of the proposition, after possibly shrinking $\mathcal{W}$ again, we obtain condition~\eqref{equation: fshinW}.
\end{proof}

	\begin{prop}
		\label{proposition: end big family of cylinders}
		Fix $\alpha \in (0,1)$. Then, there exist a positive $\epsilon,$ a $C^{3, \alpha}$-open $0 \in \mathcal{W} \subset C^\infty(\Lambda_1)$ and a family of smooth maps
\[
f_{s,h} : L \to X, \qquad (s, h) \in [0,\epsilon) \times \mathcal{W},
\]
smooth in $s$ and continuous with respect to the $C^{3,\alpha}$ topology on $\mathcal{W}$ and the $C^{1,\alpha}$ topology on $C^{\infty}(L,X)$			with the following properties:
\begin{enumerate}[label=(\arabic*)]
\item \label{item: special end} For $(s,h) \in (0,\epsilon) \times \mathcal{W}$ the map $f_{s,h}$ is an immersion  representing an imaginary special Lagrangian cylinder $Z_{s, h} \in \mathcal{SLC}(S^{n-1};\Lambda_0,\Lambda_{1,h}).$
\item \label{item:h = 0 end}
We have $f_{s,0} = \Phi_s$ for $s \in [0, \epsilon).$
\item \label{item: regularly converges}
For $h \in \mathcal{W},$ the map
\[
\Phi^h : L \times [0,\epsilon) \to X,
\]
given by
\[
\Phi^h(p,s) = f_{s,h}(p)
\]
is a regular parameterization about $q_{0,h}$ of the family of imaginary special Lagrangian cylinders $(Z_{s,h})_{s \in (0,\epsilon)}.$
\item \label{item: end continuity}
The map $\mathcal{W} \to C^\infty(L,TX)$ given by
\begin{equation}\label{equation:pdsfsh0}
h \mapsto \left.\pderiv[f_{s,h}]{s}\right|_{s = 0}
\end{equation}
is continuous with respect to the $C^{3,\alpha}$ topology on $\mathcal W$ and the $C^{1,\alpha}$ topology on $C^{\infty}(L,TX).$
			\end{enumerate}
A similar family of smooth immersions $f_{s,h}$ exists for $s \in (1-\epsilon,1].$
	\end{prop}
	
	\begin{proof}
We prove the proposition for $s \in (0,\epsilon).$
	Let $\chi : X \to [0,1]$ be smooth with compact support in $W$ and equal to 1 in a neighborhood of the intersection point~$q_0.$ Then there exist open sets $q_0 \in U \subset X$ and $0 \in \mathcal{A}_1 \subset C^1(\Lambda_1)$ such that for $h \in \mathcal{A}_1$ we have
	\[
	\varphi_{h, \chi}(U \cap \Lambda_1) \subset \Lambda_{1, h}, \qquad \varphi_{h, \chi}(U \cap \Lambda_0) \subset \Lambda_0.
	\]
	Let $\delta > 0$ such that for $s \in (0, \delta)$ we have $Z_s \subset U.$ Thus, for $s \in (0,\delta)$ and $h \in \mathcal{A}_1$,
	\[
	[\varphi_{h, \chi}\circ \Phi_s] \in \mathcal{LC}\left(S^{n-1}; \Lambda_0, \Lambda_{1, h}\right).
	\]

After possibly shrinking $U,$ identify $U$ with a ball $V \subset \C^n$ via a Darboux parameterization
		\[
		\mathbf{X} : V \to U
		\]
		such that $\mathbf{X}^{-1}(\Lambda_0)$ and $\mathbf{X}^{-1}(\Lambda_1)$ are contained in real linear subspaces. For $s \geq 0,$ let $M_s : \C^n \to \C^n$ denote multiplication by $s,$ and write
		\[
		V_s := M_s^{-1}(V).
		\]

		Let $\mathcal{A} := \mathcal{A}_1 \cap C^{3,\alpha}(\Lambda_1).$
		For $h\in\mathcal{A},$ define a complex structure and an $n$-form on $V$ by
		\[
		J_h = J_{h,1} := \mathbf{X}^* \varphi_{h, \chi}^* J, \quad \Omega_h = \Omega_{h,1} := \mathbf{X}^* \varphi_{h, \chi}^*\Omega.
		\]
		For $h \in \mathcal{A}$ and $s \in (0, \delta),$ define a complex structure and an $n$-form on $V_s$ by
		\[
		J_{h,s}:=M_s^*J_h,\quad\Omega_{h,s}:=s^{-n}M_s^*\Omega_h.
		\]
		The complex structures and $n$-forms defined in this manner are of regularity $C^{1,\alpha}.$
For $h\in\mathcal{A}$ we have
		\[
		J_{h,s}\underset{s\searrow0}{\longrightarrow}J_{h,0},\quad\Omega_{h,s}\underset{s\searrow0}{\longrightarrow}\Omega_{h,0},
		\]
		where $J_{h, 0}$ and $\Omega_{h, 0}$ are a constant complex structure and a constant $n$-form on $V_0 = \C^n,$ and the convergence is with respect to the $C^{1, \alpha}$ topology on compact subsets. Moreover, writing $\Omega_{C^\alpha}$ for differential forms of regularity $C^\alpha,$ for $s_0 \in (0,\delta)$ the map
\[
[0,s_0) \to \Omega_{C^{\alpha}}(V_{s_0}), \qquad s \mapsto \Omega_{h,s}|_{V_{s_0}},
\]
is continuously differentiable.
		
		Recall the regular parameterization $\Phi : S^{n-1} \times [0, 1] \times [0, 1] \to X.$ By the choice of $\delta$, we have $\Phi \left(S^{n-1} \times [0,1] \times [0,\delta) \right) \subset U.$ By Lemma~\ref{lemma: milnor lemma}, we have
		\[
		\mathbf{X}^{-1}\circ\Phi(p,t,s) = s \cdot \Psi(p,t,s), \quad(p,t,s) \in S^{n-1} \times[0,1] \times [0,\delta),
		\]
		where $\Psi : S^{n-1} \times [0,1] \times [0,\delta)\to \C^n$ is smooth with
		\[
		\Psi(p,t,0) = \pderiv[(\mathbf{X}^{-1}\circ\Phi)]{s}(p,t,0), \quad (p,t) \in S^{n-1} \times [0,1].
		\]
For $s \in [0,\delta),$ write
\[
\Psi_s := \Psi|_{S^{n-1}\times[0,1]\times\{s\}}.
\]
For $s \in (0,\delta),$ the map $\Psi_s$ is an immersion representing an $\Omega_{0,s}$-imaginary special Lagrangian cylinder. As $\Phi$ is regular, it follows from Definition~\ref{definition: regularity} and Definition~\ref{definition: interior regularity} \eqref{regular convergence to intersection point} that the map
$
\Psi_0
$
is an immersion nowhere tangent to the Euler vector field. Thus, by continuity
\[
Z_{0,0}':=\left[\Psi_0\right]
\]
is an $\Omega_{0,0}$-imaginary special Lagrangian cylinder.
		
By Lemma~\ref{lemma: Weinstein with boundary}, choose an immersed Weinstein neighborhood $(Y,\psi)$ of $Z_{0,0}'$ compatible with $\mathbf{X}^{-1}(\Lambda_0)$ and $\mathbf{X}^{-1}(\Lambda_1),$ where $Y \subset T^*L$ and $\psi : Y \to \C^n$ with $\psi|_L = \Psi_{0}.$ Let $\pi_L : T^*L \to L$ denote the projection. For $u\in\cob{2,\alpha}(L),$ let $\mathrm{Graph}(du) \subset T^*L$ denote the graph.
Let $0\in\mathcal{U}\subset \cob{2,\alpha}(L)$ be open such that for $u\in\mathcal{U}$ we have $\mathrm{Graph}(du) \subset Y.$ For $u \in \mathcal{U},$ let $j_u : L \to X$ be given by
\[
j_u = \psi \circ \left( \pi_L|_{\mathrm{Graph}(du)}\right)^{-1}.
\]
If necessary, diminish $\delta$ so that $\psi(Y) \subset V_s$ for $s \in [0,\delta).$
Define a differential operator
		\[
		\mathcal{F} : \mathcal{U} \times \mathcal{A} \times[0,\delta) \to C^\alpha(L), \quad (u, h, s) \mapsto *j_u^* \real \Omega_{h,s}.
		\]
For $(u, h, s) \in \mathcal{U} \times \mathcal{A} \times (0,\delta),$ the immersion $\varphi_{h,\chi}\circ \mathbf{X} \circ M_s \circ j_u$ represents a Lagrangian cylinder in $\mathcal{LC}\left(S^{n-1}; \Lambda_0, \Lambda_{1, h}\right)$, and this cylinder is imaginary special Lagrangian if and only if $\mathcal{F}(u, h,s) = 0.$

We claim $\mathcal{F}$ is continuously differentiable, and for a fixed $h \in \mathcal{A} \cap C^\infty(\Lambda_1),$ the map $(u,s) \mapsto \mathcal{F}(u,h,s)$ is smooth. Indeed, recalling Remark~\ref{remark: cutoff}, since the map
\begin{equation}\label{equation: calA}
\mathcal{A} \to C^{2,\alpha}(V,X), \qquad h \mapsto \varphi_{h,\chi}\circ\mathbf{X},
\end{equation}
is continuously differentiable, it follows from Lemma~\ref{lemma: Holder pullback} that the map
\begin{equation}\label{equation: vhc*o}
\mathcal{A} \to \Omega^n_{C^{1,\alpha}}(V), \qquad h \mapsto \Omega_h = \mathbf{X}^*\varphi^*_{h,\chi}\real \Omega,
\end{equation}
is continuously differentiable. Similarly, since the map
\[
\mathcal{U}\times [0,\delta) \to C^{1,\alpha}(L,V), \qquad (u,s) \mapsto M_s \circ j_u,
\]
is smooth, it follows from Lemma~\ref{lemma: Holder pullback} that the map
\[
\Upsilon : \mathcal{U} \times [0,\delta)\times\Omega^*_{C^{l,\alpha}}(V) \to \Omega^*_{C^{\alpha}}(L), \qquad (u,s,\eta) \mapsto j_u^*M_s^* \eta,
\]
is of regularity $C^l.$ We claim that the map
\[
\widetilde{\Upsilon}: \mathcal{U} \times [0,\delta)\times\Omega^n_{C^{l,\alpha}}(V) \to \Omega^n_{C^{\alpha}}(L), \qquad (u,s,\eta) \mapsto s^{-n} j_u^*M_s^* \eta,
\]
is also of regularity $C^{l}.$ Indeed, letting $I$ run over multi-indices of length $n,$ write
\[
\eta = \sum_I \eta_I dx^I.
\]
Then,
\[
s^{-n} M_s^* \eta = \sum_I (\eta_I\circ M_s) dx^I.
\]
So,
\[
s^{-n}j_u^* M_s^* \eta = \sum_I (\eta_I \circ M_s \circ j_u) j_u^* dx^I =  \sum_I \Upsilon(u,s,\eta_I) \Upsilon(u,1,dx^I).
\]
It follows that $\widetilde \Upsilon$ is $C^l$ regular as the sum of products of $C^l$ regular maps.
So, the map $\mathcal{F}$ is continuously differentiable as the composition of the two continuously differentiable maps~\eqref{equation: vhc*o} and $\widetilde \Upsilon.$  Moreover, for fixed $h \in \mathcal{A}\cap C^\infty(\Lambda_1),$ we have $\Omega_h = \mathbf{X}^*\varphi^*_{h,\chi}\real \Omega \in \Omega^*_{C^{\infty}}(V),$ so the map $(u,s) \mapsto \mathcal{F}(u,h,s)$ is smooth.

		Consider the linearization of $\mathcal{F},$
		\[
		d \mathcal{F}_{(0, 0,0)} : \cob{2, \alpha}(L) \times C^{3, \alpha}(\Lambda_1)\times \R \to C^\alpha(L).
		\]
		By Lemma~\ref{lemma: linearized operator} and Lemma~\ref{lemma: Delta rho is a cool operator},  the restriction of $d \mathcal{F}_{(0, 0,0)}$ to the subspace
		\[
		C^{2, \alpha}(L; \partial L) \times \{0\}\times\{0\} \subset \cob{2, \alpha}(L) \times C^{3, \alpha}(\Lambda_1) \times \R
		\]
		is an isomorphism onto $C^\alpha(L).$ Abbreviate $\mathbb{V} := C^{2, \alpha}(L; \partial L).$
		Let $\ell \subset \cob{2, \alpha}(L)$ be a one dimensional subspace consisting of smooth functions such that
		\[
		\cob{2, \alpha}(L) = \mathbb{V} \oplus \ell.
		\]
		By the implicit function theorem, there exist open neighborhoods
		\[
		0 \in \mathbb{V}_0 \subset \mathbb{V}, \quad 0 \in \ell_0 \subset \ell, \quad 0 \in \mathcal{A}_0 \subset \mathcal{A},
		\]
and $\epsilon \leq \delta$ such that for $(l, h,s) \in \ell_0 \times \mathcal{A}_0 \times [0,\epsilon)$ there exists a unique
\[
v = v(l, h,s) \in \mathbb{V}_0
\]
with
		\[
		\mathcal{F}(v + l, h,s) = 0.
		\]
		By elliptic regularity (e.g.\ \cite[Chapter 17]{gilbarg-trudinger}), the function $v(l, h,s)$ is smooth if $h$ is. Write
		\[
		\mathcal{W} := C^\infty(\Lambda_1) \cap \mathcal{A}_0.
		\]
Since $\mathcal{F}$ is continuously differentiable, it follows that
\begin{equation}\label{equation: v}
v : \ell_0 \times \mathcal{A}_0 \times [0,\epsilon) \to \mathbb{V}_0
\end{equation}
is continuously differentiable. Moreover, since for fixed $h \in C^{\infty}(\Lambda_1) \cap \mathcal{A}$ the map $(u,s) \mapsto \mathcal{F}(u,h,s)$ is smooth, it follows that for fixed $h \in \mathcal{W}$ the map
\[
\ell_0 \times [0,\epsilon) \to \mathbb{V}_0, \qquad (l,s) \mapsto v(l,h,s),
\]
is smooth.

Since $\Psi_0 = j_0$ and $(\Psi_s)_{s \in [0,\delta)}$ is a smooth family of immersions, after possibly shrinking $\epsilon,$ for each $s \in [0,\epsilon)$ there exists a unique $u_s \in \mathbb{V}_0 + \ell_0$ such that $j_{u_s}$ and $\Psi_s$ represent the same immersed $\Omega_{0,s}$-imaginary special Lagrangian cylinder. In particular, $\mathcal{F}(u_s,0,s) = 0.$
Decompose
\[
u_s = v_s + l_s, \qquad v_s \in \mathbb{V}_0, \quad l_s \in \ell_0.
\]
Since $v = v(l_s,0,s)$ is the unique solution to $\mathcal{F}(v + l_s ,0,s) = 0,$ we conclude that
\begin{equation}\label{equation:v_s}
v(l_s,0,s) = v_s.
\end{equation}

Define a family of smooth maps
		\begin{equation}
		\label{another big family of cylinders}
		\widetilde f_{s, h} := \varphi_{h, \chi}\circ \mathbf{X}\circ M_s\circ j_{l_s + v(l_s,h, s)}, \qquad (s, h) \in [0, \epsilon) \times \mathcal{W}.
		\end{equation}
For $(s, h) \in (0,\epsilon) \times \mathcal{W},$ the maps $\widetilde f_{s,h}$ are immersions representing imaginary special Lagrangian cylinders in  $\mathcal{SLC}\left(S^{n-1}; \Lambda_0, \Lambda_{1, h}\right).$ Moreover, the map
\[
[0,\epsilon) \times \mathcal{W} \to C^{\infty}(L,X), \qquad (s,h) \mapsto \widetilde f_{s,h},
\]
is continuous with respect to the $C^{3,\alpha}$ topology on $\mathcal{W}$ and the $C^{1,\alpha}$ topology on $C^\infty(L,X).$ Also, for fixed $h \in \mathcal{W},$ the map $s \mapsto \widetilde f_{s,h}$ is smooth.

Recall that for $s \in [0,\epsilon),$ the immersion $j_{u_s}$ represents the same immersed cylinder as the immersion $\Psi_s.$
Let $\zeta_s \in \diff(L)$ be the diffeomorphism such that
\[
j_{u_s}\circ \zeta_s = \Psi_s.
\]
Observe that the map
\[
[0, \epsilon) \to C^\infty(L,L), \qquad s \mapsto \zeta_s,
\] 		
is smooth. Moreover, since $\Psi_0 = j_0,$ we have
\begin{equation}\label{equation:zeta0}
\zeta_0 = \id_L.
\end{equation}
By equation~\eqref{equation:v_s} we have $u_s = l_s + v(l_s,0,s),$ so it follows from equation~\eqref{another big family of cylinders} that
\[
\widetilde f_{s,h}\circ \zeta_s = \Phi_s.
\]	
Take
\[
f_{s,h} := \widetilde f_{s,h} \circ \zeta_s.
\]
Since the maps $(s,h) \mapsto \widetilde f_{s,h}$ and $s \mapsto \zeta_s$ are continuous in the topologies specified above, it follows that the map $(s,h) \mapsto f_{s,h}$ is continuous as desired. Moreover, for fixed $h \in \mathcal{W},$ the map $s \mapsto f_{s,h}$ is smooth.

Write $Z_{s,h} : = [f_{s,h}]$ for $(s,h) \in (0,\epsilon)\times \mathcal{W}.$ Properties~\ref{item: special end} and~\ref{item:h = 0 end} claimed in the proposition are immediate from the construction. We proceed with the proof of property~\ref{item: end continuity}. Indeed, by equation~\eqref{equation:zeta0}, we have
\[
\left.\pderiv[f_{s,h}]{s}\right|_{s=0} = d(\varphi_{h,\chi}\circ \mathbf{X})\circ j_{l_0 + v(l_0,h,0)}.
\]
So, the continuity of the map~\eqref{equation:pdsfsh0} follows from the continuity of the maps~\eqref{equation: calA} and~\eqref{equation: v}.

To establish property~\ref{item: regularly converges} claimed in the proposition, we argue as follows. The map $\Phi^h$ is smooth because the map $s \mapsto f_{s,h}$ is smooth. Condition~\ref{critical point} of Definition~\ref{definition: interior regularity}~\eqref{regular convergence to intersection point} is a consequence of the fact that $\varphi_{h,\chi} \circ \mathbf{X} \circ M_0$ is the constant map with image $q_{0,h}.$
Conditions~\ref{item:derivative Phi immersion} and~\ref{item:nowhere tangent Euler} of Definition~\ref{definition: interior regularity}~\eqref{regular convergence to intersection point} hold after possibly shrinking $\mathcal W$ by the following argument. Observe that
\[
\left.\pderiv[\Phi^0]{s}\right|_{s=0} = \left.\pderiv[\Phi]{s}\right|_{s=0},
\]
which satisfies Conditions~\ref{item:derivative Phi immersion} and~\ref{item:nowhere tangent Euler} of Definition~\ref{definition: interior regularity}~\eqref{regular convergence to intersection point} by assumption.
Since immersions, embeddings and transverse maps, are open in the $C^1$ topology, it suffices to show that the map
\begin{equation}\label{equation: calW}
\mathcal{W} \to C^{1,\alpha}(L,TX), \qquad h \mapsto \left.\pderiv[\Phi^h]{s}\right|_{s=0},
\end{equation}
is continuous. Since $\Phi(p,s) = f_{s,h}(p),$ this is equivalent to property~\ref{item: end continuity} of the proposition.
Condition~\ref{item:interior regular} of Definition~\ref{definition: interior regularity}~\eqref{regular convergence to intersection point} requires that $\Phi^h|_{L \times (0,\epsilon)}$ be an interior regular parameterization, which we prove as follows.  Conditions~\ref{item: rep} and~\ref{item: diffeo s} of Definition~\ref{definition: interior regularity}~\eqref{interior regularity first part} hold by construction. It remains to show that $\Phi^h|_{L \times (0,\epsilon)}$ is an immersion and $\Phi^h|_{\partial L \times (0,\epsilon)}$ is an embedding. Possibly after shrinking $\epsilon,$ this follows from Corollary~\ref{rem: easy regularity} for fixed $h.$ By the construction of $f_{s,h}$ and Lemma~\ref{lemma: immersion nowhere tangent to Euler}, we can choose $\epsilon$ uniformly in $h.$
	\end{proof}

\begin{dfn}\label{dfn:topology on geodesics}
Let $\mathcal{O}$ be a Hamiltonian isotopy class of positive Lagrangian spheres. For $\Lambda_0,\Lambda_1 \in \mathcal{O},$ we write $\Lambda_0 \pitchfork_2 \Lambda_1$ if $\Lambda_0$ and $\Lambda_1$ intersect transversally at exactly two points.  Let
\[
\mathfrak{Z}_\mathcal{O} := \left\{(\Lambda_0,\Lambda_1,\mathcal{Z})\left |
\begin{matrix}
\Lambda_i \in \mathcal{O},\; i = 0,1,\quad \Lambda_0 \pitchfork_2 \Lambda_1, \; \\
\mathcal{Z} \subset \mathcal{SLC}(\Lambda_0,\Lambda_1) \text{ a regular component}
\end{matrix}
\right.\right\}.
\]
We define the strong and weak $C^{k,\alpha}$ topologies on $\mathfrak{Z}_\mathcal{O}$ as follows.
For
\[
\mathcal{V} \subset C^\infty(S^{n-1}\times [0,1],X), \qquad \mathcal{U} \subset C^\infty(S^{n-1}\times [0,1],TX),
\]
open subsets in the $C^{k,\alpha}$ topology, write
\[
\mathcal{T}_{\mathcal{U},\mathcal{V}} := \left \{(\Lambda_0,\Lambda_1,\mathcal{Z}) \in \mathfrak{Z}_\mathcal{O} \left|
\begin{matrix}
\forall Z \in \mathcal{Z}, \; \exists f : S^{n-1}\times [0,1] \to X \text{ representing } Z \\
\text{such that } f \in \mathcal{V},\\
\forall E \text{ an end of }\mathcal{Z},\; \exists \Phi : [0,\epsilon) \to X \text{ a regular}\\
\text{parameterization of $E$ such that }\left.\pderiv[\Phi]{s}\right|_{s = 0} \in \mathcal{U}
\end{matrix}
\right.\right\}
\]
and
\[
\mathcal{X}_{\mathcal{V}} = \left \{(\Lambda_0,\Lambda_1,\mathcal{Z}) \in \mathfrak{Z}_\mathcal{O} \left|
\begin{matrix}
\exists Z \in \mathcal{Z}, \; \exists f : S^{n-1}\times [0,1] \to X \text{ representing } Z \\
\text{such that } f \in \mathcal{V}
\end{matrix}
\right.\right\}.
\]
Then, a basis for the strong $C^{k,\alpha}$ topology on $\mathfrak{Z}_\mathcal{O}$ is given by sets of the form $\mathcal{T}_{\mathcal{U},\mathcal{V}}$ and a sub-basis for the weak $C^{k,\alpha}$ topology on $\mathfrak{Z}_\mathcal{O}$ is given by sets of the form $\mathcal{X}_\mathcal{V}.$  Let
\[
\mathfrak{G}_\mathcal{O} : = \{(\Lambda_t)_{t \in [0,1]} | (\Lambda_t)_{t \in [0,1]} \text{ is a geodesic with } \Lambda_0,\Lambda_1 \in \mathcal{O}, \quad \Lambda_0 \pitchfork_2 \Lambda_1\}
\]
denote the space of geodesics with endpoints in $\mathcal{O}$ intersecting transversally at two points.
By Theorem~\ref{theorem: geodesic-cylinder correspondence}, the cylindrical transform gives a bijection
\[
\mathfrak{G}_\mathcal{O} \simeq \mathfrak{Z}_\mathcal{O}.
\]
So, the strong and weak $C^{k,\alpha}$ topologies on $\mathfrak{Z}_\mathcal{O}$ give rise to topologies on $\mathfrak{G}_\mathcal{O},$ which we also call the strong and weak $C^{k,\alpha}$ topologies respectively.
\end{dfn}

	\begin{proof}[Proof of Theorem~\ref{theorem: perturbation of geodesic}]
By Propositions~\ref{proposition: interior big family of cylinders} and~\ref{proposition: end big family of cylinders} and the compactness of $[0,1],$ we find a finite cover of $[0,1]$ by relatively open intervals $I_j,\, j = 0,\ldots,N,$ subsets $0 \in \mathcal{W}^j \subset C^\infty(\Lambda_1)$ open in the $C^{3,\alpha}$ topology, and families of smooth immersions
\[
f_{s,h}^j : L \to X, \qquad (s,h) \in I_j \times \mathcal{W}^j,
\]
continuous with respect to the $C^{3,\alpha}$ topology on $\mathcal{W}^j$ and the $C^{1,\alpha}$ topology on $C^\infty(L,X)$ that satisfy properties~\ref{item: special}-\ref{item: interior regular} of Proposition~\ref{proposition: interior big family of cylinders} if $0,1 \notin I_j$ and properties~\ref{item: special end}-\ref{item: end continuity} of Proposition~\ref{proposition: end big family of cylinders} otherwise. Moreover, if $0,1 \notin I_j,$ we have a $C^{1,\alpha}$ open set $\mathcal{V}^j \in C^\infty(L,X)$ with
\[
f_{s,h}^j \in \mathcal V^j, \qquad (s,h) \in I_j \times \mathcal{W}^j,
\]
such that for $h \in \mathcal{W}^j,$
\begin{equation}\label{equation: implication}
f \in \mathcal V^j, \quad
[f] \in \mathcal{SLC}(S^{n-1};\Lambda_0,\Lambda_{1,h}) \quad \Rightarrow \quad \exists s \in I_j, \quad [f] = [f_{s,h}^j].
\end{equation}
After possibly shrinking and relabeling the intervals $I_j,$ we can assume that $0 \in I_0,\, 1 \in I_N,$ and $I_j \cap I_k = \emptyset$ unless $k = j \pm 1.$ Moreover, we can assume that $N \geq 2.$

Let $\mathcal{W} = \cap_{j = 0}^N \mathcal{W}^j.$ For $h \in \mathcal{W},$ let
\[
U^h_j \subset \mathcal{SLC}(S^{n-1};\Lambda_0,\Lambda_{1,h})
\]
be the interval consisting of the cylinders $Z_{s,h}^j = [f_{s,h}^j]$ for $s \in I_j.$
Choose
\[
s_j \in I_j \cap I_{j+1}, \qquad j = 0,\ldots, N-1.
\]
Possibly shrinking~$\mathcal{W},$ we may assume by continuity that
\[
f^j_{s_j,h} \in \mathcal{V}^{j+1}, \qquad h \in \mathcal{W}, \quad j = 0,\ldots, N-2,
\]
and
\[
f^N_{s_{N-1},h} \in \mathcal{V}^{N-1}, \qquad h \in \mathcal{W}.
\]
It follows from implication~\eqref{equation: implication} that for $h \in \mathcal{W},$ we have
\[
U^h_j \cap U^h_{j+1} \neq \emptyset, \qquad j = 0,\ldots,N-1.
\]
Thus, the sets $\{U^h_j\}_{j=0}^N$ cover an open interval $\mathcal{Z}^h \subset \mathcal{SLC}(S^{n-1};\Lambda_0,\Lambda_{1,h}).$ For $j = 1,\ldots,N-1,$ the interval $U^h_j$ is interior regular by property~\ref{item: interior regular} of Proposition~\ref{proposition: interior big family of cylinders}. For $j = 0$ (resp. $N$) the interval $U^h_j$ converges regularly to $q_{0,h}$ (resp.~$q_{1,h}$) by property~\ref{item: regularly converges} of Proposition~\ref{proposition: end big family of cylinders}. So, the interval $\mathcal{Z}^h$ is a regular connected component by Remark~\ref{remark: regular connected component}\ref{equivalent definition of regularity}.  Take $\mathcal Y$ the $C^{2,\alpha}$ open neighborhood of $\Lambda_1$ in $\mathcal{O}$ corresponding to~$\mathcal{W}$ and take $\mathcal{X} := \mathcal{X}_{\mathcal{V}^1}.$ Let $(\Lambda_t^h)_{t \in [0,1]}$ be the geodesic corresponding to $\mathcal{Z}^h$ by Theorem~\ref{theorem: geodesic-cylinder correspondence}. By construction, $(\Lambda_t^h)_{t \in [0,1]} \in \mathcal{X}.$ For $h \in \mathcal{W}$ suppose $(\Lambda_t')_{t \in [0,1]}$ is a geodesic in $\mathcal{X}$ with $\Lambda_0' = \Lambda_0$ and $\Lambda_1' = \Lambda_{1,h}.$  Let $\mathcal{Z}' \subset \mathcal{SLC}(S^{n-1};\Lambda_0,\Lambda_{1,h})$ denote its cylindrical transform. It follows from implication~\eqref{equation: implication} that $\mathcal{Z}' \cap \mathcal{Z}^h \neq \emptyset$ and thus $\mathcal{Z}' = \mathcal{Z}^h.$ So, Theorem~\ref{theorem: geodesic-cylinder correspondence} gives $(\Lambda_t')_{t} = (\Lambda_t^h)_t.$ We have proven the existence and uniqueness part of Theorem~\ref{theorem: perturbation of geodesic}. The continuity claim follows from the continuity of the families $f_{s,h}^j$ and property~\ref{item: end continuity} of Proposition~\ref{proposition: end big family of cylinders}.
	\end{proof}

	\bibliographystyle{../../amsabbrvcnobysame}
	\bibliography{../../bibli}

\end{document}